\documentclass{amsart}
\usepackage{amsmath}
\usepackage{amssymb}
\usepackage[all]{xy}

\let\oldeqn\equation
\let\endoldeqn\endequation
\renewenvironment{equation}{\oldeqn\setlength{\thickmuskip}{10mu plus 5mu}}{\endoldeqn\relax}
\catcode`\*11
\let\oldeqn*\equation*
\let\endoldeqn*\endequation*
\renewenvironment{equation*}{\oldeqn*\setlength{\thickmuskip}{10mu plus 5mu}}{\endoldeqn*\relax}
\catcode`\*12

\newcommand{\C}{\mathbb{C}}
\newcommand{\bM}{\mathbf{M}}
\newcommand{\fM}{\mathfrak{M}}
\newcommand{\fL}{\mathfrak{L}}
\newcommand{\Z}{\mathbb{Z}}
\newcommand{\N}{\mathbb{N}}

\newcommand{\fsl}{\mathfrak{sl}}
\newcommand{\fso}{\mathfrak{so}}
\newcommand{\fsp}{\mathfrak{sp}}

\newcommand{\simto}{\mathrel{\overset{\sim}{\to}}}
\DeclareMathOperator{\Hom}{Hom}
\DeclareMathOperator{\End}{End}

\DeclareMathOperator{\GL}{GL}

\DeclareMathOperator{\img}{im}

\DeclareMathOperator{\Proj}{Proj}
\DeclareMathOperator{\Spec}{Spec}
\DeclareMathOperator{\Res}{Res}
\DeclareMathOperator{\Ind}{Ind}
\newcommand{\A}{\mathrm{A}}
\newcommand{\D}{\mathrm{D}}

\newcommand{\bv}{\mathbf{v}}
\newcommand{\bw}{\mathbf{w}}
\newcommand{\bn}{\mathbf{n}}
\newcommand{\bt}{\mathbf{t}}
\newcommand{\tI}{\widetilde{I}}

\newcommand{\tH}{\widetilde{H}}
\renewcommand{\th}{\widetilde{h}}
\newcommand{\ta}{\widetilde{a}}
\newcommand{\tA}{\widetilde{A}}
\newcommand{\tB}{\widetilde{B}}
\newcommand{\tV}{\widetilde{V}}
\newcommand{\tbv}{\widetilde{\bv}}

\newcommand{\tbw}{\widetilde{\bw}}
\newcommand{\tW}{\widetilde{W}}
\newcommand{\tOmega}{\widetilde{\Omega}}
\newcommand{\tGamma}{\widetilde{\Gamma}}
\newcommand{\tDelta}{\widetilde{\Delta}}
\newcommand{\hI}{\widehat{I}}
\newcommand{\hH}{\widehat{H}}
\newcommand{\hOmega}{\widehat{\Omega}}
\newcommand{\hh}{\widehat{h}}
\newcommand{\cB}{\mathcal{B}}

\newcommand{\cO}{{\mathcal{O}}}

\newcommand{\cV}{\mathcal{V}}
\newcommand{\cW}{\mathcal{W}}
\newcommand{\cD}{\mathcal{D}}
\newcommand{\cS}{\mathcal{S}}
\newcommand{\cF}{\mathcal{F}}
\newcommand{\tcS}{\widetilde{\mathcal{S}}}
\newcommand{\talpha}{\mathbb{T}}
\newcommand{\sbeta}{\mathbb{S}}

\numberwithin{equation}{section}
\newtheorem{thm}{Theorem}[section]
\newtheorem{lem}[thm]{Lemma}
\newtheorem{prop}[thm]{Proposition}
\newtheorem{cor}[thm]{Corollary}
\theoremstyle{definition}

\theoremstyle{remark}
\newtheorem{rmk}[thm]{Remark}
\newtheorem{ex}[thm]{Example}

\title[Diagram automorphisms of quiver varieties]{Diagram automorphisms of quiver varieties}

\author{Anthony Henderson}
\address{School of Mathematics and Statistics\\
  University of Sydney, NSW 2006\\
  Australia}
\email{anthony.henderson@sydney.edu.au}

\author{Anthony Licata}
\address{Mathematical Sciences Institute\\
   Australian National University, ACT 0200\\
   Australia}
\email{amlicata@gmail.com}


\thanks{A.H. was supported by ARC Future Fellowship No.~FT110100504. A.L. was supported by ARC Discovery Early Career Researcher Award No.~DE120102369.}

\begin{document}

\begin{abstract}
We show that the fixed-point subvariety of a Nakajima quiver variety under a diagram automorphism is a disconnected union of quiver varieties for the `split-quotient quiver' introduced by Reiten and Riedtmann. As a special case, quiver varieties of type D arise as the connected components of fixed-point subvarieties of diagram involutions of quiver varieties of type A. In the case where the quiver varieties of type A correspond to small self-dual representations, we show that the diagram involutions coincide with classical involutions of two-row Slodowy varieties. It follows that certain quiver varieties of type D are isomorphic to Slodowy varieties for orthogonal or symplectic Lie algebras.
\end{abstract}

\maketitle


\section{Introduction}
\label{sect:intro}


\subsection{Background and summary}

In \cite{nak1,nak2}, Nakajima introduced his quiver varieties $\fM(\bv,\bw)$ and used them to give a geometric construction of the integrable highest-weight representations of Kac--Moody algebras with symmetric Cartan matrices. They have since found a wide range of other applications in geometric representation theory. In particular, following the philosophy of Lusztig in~\cite{lusztig}, one can use diagram automorphisms of quiver varieties to handle the case of symmetrizable (rather than symmetric) Cartan matrices; see the articles of Xu~\cite{xu} and Savage~\cite{savage}. In this paper we address the following questions about diagram automorphisms of quiver varieties:
\begin{enumerate}
\item What is the fixed-point subvariety $\fM(\bv,\bw)^\theta$ of a quiver variety $\fM(\bv,\bw)$ under a diagram automorphism $\theta$? (Note that this question does not arise in~\cite{xu} or~\cite{savage}, which consider instead the $\theta$-stable irreducible components of the projective subvariety $\fL(\bv,\bw)$.)
\item As conjectured by Nakajima~\cite{nak1} and proved by Maffei~\cite{maffei}, each quiver variety of type A is isomorphic to a Slodowy variety of type A, i.e.\ the resolution of a Slodowy slice in some nilpotent orbit closure for $\fsl_D$. What involutions of Slodowy varieties correspond to the diagram involutions of quiver varieties of type A?
\item What information is obtained by combining the answers to (1) and (2)?
\end{enumerate} 
We answer question (1) in complete generality (indeed, with a slightly more general definition of diagram automorphism than has hitherto appeared), and questions (2) and (3) in special cases. When $\fM(\bv,\bw)$ is of type A, the answer to question (1) is that $\fM(\bv,\bw)^\theta$ is a disconnected union of quiver varieties of type D. As a result, what comes out of (3) are some analogues of Maffei's isomorphisms, this time between `small' type-D quiver varieties and `two-row' Slodowy varieties of types C and D.

\subsection{Results and outline}

The answer to question (1) involves the definition, first given by Reiten and Riedtmann \cite{RR} and recalled in Section~\ref{sect:split-quotient}, of the \emph{split-quotient quiver} of a quiver with an admissible automorphism. In the case of a quiver whose underlying graph is a simply-laced finite or affine Dynkin diagram, the split-quotient quiver is obtained by folding the original quiver, dualizing, and then unfolding; see Example~\ref{ex:dynkin}. This case of the definition is implicit in Slodowy's enhancement of the McKay correspondence~\cite{slod}, in which automorphisms of affine Dynkin diagrams arise from pairs of subgroups of $SL_2(\C)$; we explain this connection further in \S\ref{subsect:McKay}. 

In Section~\ref{sect:diag-aut} we answer question (1): the fixed-point subvariety of a quiver variety under a diagram automorphism is a disconnected union of quiver varieties for the split-quotient quiver. The general statement, applying to admissible automorphisms of arbitrary quivers, is Theorem~\ref{thm:fixed-points}. We state here for illustration only the special case that is needed for the latter parts of the paper. 

Consider a quiver variety $\fM^{\A_{2n-1}}(\bv,\bw)$ of type A$_{2n-1}$, for a positive integer $n$, where the data $\bv=(v_1,\cdots,v_{2n-1})$, $\bw=(w_1,\cdots,w_{2n-1})$ are symmetric under the Dynkin diagram involution, in the sense that $v_i=v_{2n-i}$ and $w_i=w_{2n-i}$ for all $i$. Note that the condition $w_i=w_{2n-i}$ means that the quiver variety corresponds, in Nakajima's theory, to a weight space of a self-dual representation of $\fsl_{2n}$. The definition of $\fM^{\A_{2n-1}}(\bv,\bw)$ is recalled in \S\ref{subsect:quiver-vars}; see especially Example~\ref{ex:typea}.

We can define a diagram involution $\theta$ of $\fM^{\A_{2n-1}}(\bv,\bw)$ in a natural way, incorporating an involution of the `middle' vector space $W_n$; see \S\ref{subsect:diag-aut} and especially Example~\ref{ex:typea-inv}. Suppose that the involution of $W_n$ has signature $(w_+,w_-)$ where $w_++w_-=w_n$. Let $\fM^{\A_{2n-1}}(\bv,\bw)^\theta$ denote the fixed-point subvariety of $\theta$.
\begin{thm} \label{thm:fixed-points-AtoD}
There is an isomorphism of varieties
\[
\fM^{\A_{2n-1}}(\bv,\bw)^\theta
\cong
\!\!\!\coprod_{\substack{(v_+,v_-)\\v_++v_-=v_n}}\!\!\!
\fM^{\D_{n+1}}((v_1,\cdots,v_{n-1},v_+,v_-),(w_1,\cdots,w_{n-1},w_+,w_-)).
\]
\end{thm}
\noindent
Here, on the right-hand side we have the disconnected union of the quiver varieties of type D$_{n+1}$ associated to the indicated data, where we label the nodes of the Dynkin diagram as $1,\cdots,n-1$ (branch node), $+$ and $-$ (end nodes adjacent to the branch).  Note that, in general, many of these quiver varieties of type D$_{n+1}$ will be empty. Theorem~\ref{thm:fixed-points-AtoD} says that the nonempty ones constitute the connected components of the fixed-point subvariety $\fM^{\A_{2n-1}}(\bv,\bw)^\theta$. In Proposition~\ref{prop:lagrangian} we show that the isomorphism of Theorem~\ref{thm:fixed-points-AtoD} respects the projective subvarieties $\fL^{\A_{2n-1}}(\bv,\bw)^\theta$ and $\fL^{\D_{n+1}}((v_1,\cdots,v_{n-1},v_+,v_-),(w_1,\cdots,w_{n-1},w_+,w_-))$.

Quiver varieties of affine type are particularly interesting, as they can be interpreted as moduli spaces of torsion-free sheaves on surfaces, in the context of the McKay correspondence~\cite{nak1,vv}. In \S\ref{subsect:nakajima} we outline a different proof of Theorem~\ref{thm:fixed-points}, suggested to us by Nakajima, which makes use of this interpretation and thus applies only to certain automorphisms of simply-laced finite or affine Dynkin diagrams. In this approach, the ultimate `reason' for Theorem~\ref{thm:fixed-points-AtoD} is simply that the cyclic group corresponding to $\widetilde{\A}_{2n-1}$ is an index-$2$ subgroup of the binary dihedral group corresponding to $\widetilde{\D}_{n+2}$.  

In Section~\ref{sect:maffei}, we assume that $\bw=(0,\cdots,0,1,0\cdots,0,1,0,\cdots,0)$, where the ones are in positions $k$ and $2n-k$ for some $1\leq k\leq n$, the $k=n$ case giving $\bw=(0,\cdots,0,2,0,\cdots,0)$. Thus, the associated representation of $\fsl_{2n}$, namely the one with highest weight $\varpi_k+\varpi_{2n-k}$, is not only self-dual but also \emph{small} in the sense of \cite{broer}. In this case, Maffei's result~\cite[Theorem 8]{maffei} says that $\fM^{\A_{2n-1}}(\bv,\bw)$ is isomorphic to a certain resolution of a closed subvariety of the Slodowy slice to the orbit with Jordan type $(2n-k,k)$ in the nilpotent cone of $\fsl_{2n}$ (a \emph{two-row} Slodowy slice). See Proposition~\ref{prop:maffei1} for the precise statement. 

This isomorphism descends to one between the affine variety $\fM_1^{\A_{2n-1}}(\bv,\bw)$ and the appropriate subvariety of the nilpotent cone of $\fsl_{2n}$. On the level of these affine varieties, we show in Theorem~\ref{thm:involutions} that the involution $\theta$ of $\fM_1^{\A_{2n-1}}(\bv,\bw)$ corresponds to a `classical' involution of $\fsl_{2n}$, namely one defined by taking negative transpose with respect to a suitable nondegenerate bilinear form on $\C^{2n}$. When $k<n$, the form is skew-symmetric if $k$ is even and symmetric if $k$ is odd; when $k=n$, the type of the form depends on the parity of $n$ and on $(w_+,w_-)$. The proof of this result is made difficult by the fact that Maffei's isomorphism is given by a recursive characterization rather than by an explicit formula; we cannot use the more explicit version found by Mirkovi\'c--Vybornov~\cite{mvy-cr,mvy}, since their slice is not adapted to the bilinear form. We then make some further assumptions on $\bv$ in order to deduce a similar result for $\fM^{\A_{2n-1}}(\bv,\bw)$ itself, Corollary~\ref{cor:involutions}, our answer to question (2).

In Section~\ref{sect:conseq} we derive some consequences of these results. Notably, Theorem~\ref{thm:upstairs} is a description of certain fixed-point subvarieties $\fM^{\A_{2n-1}}(\bv,\bw)^\theta$ in terms of Slodowy varieties for $\fsp_{2n}$ or $\fso_{2n}$. Combining this with Theorem~\ref{thm:fixed-points-AtoD} as suggested by question (3), we obtain the following result. Let $\cS^{\mathrm{C}_n}_{(2n-k,k)}$ and $\cS^{\mathrm{D}_n}_{(2n-k,k)}$ denote the Slodowy slices to the orbits with Jordan type $(2n-k,k)$ in the nilpotent cones of $\fsp_{2n}$ and $\fso_{2n}$, respectively, the former being defined when $k$ is even or $k=n$, and the latter being defined when $k$ is odd or $k=n$. Let $\tcS^{\mathrm{C}_n}_{(2n-k,k)}$ and $\tcS^{\mathrm{D}_n}_{(2n-k,k)}$ denote the Springer resolutions of these slices.

\begin{thm} \label{thm:isomorphisms-intro}
We have isomorphisms of varieties
\[
\begin{split}
\tcS^{\mathrm{C}_n}_{(2n-k,k)}&\cong \begin{cases}
&\fM^{\D_{n+1}}((1,2,\cdots,k,k,\cdots,k,\frac{k}{2},\frac{k}{2}),(0,\cdots,0,1,0,\cdots,0))\\
&\quad\text{ if }k<n\textup{ (}k\text{ even}\textup{)},\text{ where the $1$ on the right is in position $k$,}\\
&\fM^{\D_{n+1}}((1,2,\cdots,n-1,\frac{n}{2},\frac{n}{2}),(0,\cdots,0,1,1))\\
&\quad\text{ if }k=n\text{ is even,}\\
&\fM^{\D_{n+1}}((1,2,\cdots,n-1,\frac{n-1}{2},\frac{n+1}{2}),(0,\cdots,0,2))\\
&\quad\text{ if }k=n\text{ is odd,}
\end{cases}
\\
\tcS^{\D_n}_{(2n-k,k)}&\cong \begin{cases}
&\fM^{\D_{n+1}}((1,2,\cdots,k,k,\cdots,k,\frac{k-1}{2},\frac{k+1}{2}),(0,\cdots,0,1,0,\cdots,0))\\
&\quad\text{ if }k<n\textup{ (}k\text{ odd}\textup{)},\text{ where the $1$ on the right is in position $k$,}\\
&\fM^{\D_{n+1}}((1,2,\cdots,n-1,\frac{n-1}{2},\frac{n+1}{2}),(0,\cdots,0,1,1))\\
&\quad\text{ if }k=n\text{ is odd,}\\
&\fM^{\D_{n+1}}((1,2,\cdots,n-1,\frac{n}{2},\frac{n}{2}),(0,\cdots,0,2))\\
&\quad\text{ if }k=n\text{ is even.}
\end{cases}
\end{split}
\]
In each case we also have two further isomorphisms: between the Slodowy slice $\cS^{\mathrm{C}_n}_{(2n-k,k)}$ or $\cS^{\D_n}_{(2n-k,k)}$ and the affine variety $\fM_1^{\D_{n+1}}(\cdot,\cdot)$; and between the Springer fibre at the base-point of the slice, i.e.\ $\cB^{\mathrm{C}_n}_{(2n-k,k)}$ or $\cB^{\D_n}_{(2n-k,k)}$ as appropriate, and the projective variety $\fL^{\D_{n+1}}(\cdot,\cdot)$. Moreover, the action on $\tcS^{\mathrm{C}_n}_{(2n-k,k)}$ or $\tcS^{\D_n}_{(2n-k,k)}$ of the stabilizer of the $\fsl_2$-triple corresponds in each case to a natural reductive group action on the right-hand side. See \S\ref{subsect:the-end} for the details of these statements.
\end{thm}

\subsection{Relations to other work}

Theorem~\ref{thm:fixed-points-AtoD}, or more generally Theorem~\ref{thm:fixed-points}, is a natural next step from the appearances of the split-quotient quiver in the representation theory of path algebras \cite{RR} and preprojective algebras \cite{demonet} in the presence of a quiver automorphism. Nakajima quiver varieties can be interpreted as moduli spaces of certain representations of preprojective algebras~\cite[Section 1]{c-b}, so Theorem~\ref{thm:fixed-points} can be regarded as a moduli-space version of the equivalence of categories proved by Demonet~\cite[Corollary 3]{demonet}.

One non-trivial case of Theorem~\ref{thm:isomorphisms-intro} was already known, namely that $\tcS^{\mathrm{C}_n}_{(2n-2,2)}\cong\fM^{\D_{n+1}}((1,2,\cdots,2,1,1),(0,1,0,\cdots,0))$ for $n\geq 3$, because each of these varieties is isomorphic to the minimal resolution of a Kleinian singularity of type $\D_{n+1}$, by the results of Slodowy~\cite[Section 6.4, Theorem]{slod} and Kronheimer~\cite[Corollary 3.12]{kron}, respectively (see Cassens--Slodowy~\cite[Section 7, Theorem]{cass-slod} for an algebraic version of Kronheimer's result). It is noteworthy that our methods recover this statement mixing types C$_n$ and D$_{n+1}$, but not the statement $\tcS^{\D_{n+1}}_{(2n-1,3)}\cong\fM^{\D_{n+1}}((1,2,\cdots,2,1,1),(0,1,0,\cdots,0))$, known for the same reasons. We originally expected to generalize the latter statement, by finding an isomorphism between $\tcS^{\D_{n+1}}_{(2n-k+1,k+1)}$ and the quiver variety of type D$_{n+1}$ associated to $\tcS^{\mathrm{C}_n}_{(2n-k,k)}$ in Theorem~\ref{thm:isomorphisms-intro}. This would have partly verified a conjecture made independently by McGerty and Lusztig stating that the isomorphisms of Weyl group representations found by Reeder in \cite{reeder} lift to isomorphisms between the corresponding Nakajima and Slodowy varieties.

As we show in Corollary~\ref{cor:column}, the affine varieties $\cS^{\mathrm{C}_n}_{(2n-k,k)}$ and $\cS^{\D_{n+1}}_{(2n-k+1,k+1)}$ are isomorphic. This sharpens the previously known results that their singularities at the base-point are smoothly equivalent~\cite{kp} and isomorphic as complex analytic germs~\cite{lns}. However, there is no obvious way to deduce from this that the Springer resolutions $\tcS^{\mathrm{C}_n}_{(2n-k,k)}$ and $\tcS^{\D_{n+1}}_{(2n-k+1,k+1)}$ are isomorphic. It would be interesting to check whether the Springer fibres $\cB^{\mathrm{C}_n}_{(2n-k,k)}$ and $\cB^{\D_{n+1}}_{(2n-k+1,k+1)}$ are isomorphic; explicit descriptions of the two-row Springer fibres in type D were given by Ehrig and Stroppel~\cite{ehrigstroppel}.

In the work of Mirkovi\'c--Vybornov \cite{mvy-cr,mvy}, Maffei's result that quiver varieties of type A are isomorphic to Slodowy varieties of type A is supplemented by a result that should be thought of as dual to it: namely, that resolutions of slices in the affine Grassmannian of type A are also isomorphic to Slodowy varieties of type A. The (partly conjectural) duality here is that of symplectic dual pairs, in the sense of~\cite[Remark 1.5]{blpw}. Along the same lines, the isomorphisms of Theorem~\ref{thm:isomorphisms-intro} could be thought of as dual to some of the connections proved in~\cite{ah} between the small part of the affine Grassmannian and the nilpotent cone, in types other than A. This may explain why one should have to restrict to quiver varieties corresponding to small representations and Slodowy slices to `big' nilpotent orbits.

Finally, we recall that Losev~\cite{losev} has used Maffei's result to prove that the quantum Hamiltonian reductions associated to quiver varieties of type A are isomorphic to parabolic $W$-algebras of type A. We hope that our results on quiver varieties of type D will have a similar application.

\subsection{Acknowledgements}

We are grateful for the helpful comments of H.~Kraft, G.~Lusztig, A.~Maffei, K.~McGerty, H.~Nakajima, C.~Sorger, C.~Stroppel, and an anonymous referee.


\section{Split-quotient quivers}
\label{sect:split-quotient}


In this section we recall the general definition, due to Reiten and Riedtmann \cite{RR}, of the split-quotient quiver (our terminology) and discuss the relationship between this definition and constructions of folding and the McKay correspondence.

\subsection{Quivers and admissible automorphisms}
\label{subsect:adm-aut}
Let $(I,H)$ be a finite graph without edge loops. Here $I$ denotes the set of vertices and $H$ denotes the set of pairs of an edge together with an orientation of that edge, with $s,t:H\to I$ being the source and target functions and $\overline{\phantom{h}}:H\to H$ the involution given by reversing orientation. Let $\Omega$ be an orientation of $(I,H)$, i.e.\ a subset of $H$ such that $H$ is the disjoint union of $\Omega$ and $\overline{\Omega}$, and assume that $\Omega$ contains no oriented cycle. The triple $(I,H,\Omega)$ is what we refer to as a quiver.

Following Lusztig, we consider an \emph{admissible automorphism} $a$ of $(I,H)$. This comprises permutations of $I$ and $H$, each denoted $a$, such that $s(a(h))=a(s(h))$, $t(a(h))=a(t(h))$, $\overline{a(h)}=a(\overline{h})$, and no two adjacent vertices belong to the same $\langle a\rangle$-orbit. We assume moreover that $a(\Omega)=\Omega$; given $(I,H,a)$, there is always at least one orientation $\Omega$ compatible with $a$ in this sense~\cite[Section 12.1.1]{lusztig}. 

For any $i\in I$, let $d_i$ be the smallest positive integer such that $a^{d_i}(i)=i$. In other words, $d_i$ is the size of the $\langle a\rangle$-orbit of $i$. For $h\in H$, define $d_h$ similarly; clearly $d_h$ is a common multiple of $d_{s(h)}$ and $d_{t(h)}$ (we do not assume that it is the least common multiple, i.e.\ we allow two edges in the same $\langle a\rangle$-orbit to have the same source and target). Let $\bn$ be a common multiple of all $d_i$ and $d_h$. Then $a^\bn$ is the identity on $I$ and on $H$; in examples we will often take $\bn$ to be the order of the automorphism $a$, but it is convenient not to make this a general assumption. Set $e_i=\bn/d_i$, $e_h=\bn/d_h$.
 
\subsection{The split-quotient quiver}
\label{subsect:split-quotient}

We define a quiver $(\tI,\tH,\tOmega)$, which we call the \emph{split-quotient quiver} of $(I,H,\Omega,a,\bn)$.  In the notation of \cite{RR,demonet}, this is $Q_G$ where $Q = (I,H,\Omega)$ and $G$ is the cyclic group of order $\bn$ where the generator acts by $a$.

First we form the quotient quiver $(\hI,\hH,\hOmega)$ of $(I,H,\Omega)$ by $\langle a\rangle$. Namely, let $\hI$ denote a set of representatives for the $\langle a\rangle$-orbits in $I$, let $\hH$ be the set of $\langle a\rangle$-orbits in $H$, and let $\hOmega\subset\hH$ be the set of $\langle a\rangle$-orbits in $\Omega$. (The reason we have defined $\hI$ to be a set of representatives for the orbits, rather than the set of orbits itself, is just notational convenience.) For $\hh=\langle a\rangle\cdot h\in\hH$, define $s(\hh)$ to be the representative in $\hI$ of $\langle a\rangle\cdot s(h)$, define $t(\hh)$ similarly, and define $d_{\hh}=d_h$, $e_{\hh}=e_h$. Define $\overline{\phantom{h}}:\hH\to \hH$ by $\overline{\langle a\rangle\cdot h}=\langle a \rangle\cdot\overline{h}$. Then $(\hI,\hH,\hOmega)$ is another quiver with no edge-loops or directed cycles.

Now let $\tI$ be the set of pairs $(i,\zeta)$ where $i\in\hI$ and $\zeta\in\mu_{e_i}$  (the group of complex $(e_i)$th roots of $1$). Let $\tH$ be the set of triples $(\hh,\zeta,\zeta')$ where $\hh\in\hH$, $\zeta\in\mu_{e_{s(\hh)}}$, $\zeta'\in\mu_{e_{t(\hh)}}$, and 
\begin{equation} \label{eqn:zeta-first}
\zeta^{e_{s(\hh)}/e_{\hh}}=(\zeta')^{e_{t(\hh)}/e_{\hh}}.
\end{equation}
Let $\tOmega$ be the set of $(\hh,\zeta,\zeta')\in\tH$ such that $\hh\in\hOmega$. For $\th=(\hh,\zeta,\zeta')\in\tH$, define $s(\th)=(s(\hh),\zeta)$ and $t(\th)=(t(\hh),\zeta')$. Define $\overline{\phantom{h}}:\tH\to \tH$ by $\overline{(\hh,\zeta,\zeta')}=(\overline{\hh},\zeta',\zeta)$. Then $(\tI,\tH,\tOmega)$ is another quiver with no edge-loops or directed cycles.

The split-quotient quiver $(\tI,\tH,\tOmega)$ comes with its own natural admissible automorphism $\ta$, defined by
\begin{equation}
\ta(i,\zeta)=(i,\zeta\eta^{d_i}),\quad \ta(\hh,\zeta,\zeta')=(\hh,\zeta\eta^{d_{s(\hh)}},\zeta'\eta^{d_{t(\hh)}}),
\end{equation}
where $\eta$ is a fixed primitive $\bn$th root of $1$. Notice that as a set of representatives for the $\langle\ta\rangle$-orbits in $\tI$ we can take $\{(i,1)\,|\,i\in\hI\}$, which we can identify with $\hI$. (However, without further assumptions we cannot necessarily identify the set of $\langle\ta\rangle$-orbits in $\tH$ with $\hH$.) 

Lusztig in~\cite[Section 16]{lusztig-canonical} associated a symmetrizable Cartan matrix $(c_{ij})_{i,j\in\hI}$ to a triple $(I,H,a)$ as above, where the rows and columns are indexed by $\hI$ and the off-diagonal entries are defined by
\[
\begin{split}
-c_{ij}&=d_i^{-1}\;|\{h\in H\,|\,s(h)\in\langle a\rangle\cdot i,\,t(h)\in\langle a\rangle\cdot j\}|\\
&=d_i^{-1}\sum_{\substack{\hh\in \hH\\s(\hh)=i, t(\hh)=j}} d_{\hh}.
\end{split}
\]

\begin{lem} \label{lem:langlands}
The Cartan matrix associated to $(\tI,\tH,\ta)$ is the transpose of that associated to $(I,H,a)$.
\end{lem}

\begin{proof}
If $(c_{(i,1)(j,1)})_{i,j\in\hI}$ denotes the Cartan matrix associated to $(\tI,\tH,\ta)$, we have:
\[
\begin{split}
-c_{(i,1)(j,1)}&=e_i^{-1}\;|\{(\hh,\zeta,\zeta')\in\tH\,|\,s(\hh)=i,\,t(\hh)=j\}|\\
&=e_i^{-1}\sum_{\substack{\hh\in \hH\\s(\hh)=i, t(\hh)=j}} \frac{e_ie_j}{e_{\hh}}\\
&=d_j^{-1}\sum_{\substack{\hh\in \hH\\s(\hh)=i, t(\hh)=j}} d_{\hh}=-c_{ji},
\end{split}
\]
as claimed.
\end{proof}

\begin{ex} \label{ex:typea-typed}

From Section~\ref{sect:maffei} onwards, we will concentrate on the following example. Fix a positive integer $n$. Let $(I,H,\Omega)$ be the quiver of type A$_{2n-1}$ defined as follows: $I=\{1,2,\cdots,2n-1\}$, there is an edge in $H$ from $i$ to $j$ if and only if $|i-j|=1$, and $\Omega$ consists of the oriented edges $h$ such that $t(h)$ is closer to the middle vertex $n$ than $s(h)$ is. 
Let $a$ be the involution of the quiver that interchanges $i$ and $2n-i$. Then $d_i$ equals $2$ for $i\neq n$ and $1$ for $i=n$, and $d_h=2$ for all $h$. We take $\bn=2$, so $e_i$ equals $1$ for $i\neq n$ and $2$ for $i=n$, and $e_h=1$ for all $h$.

The quotient quiver $(\hI,\hH,\hOmega)$ is then of type A$_{n}$: we set $\hI=\{1,\cdots,n\}$, so again there is an edge in $\hH$ from $i$ to $j$ if and only if $|i-j|=1$, and $\hOmega$ consists of the edges from $i$ to $i+1$ for $1\leq i\leq n-1$. Since $e_i=1$ for all $i\neq n$, the only modification required to obtain the split-quotient quiver $(\tI,\tH,\tOmega)$ is that the vertex $n$ should be split into two vertices $+$ and $-$, short for $(n,+1)$ and $(n,-1)$ respectively. Thus $(\tI,\tH,\tOmega)$ is a quiver of type D$_{n+1}$. (Here we follow the usual convention that $\D_2$ means $\text{A}_1\times\text{A}_1$ and $\D_3$ means A$_3$.) 
Pictorially:
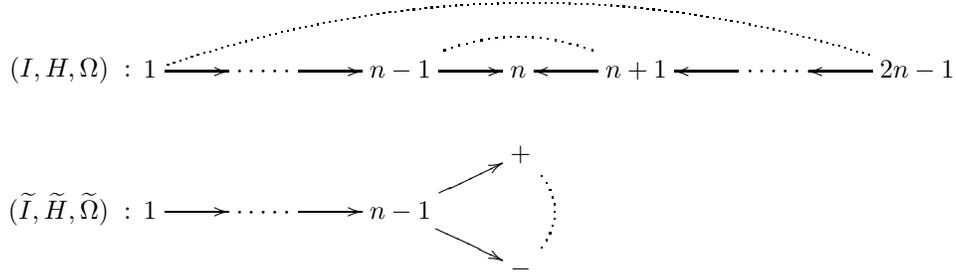
\begin{equation*}
\begin{split}
\\
(I,H,\Omega):\,&
\vcenter{
\xymatrix{
1\ar@{->}[r]\ar@/^6ex/@{.}[rrrrrr]&\cdot\cdots\cdot\ar@{->}[r]&n-1\ar@{->}[r]\ar@/^3ex/@{.}[rr]&n\ar@{<-}[r]&n+1\ar@{<-}[r]&\cdot\cdots\cdot\ar@{<-}[r]&2n-1
}
}
\\
\\
(\tI,\tH,\tOmega):\,&
\vcenter{
\xymatrix@R=2ex{
&&&+\ar@/^3ex/@{.}[dd]\\
1\ar@{->}[r]&\cdot\cdots\cdot\ar@{->}[r]&n-1\ar@{->}[ur]\ar@{->}[dr]&\\
&&&-
}
}
\end{split}
\end{equation*}
Here, the dotted lines indicate the non-trivial orbits of the respective involutions $a$ and $\ta$. Lemma~\ref{lem:langlands} is well known in this case: the Cartan matrices associated to $(I,H,a)$ and $(\tI,\tH,\ta)$ are those of type B$_n$ and C$_n$ respectively.

Observe that in this example the split-quotient quiver of $(\tI,\tH,\tOmega,\ta,\bn)$ can be identified with the original quiver $(I,H,\Omega)$. 
\end{ex}
  
\begin{ex} \label{ex:dynkin}
More generally, suppose $(I,H,\Omega)$ is an orientation of any simply-laced finite or affine Dynkin diagram, with $a$ being a non-trivial admissible automorphism that fixes at least one vertex and $\bn$ being its order. A list of such triples $(I,H,a)$ is given in~\cite[Section 14.1]{lusztig}; the associated Cartan matrices are exactly the non-symmetric Cartan matrices of finite or affine type. It is trivial to check that in each case $(\tI,\tH,\ta)$ is another triple of the same kind with the order of $\ta$ also being $\bn$; in accordance with Lemma~\ref{lem:langlands}, $(\tI,\tH,\ta)$ is the triple giving the transpose Cartan matrix. In other words, $(\tI,\tH,\ta)$ is obtained by `unfolding' the non-simply-laced Dynkin diagram that is dual to that obtained by `folding' $(I,H,a)$. This rule is clearly self-inverse, so in all these cases the split-quotient quiver of $(\tI,\tH,\tOmega,\ta,\bn)$ can be identified with the original quiver $(I,H,\Omega)$. 
\end{ex}

\begin{ex} \label{ex:cyclic}
Suppose $(I,H,\Omega)$ is the quiver of type $\widetilde{\A}_3$ where $I=\{0,1,2,3\}$ and $\Omega$ consists of the edges from $1$ to $2$, $3$ to $2$, $1$ to $0$, and $3$ to $0$. Let $a$ be the `rotation' that interchanges $0$ and $2$, and $1$ and $3$, and let $\bn=2$. Then $e_i=1$ for all $i$, so the split-quotient quiver is the same as the quotient quiver $(\hI,\hH,\hOmega)$, which is of type $\widetilde{\A}_1$: explicitly, $\hI=\{1,2\}$ with $\hOmega$ consisting of two edges from $1$ to $2$. In this case $\ta$ is the identity.
\end{ex}

\begin{ex} \label{ex:redundant}
We have allowed $\bn$ to be a multiple of the order $\bn'$ of the automorphism $a$, rather than requiring $\bn=\bn'$. This does not create any substantially new examples of split-quotient quivers, since it follows easily from the definition that $(\tI,\tH,\tOmega)$ is the disconnected union of $\bn/\bn'$ copies of the split-quotient quiver defined using $\bn'$ instead of $\bn$. In particular, if $a$ is the identity automorphism, $(\tI,\tH,\tOmega)$ is the disconnected union of $\bn$ copies of $(I,H,\Omega)$.
\end{ex}

\subsection{Interpretation via the McKay correspondence}
\label{subsect:McKay}
Recall the McKay correspondence~\cite{mckay} between simply-laced affine Dynkin diagrams and conjugacy classes of nontrivial finite subgroups $\Gamma$ of $SL_2(\C)$, in which the vertices of the diagram $\widetilde{\Delta}_\Gamma$ label the isomorphism classes of irreducible representations of $\Gamma$. Slodowy observed~\cite[Section 6.2, Appendix III]{slod} that certain automorphisms of affine Dynkin diagrams arise from pairs $(\Gamma,\Gamma')$ of such groups where $\Gamma\subsetneq\Gamma'$, $\Gamma$ is normal in $\Gamma'$, the quotient $\Gamma'/\Gamma$ is cyclic of order $\bn$ and the centralizer of $\Gamma$ in $\Gamma'$ is contained in $\Gamma$. (This characterization of the relevant pairs is from~\cite[Section 2]{lusztig-dual}.) Explicitly, given such a pair $(\Gamma,\Gamma')$:
\begin{enumerate}
\item the conjugation action of a generator of $\Gamma'/\Gamma$ on irreducible representations of $\Gamma$ defines an order-$\bn$ automorphism $a$ of $\widetilde{\Delta}_\Gamma$, and
\item tensoring irreducible representations of $\Gamma'$ with a generator of $\widehat{\Gamma'/\Gamma}$ (thought of as a linear character of $\Gamma'$) defines an order-$\bn$ automorphism $a'$ of $\widetilde{\Delta}_{\Gamma'}$.
\end{enumerate}
By inspection of the various cases, automorphisms arising in either of these two ways are admissible (and $\bn$ is always either $2$ or $3$). The nontrivial admissible automorphisms of simply-laced affine Dynkin diagrams that do not arise thus are the `rotations' of $\widetilde{\A}_m$ and the order-$4$ automorphisms of $\widetilde{\D}_{2m}$. 

\begin{lem} \label{lem:clifford}
Let $(\Gamma,\Gamma')$ be a pair of subgroups of $SL_2(\C)$ satisfying the above conditions. Neglecting orientations, the split-quotient quiver of $(\widetilde{\Delta}_\Gamma,a)$ is $(\widetilde{\Delta}_{\Gamma'},a')$ and vice versa.
\end{lem}

\begin{proof}
This is obvious from the list of cases in~\cite[Appendix III]{slod}, but it is worth pointing out that it also follows from the basic Clifford-theory relationship between irreducible representations of $\Gamma$ and $\Gamma'$, at its simplest here because $\Gamma'/\Gamma$ is assumed to be cyclic. Explicitly, if $S_i$ is the irreducible representation of $\Gamma$ labelled by the vertex $i$ of $\widetilde{\Delta}_\Gamma$, then the irreducible representations of $\Gamma'$ labelled by the corresponding $\langle a'\rangle$-orbit of vertices of $\widetilde{\Delta}_{\Gamma'}$ are the constituents of $\Ind_\Gamma^{\Gamma'}(S_i)$; and one has the analogous statement for an irreducible representation of $\Gamma'$, using $\Res_\Gamma^{\Gamma'}$. The edges can be matched up using the projection and Mackey formulas. 
\end{proof}

\begin{ex} \label{ex:affinetypea-typed}
One case of such a pair $(\Gamma,\Gamma')$ occurs when $\Gamma$ is cyclic of order $2n$ and $\Gamma'$ is binary dihedral of order $4n$, for $n\geq 2$. Then $\widetilde{\Delta}_\Gamma$ is of type $\widetilde{\A}_{2n-1}$ with $a$ being the involution that fixes the extra vertex and is the usual diagram involution of the sub-diagram of type $\A_{2n-1}$, whereas $\widetilde{\Delta}_{\Gamma'}$ is of type $\widetilde{\D}_{n+2}$ with $a'$ being the involution that exchanges the end vertices in pairs and fixes all other vertices. Importantly for our arguments in \S\ref{subsect:nakajima}, one can recover from this example the $\A_{2n-1}$/$\D_{n+1}$ split-quotient pair of Example~\ref{ex:typea-typed}, by deleting the extra vertex of $\widetilde{\A}_{2n-1}$ and the $\langle a'\rangle$-orbit of the extra vertex of $\widetilde{\D}_{n+2}$.
\end{ex}


\section{Diagram automorphisms of Nakajima quiver varieties}
\label{sect:diag-aut}


In this section, we recall the definition and basic properties of Nakajima quiver varieties, and describe the fixed-point subvariety of such a quiver variety under a suitable kind of diagram automorphism in Theorem \ref{thm:fixed-points}. The description involves the split-quotient quiver of \S\ref{sect:split-quotient}.

\subsection{Quiver varieties}
\label{subsect:quiver-vars}

We follow the notation of Nakajima~\cite{nak2} with minor modifications. 

Let $V=\bigoplus_{i\in I}V_i$ and $W=\bigoplus_{i\in I}W_i$ be $I$-graded finite-dimensional complex vector spaces. Define the vector space
\begin{equation}
\bM(V,W)=\bigoplus_{h\in H}\Hom(V_{s(h)},V_{t(h)})\oplus\bigoplus_{i\in I}\Hom(W_i,V_i)\oplus\bigoplus_{i\in I}\Hom(V_i,W_i).
\end{equation}
An element of $\bM(V,W)$ will be written as a tuple $(B_h,\Gamma_i,\Delta_i)$,
it being understood that $h$ runs over $H$ and $i$ over $I$.

Let $G_V=\prod_{i\in I}\GL(V_i)$ and $G_W=\prod_{i\in I}\GL(W_i)$. These groups have commuting actions on $\bM(V,W)$ defined as follows: for $(g_i)\in G_V$ and $(\alpha_i)\in G_W$,
\begin{equation}
\begin{split}
(g_i)\cdot(B_h,\Gamma_i,\Delta_i)&=(g_{t(h)}B_hg_{s(h)}^{-1},g_i\Gamma_i,\Delta_i g_i^{-1}),\\
(\alpha_i)\cdot(B_h,\Gamma_i,\Delta_i)&=(B_h,\Gamma_i\alpha_i^{-1},\alpha_i\Delta_i).
\end{split}
\end{equation}

Let $\Lambda(V,W)$ be the closed subvariety of $\bM(V,W)$ defined by the following collection of equations (one equation for each $i\in I$, with both sides belonging to $\End(V_i)$):
\begin{equation} \label{eqn:adhm}
\sum_{\substack{h\in\Omega\\s(h)=i}}B_{\overline{h}}B_h\;-\;\sum_{\substack{h\in\Omega\\t(h)=i}}B_hB_{\overline{h}}=\Gamma_i\Delta_i.
\end{equation}
It is obvious that $\Lambda(V,W)$ is preserved by the actions of $G_V$ and $G_W$.

Following~\cite{lusztig-quiver, maffei} we say that a point $(B_h,\Gamma_i,\Delta_i)\in\Lambda(V,W)$ is \emph{stable} if the subspace $\bigoplus_{i\in I}\img(\Gamma_i)$ generates $V$ under the action of all the maps $B_h$. Note that this is dual to the stability condition used in~\cite{nak2}, see~\cite[Remark 6]{maffei}. Let $\Lambda(V,W)^s$ be the (possibly empty) open subvariety of $\Lambda(V,W)$ consisting of stable points. It is obvious that $\Lambda(V,W)^s$ is preserved by the actions of $G_V$ and $G_W$.

\begin{prop}\cite[Lemma 3.10]{nak2} \label{prop:triv-stab}
The variety $\Lambda(V,W)^s$ is nonsingular when it is nonempty, and $G_V$ acts freely on $\Lambda(V,W)^s$. 
\end{prop}

Let $\C[\Lambda(V,W)]$ denote the ring of regular functions on $\Lambda(V,W)$. Consider the graded algebra $A(V,W)=\bigoplus_{n\geq 0}A(V,W)_n$, where
\begin{equation}
\begin{split}
&A(V,W)_n=\left\{f\in\C[\Lambda(V,W)]\,\left|\,
f((g_i)\cdot x)=(\prod_{i\in I}\det(g_i))^n f(x),\right.\right.\\
&\left.\phantom{f((g_i)\cdot x)=(\prod_{i\in I}\det(g_i))^n f(x)}\,\text{ for all }(g_i)\in G_V,\,x\in\Lambda(V,W)\right\}.
\end{split}
\end{equation}
Note that, in accordance with our choice of stability condition, we have replaced $\det^{-1}$ in Nakajima's definition~\cite{nak2} with $\det$. We then define the quiver variety $\fM(V,W)$, the associated affine varieties $\fM_0(V,W)$ and $\fM_1(V,W)$, and the associated projective variety $\fL(V,W)$ by
\begin{equation}
\begin{split}
\fM(V,W)&=\Proj\, A(V,W),\\
\fM_0(V,W)&=\Spec A(V,W)_0=G_V\backslash\!\backslash\Lambda(V,W),\\
\fM_1(V,W)&=\pi(\fM(V,W))\;\text{ (a closed subvariety of $\fM_0(V,W)$),}\\
\fL(V,W)&=\pi^{-1}([0])\;\text{ (a closed subvariety of $\fM(V,W)$),}
\end{split}
\end{equation} 
where $\pi:\fM(V,W)\to\fM_0(V,W)$ is the canonical projective morphism, and $[0]\in\fM_0(V,W)$ is the maximal ideal of $A(V,W)_0$ consisting of functions vanishing at $0\in\Lambda(V,W)$. The group $G_W$ clearly acts on $A(V,W)$ by grading-preserving automorphisms, so it acts on $\fM(V,W)$, $\fM_0(V,W)$, $\fM_1(V,W)$ and $\fL(V,W)$, and $\pi$ is $G_W$-equivariant.

\begin{prop}\cite[Lemma 3.8, Corollary 3.12, Theorem 3.21]{nak2} \label{prop:geom-pts}
\begin{enumerate}
\item For $x\in\Lambda(V,W)$, $x$ is stable if and only if there exists some $f\in A(V,W)_n$, $n>0$, such that $f(x)\neq 0$.
\item The resulting morphism $\Lambda(V,W)^s\to\fM(V,W)$ is a geometric quotient by the action of $G_V$. In particular, the geometric points of $\fM(V,W)$ are in bijection with the $G_V$-orbits in $\Lambda(V,W)^s$.
\item The variety $\fM(V,W)$ is nonsingular when it is nonempty.
\item When $\fM(V,W)$ is nonempty, so is $\fL(V,W)$ \textup{(}i.e.\ $[0]\in\fM_1(V,W)$\textup{)}.
\end{enumerate}
\end{prop}

A supplementary result was shown by Crawley--Boevey:

\begin{prop}\cite[Section 1]{c-b}\label{prop:connected}
The variety $\fM(V,W)$ is connected when it is nonempty.
\end{prop}

Recall that the geometric points of $\fM_0(V,W)$ are in bijection with the closed $G_V$-orbits in $\Lambda(V,W)$, and that the map induced by $\pi$ on geometric points sends a $G_V$-orbit in $\Lambda(V,W)^s$ to the unique closed $G_V$-orbit in its closure in $\Lambda(V,W)$. Hence the geometric points of $\fL(V,W)$ are in bijection with the $G_V$-orbits in $\Lambda(V,W)^s$ whose closure in $\Lambda(V,W)$ contains $0$. Lusztig gave an alternative characterization of these orbits:

\begin{prop} \cite[Proposition 14.2(a)]{lusztig-canonical} \cite[Lemma 2.22]{lusztig-quiver} \label{prop:nilpotency}
For $(B_h,\Gamma_i,\Delta_i)\in\Lambda(V,W)^s$, the closure of $G_V\cdot(B_h,\Gamma_i,\Delta_i)$ in $\Lambda(V,W)$ contains $0$ if and only if:
\begin{enumerate}
\item $\Delta_i=0$ for all $i\in I$, and
\item $(B_h)$ is nilpotent, i.e.\ every sufficiently long product $B_{h_1}B_{h_2}\cdots B_{h_s}$ is zero.
\end{enumerate}
Moreover, if the underlying graph $(I,H)$ is of simply-laced finite Dynkin type, then condition \textup{(1)} implies condition \textup{(2)}.
\end{prop}

The varieties $\bM(V,W)$, $\Lambda(V,W)$, etc.\ clearly depend (up to isomorphism) only on the graded dimensions $\bv=(v_i)$ and $\bw=(w_i)$ where $v_i=\dim V_i$ and $w_i=\dim W_i$, and it is therefore common to rename them as $\bM(\bv,\bw)$, $\Lambda(\bv,\bw)$, etc., when it is not important to stress the particular graded vector spaces $V$ and $W$.

\begin{rmk} \label{rmk:weight-space}
In Nakajima's geometric construction~\cite{nak2} of integrable highest-weight modules for the Kac--Moody algebra with symmetric Cartan matrix given by $(I,H)$, the top-degree homology of the variety $\fM(\bv,\bw)$ is a weight space of the module with highest weight $\sum_{i\in I}w_i\varpi_i$, namely the one with weight $\sum_{i\in I}w_i\varpi_i-\sum_{i\in I}v_i\alpha_i$, where $\varpi_i$ and $\alpha_i$ denote the fundamental weights and simple roots. Moreover, the irreducible components of $\fL(\bv,\bw)$ index a basis of this weight space. In particular, this means that $\fM(\bv,\bw)$ is nonempty precisely when the corresponding weight space does not vanish. This construction plays no logical role in the present paper.
\end{rmk}

\begin{ex} \label{ex:typea}
Fix a positive integer $n$. 
Let $(I,H,\Omega)$ be the quiver of type A$_{2n-1}$ from Example \ref{ex:typea-typed}. 
We indicate this particular quiver by means of a superscript $\A_{2n-1}$, as in the notation $\fM^{\A_{2n-1}}(\bv,\bw)$ of the introduction.
When considering this quiver, to make the notation more transparent, we write $B_{i,i+1}$, say, instead of $B_h$ where $s(h)=i+1$ and $t(h)=i$. Graphically, the configuration of maps in the definition of $\bM^{\A_{2n-1}}(V,W)$ may be represented as follows:
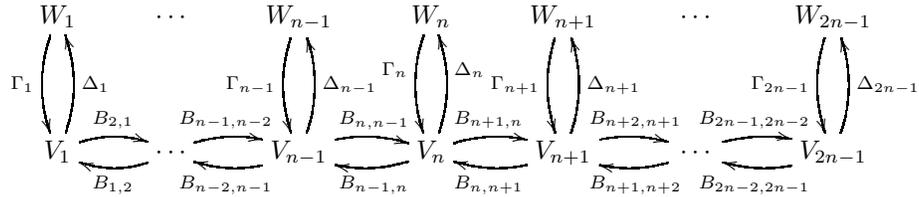
\begin{equation*}
\vcenter{
\xymatrix@R=15pt{
W_1\ar@/_/[dd]_-{\Gamma_1}&\cdots&W_{n-1}\ar@/_/[dd]_-{\Gamma_{n-1}}&W_n\ar@/_/[dd]_(.4){\Gamma_n}&W_{n+1}\ar@/_/[dd]_-{\Gamma_{n+1}}&\,\cdots\,&W_{2n-1}\ar@/_/[dd]_-{\Gamma_{2n-1}}
\\
\\
V_{1}\ar@/^/[r]^-{B_{2,1}}\ar@/_/[uu]_-{\Delta_1}&\cdots\ar@/^/[l]^-{B_{1,2}}\ar@/^/[r]^-{B_{n-1,n-2}}&V_{n-1}\ar@/^/[l]^-{B_{n-2,n-1}}\ar@/^/[r]^-{B_{n,n-1}}\ar@/_/[uu]_-{\Delta_{n-1}}&V_{n}\ar@/^/[l]^-{B_{n-1,n}}\ar@/^/[r]^-{B_{n+1,n}}\ar@/_/[uu]_(.6){\Delta_n}&V_{n+1}\ar@/^/[l]^-{B_{n,n+1}}\ar@/^/[r]^-{B_{n+2,n+1}}\ar@/_/[uu]_-{\Delta_{n+1}}&\,\cdots\,\ar@/^/[l]^-{B_{n+1,n+2}}\ar@/^/[r]^-{B_{2n-1,2n-2}}&V_{2n-1}\ar@/^/[l]^-{B_{2n-2,2n-1}}\ar@/_/[uu]_-{\Delta_{2n-1}}
}
}
\end{equation*}

We extend the notation $B_{i,i+1}$ to paths of length $>1$, recording in the subscript only the start and end of the path and any intermediate turning points. For instance, $B_{i,i-2,i+1}$ denotes $B_{i,i-1}B_{i-1,i-2}B_{i-2,i-1}B_{i-1,i}B_{i,i+1}$. In this notation, the defining equations \eqref{eqn:adhm} of $\Lambda^{\A_{2n-1}}(V,W)$ become:
\begin{equation} \label{eqn:quadratic}
\begin{split}
B_{1,2,1}&=\Gamma_{1}\Delta_{1},\\
B_{2,3,2}-B_{2,1,2}&=\Gamma_{2}\Delta_{2},\\
&\vdots\\
B_{n-1,n,n-1}-B_{n-1,n-2,n-1}&=\Gamma_{n-1}\Delta_{n-1},\\
-B_{n,n-1,n}-B_{n,n+1,n}&=\Gamma_{n}\Delta_{n},\\
B_{n+1,n,n+1}-B_{n+1,n+2,n+1}&=\Gamma_{n+1}\Delta_{n+1},\\
&\vdots\\
B_{2n-2,2n-3,2n-2}-B_{2n-2,2n-1,2n-2}&=\Gamma_{2n-2}\Delta_{2n-2},\\
B_{2n-1,2n-2,2n-1}&=\Gamma_{2n-1}\Delta_{2n-1}.
\end{split}
\end{equation}
\end{ex}

\begin{ex} \label{ex:typed}
The other main example for this paper is the quiver of type $\D_{n+1}$ (for $n$ a positive integer) denoted $(\tI,\tH,\tOmega)$ in Example ~\ref{ex:typea-typed}. 
We indicate this particular quiver by means of a superscript $\D_{n+1}$.
With notation analogous to that of Example~\ref{ex:typea}, the configuration of maps in the definition of $\bM^{\D_{n+1}}(V,W)$ may be represented:
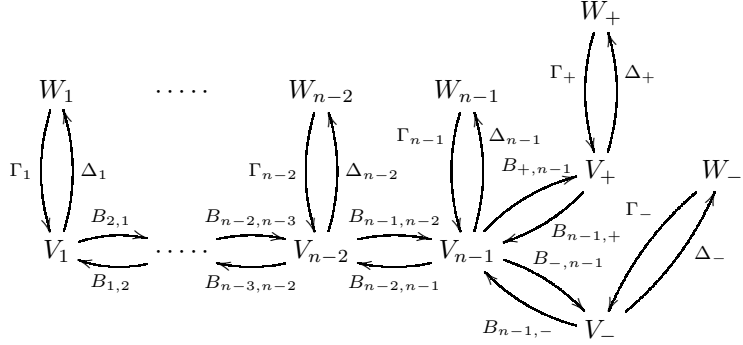
\begin{equation*}
\vcenter{
\xymatrix@R=15pt{
&&&&W_+\ar@/_/[dd]_(.4){\Gamma_+}\\
W_1\ar@/_/[dd]_-{\Gamma_1}&\cdot\cdots\cdot&W_{n-2}\ar@/_/[dd]_-{\Gamma_{n-2}}&W_{n-1}\ar@/_/[dd]_(.3){\Gamma_{n-1}}&
\\
&&&&V_+\ar@/_/[uu]_(.6){\Delta_+}\ar@/^/[dl]^(.4){B_{n-1,+}}&W_-\ar@/_/[ddl]_(.4){\Gamma_-}
\\
V_{1}\ar@/^/[r]^-{B_{2,1}}\ar@/_/[uu]_-{\Delta_1}&\cdot\cdots\cdot\ar@/^/[l]^-{B_{1,2}}\ar@/^/[r]^-{B_{n-2,n-3}}&V_{n-2}\ar@/^/[l]^-{B_{n-3,n-2}}\ar@/^/[r]^-{B_{n-1,n-2}}\ar@/_/[uu]_-{\Delta_{n-2}}&V_{n-1}\ar@/^/[l]^-{B_{n-2,n-1}}\ar@/^/[ur]^(.7){B_{+,n-1}}\ar@/^/[dr]^(.5){B_{-,n-1}}\ar@/_/[uu]_(.7){\Delta_{n-1}}&\\
&&&&V_-\ar@/_/[uur]_(.6){\Delta_-}\ar@/^/[ul]^(.4){B_{n-1,-}}\\
}
}
\end{equation*}
and the defining equations of $\Lambda^{\D_{n+1}}(V,W)$ are:
\begin{equation} \label{eqn:d-quadratic}
\begin{split}
B_{1,2,1}&=\Gamma_{1}\Delta_{1},\\
B_{2,3,2}-B_{2,1,2}&=\Gamma_{2}\Delta_{2},\\
&\vdots\\
B_{n-2,n-1,n-2}-B_{n-2,n-3,n-2}&=\Gamma_{n-2}\Delta_{n-2},\\
B_{n-1,+,n-1}+B_{n-1,-,n-1}-B_{n-1,n-2,n-1}&=\Gamma_{n-1}\Delta_{n-1},\\
-B_{+,n-1,+}&=\Gamma_{+}\Delta_{+},\\
-B_{-,n-1,-}&=\Gamma_{-}\Delta_{-}.
\end{split}
\end{equation}
\end{ex}

\subsection{Diagram automorphisms}
\label{subsect:diag-aut}

Now let $a$ be an admissible automorphism of the quiver $(I,H,\Omega)$, as in \S\ref{subsect:adm-aut}.
We want to define a diagram automorphism of the quiver variety $\fM(V,W)$ corresponding to $a$. 

For this we need to choose isomorphisms $\varphi_i:V_i\simto V_{a(i)}$ and $\sigma_i:W_i\simto W_{a(i)}$ for all $i\in I$. In particular, we require $v_i=v_{a(i)}$ and $w_i=w_{a(i)}$ for all $i\in I$. As in~\cite[Section 12.1.2]{lusztig}, we impose the condition that 
\begin{equation} \label{eqn:varphi}
\varphi_{a^{d_i-1}(i)}\cdots\varphi_{a(i)}\varphi_i=\mathrm{id}_{V_i}\text{ for all $i\in I$.}
\end{equation} 
Thus, the isomorphisms $\varphi_i$ canonically identify the various vector spaces $V_i$ as $i$ runs over an $\langle a\rangle$-orbit in $I$. 
By contrast, for the isomorphisms $\sigma_i$ we impose only the weaker condition that
\begin{equation} \label{eqn:sigma}
(\sigma_{a^{d_i-1}(i)}\cdots\sigma_{a(i)}\sigma_i)^{e_i}=\mathrm{id}_{W_i}\text{ for all $i\in I$.}
\end{equation}
Here $e_i$ depends on the choice of $\bn$ as in \S\ref{subsect:adm-aut}.
For any $i\in I$ and $\zeta\in\mu_{e_i}$, let $W_{i,\zeta}$ denote the $\zeta$-eigenspace of $\sigma_{a^{d_i-1}(i)}\cdots\sigma_{a(i)}\sigma_i$, and let $w_{i,\zeta}=\dim W_{i,\zeta}$. Thus $W_i=\bigoplus_{\zeta\in\mu_{e_i}}W_{i,\zeta}$, $w_i=\sum_{\zeta\in\mu_{e_i}}w_{i,\zeta}$. It is obvious that $\sigma_i$ restricts to an isomorphism $W_{i,\zeta}\simto W_{a(i),\zeta}$, so $w_{i,\zeta}=w_{a(i),\zeta}$.  

We can now define an automorphism $\theta$ of the vector space $\bM(V,W)$ by the natural formula
\begin{equation}
\begin{split}
&\quad\theta\left((B_h,\Gamma_i,\Delta_i)\right)=\\
&(\varphi_{a^{-1}(t(h))}B_{a^{-1}(h)}\varphi_{a^{-1}(s(h))}^{-1},\varphi_{a^{-1}(i)}\Gamma_{a^{-1}(i)}\sigma_{a^{-1}(i)}^{-1},\sigma_{a^{-1}(i)}\Delta_{a^{-1}(i)}\varphi_{a^{-1}(i)}^{-1}).
\end{split}
\end{equation}
We define automorphisms $\theta$ of $G_V$ and of $G_W$, respectively, by the formulas
\begin{equation}
\begin{split}
\theta((g_i))&=(\varphi_{a^{-1}(i)}g_{a^{-1}(i)}\varphi_{a^{-1}(i)}^{-1}),\\
\theta((\alpha_i))&=(\sigma_{a^{-1}(i)}\alpha_{a^{-1}(i)}\sigma_{a^{-1}(i)}^{-1}).
\end{split}
\end{equation}
From~\eqref{eqn:varphi} and \eqref{eqn:sigma} it is clear that $\theta^{\bn}$ is the identity for all of these meanings of $\theta$. 

It is trivial to check that $\theta(h\cdot x)=\theta(h)\cdot\theta(x)$ for any $h$ in $G_V$ or $G_W$ and any $x$ in $\bM(V,W)$, and that $\theta$ preserves the subvarieties $\Lambda(V,W)$ and $\Lambda(V,W)^s$ of $\bM(V,W)$. Moreover, since the automorphism $\theta$ of $G_V$ preserves the character $(g_i)\mapsto\prod_{i\in I}\det(g_i)$, $\theta$ induces a grading-preserving automorphism of $A(V,W)$. 

We therefore obtain our desired diagram automorphism of $\fM(V,W)$, which we still denote $\theta$. On the level of geometric points, $\theta$ sends the $G_V$-orbit of $x\in\Lambda(V,W)^s$ to the $G_V$-orbit of $\theta(x)$. Similarly, we have an automorphism $\theta$ of $\fM_0(V,W)$, and the morphism $\pi$ is compatible with these automorphisms. Hence the subvarieties $\fL(V,W)$ of $\fM(V,W)$ and $\fM_1(V,W)$ of $\fM_0(V,W)$ are $\theta$-stable.

\begin{ex} \label{ex:typea-inv}
Continue with the notation of Example~\ref{ex:typea}, where the quiver is of type A$_{2n-1}$. Let $a$ be the involution of the quiver that interchanges $i$ and $2n-i$ and take $\bn=2$, as in Example~\ref{ex:typea-typed}, so that $e_i$ equals $1$ for $i\neq n$ and $2$ for $i=n$.

To define the diagram involution $\theta$ in this case, we need isomorphisms $\varphi_i:V_i\simto V_{2n-i}$ and $\sigma_i:W_i\simto W_{2n-i}$ for all $i\in I$. In particular, we require $v_i=v_{2n-i}$ and $w_i=w_{2n-i}$ for all $i$. The conditions~\eqref{eqn:varphi} and~\eqref{eqn:sigma} become that $\varphi_{2n-i}=\varphi_i^{-1}$ and $\sigma_{2n-i}=\sigma_i^{-1}$ for all $i\neq n$, $\varphi_{n}$ is the identity, and $\sigma_{n}$ is an involution. To simplify the notation, we may as well assume that $V_i=V_{2n-i}$ and $W_i=W_{2n-i}$ for every $i$, that $\varphi_i$ is the identity for all $i$, and that $\sigma_i$ is the identity for all $i\neq n$. However, the choice of the involution $\sigma_n:W_n\to W_n$ remains and will play a key role. We write its $(+1)$-eigenspace and $(-1)$-eigenspace as $W_+$ and $W_-$ (abbreviating the notation $W_{n,+1}$ and $W_{n,-1}$ introduced in the general setting), and their dimensions as $w_+$ and $w_-$. Thus $W_n=W_+\oplus W_-$, $w_n=w_++w_-$. 

With these conventions, the involution $\theta$ of $\bM(V,W)$ takes the configuration of maps illustrated in Example~\ref{ex:typea} to the following `$\sigma_n$-modified reflection' of itself:
\begin{equation*}
\vcenter{
\xymatrix@R=15pt{
W_1\ar@/_/[dd]_-{\Gamma_{2n-1}}&\,\cdots\,&W_{n-1}\ar@/_/[dd]_-{\Gamma_{n+1}}&W_n\ar@/_/[dd]_(.3){\Gamma_n\sigma_n}&W_{n+1}\ar@/_/[dd]_-{\Gamma_{n-1}}&\cdots&W_{2n-1}\ar@/_/[dd]_-{\Gamma_{1}}
\\
\\
V_{1}\ar@/^/[r]^(.6){B_{2n-2,2n-1}}\ar@/_/[uu]_-{\Delta_{2n-1}}&\,\cdots\,\ar@/^/[l]^(.4){B_{2n-1,2n-2}}\ar@/^/[r]^-{B_{n+1,n+2}}&V_{n-1}\ar@/^/[l]^-{B_{n+2,n+1}}\ar@/^/[r]^-{B_{n,n+1}}\ar@/_/[uu]_-{\Delta_{n+1}}&V_{n}\ar@/^/[l]^-{B_{n+1,n}}\ar@/^/[r]^-{B_{n-1,n}}\ar@/_/[uu]_(.7){\sigma_n\Delta_n}&V_{n+1}\ar@/^/[l]^-{B_{n,n-1}}\ar@/^/[r]^-{B_{n-2,n-1}}\ar@/_/[uu]_-{\Delta_{n-1}}&\cdots\ar@/^/[l]^-{B_{n-1,n-2}}\ar@/^/[r]^-{B_{1,2}}&V_{2n-1}\ar@/^/[l]^-{B_{2,1}}\ar@/_/[uu]_-{\Delta_{1}}
}
}
\end{equation*}
The fact that $\Lambda(V,W)$ is $\theta$-stable is evident from the explicit equations~\eqref{eqn:quadratic}.
\end{ex}

\subsection{Motivation}
\label{subsect:motivation}

Our aim now is to describe the fixed-point subvariety $\fM(V,W)^\theta$. To motivate the appearance of the split-quotient quiver $(\tI,\tH,\tOmega)$ in the statement (Theorem~\ref{thm:fixed-points}), we make some preliminary observations. 

The geometric points of $\fM(V,W)^\theta$ are in bijection with the $\theta$-stable $G_V$-orbits in $\Lambda(V,W)^s$. For $x=(B_h,\Gamma_i,\Delta_i)\in\Lambda(V,W)^s$, the $G_V$-orbit of $x$ is $\theta$-stable if and only if there is some $(g_i^x)\in G_V$ such that $(g_i^x)\cdot\theta(x)=x$, or in other words
\begin{equation} \label{eqn:fixed}
\begin{split}
B_{a(h)}&=g_{a(t(h))}^x\varphi_{t(h)}B_h(g_{a(s(h))}^x\varphi_{s(h)})^{-1},\\
\Gamma_{a(i)}&=g_{a(i)}^x\varphi_i\Gamma_i\sigma_i^{-1},\\
\Delta_{a(i)}&=\sigma_i\Delta_i(g_{a(i)}^x\varphi_i)^{-1},
\end{split}
\end{equation}
for all $h\in H$ and $i\in I$. If this is the case, then by Proposition~\ref{prop:triv-stab} the element $(g_i^x)\in G_V$ is uniquely determined by $x$.

Suppose~\eqref{eqn:fixed} holds. For any $i\in I$, define 
\begin{equation}
\tau_i^x=g_i^x\varphi_{a^{d_i-1}(i)}\cdots g_{a^2(i)}^x\varphi_{a(i)}g_{a(i)}^x\varphi_i\in GL(V_i).
\end{equation}
Then \eqref{eqn:fixed} implies that
\begin{equation} \label{eqn:intertwine}
\begin{split}
B_h&=(\tau_{t(h)}^x)^{e_{t(h)}/e_h}B_h(\tau_{s(h)}^x)^{-e_{s(h)}/e_h},\\
\Gamma_i&=\tau_i^x\Gamma_i(\sigma_{a^{d_i-1}(i)}\cdots\sigma_{a(i)}\sigma_i)^{-1},\\
\Delta_i&=(\sigma_{a^{d_i-1}(i)}\cdots\sigma_{a(i)}\sigma_i)\Delta_i(\tau_i^x)^{-1}.
\end{split}
\end{equation}
In view of \eqref{eqn:sigma}, \eqref{eqn:intertwine} implies that $((\tau_i^x)^{e_i})\in G_V$ fixes $x$, so $(\tau_i^x)^{e_i}$ is the identity by Proposition~\ref{prop:triv-stab}. 
For any $\zeta\in\mu_{e_i}$, let $V_{i,\zeta}^{x}$ denote the $\zeta$-eigenspace of $\tau_i^x$. Thus $V_i=\bigoplus_{\zeta\in\mu_{e_i}}V_{i,\zeta}^x$. 
Moreover, \eqref{eqn:intertwine} implies that $\Gamma_i(W_{i,\zeta})\subseteq V_{i,\zeta}^x$ and $\Delta_i(V_{i,\zeta}^x)\subseteq W_{i,\zeta}$, while
\begin{equation} \label{eqn:restriction}
B_h(V_{s(h),\zeta}^x)\subseteq\bigoplus_{\substack{\zeta'\in\mu_{e_{t(h)}}\\\zeta\overset{h}{\longrightarrow}\zeta'}}V_{t(h),\zeta'}^x,
\end{equation}
where the notation $\zeta\overset{h}{\longrightarrow}\zeta'$ stands for the following equation (cf.~\eqref{eqn:zeta-first}):
\begin{equation} \label{eqn:zeta}
\zeta^{e_{s(h)}/e_h}=(\zeta')^{e_{t(h)}/e_h}.
\end{equation}

Thus, from our original configuration of linear maps $x=(B_h,\Gamma_i,\Delta_i)$, a new one has emerged where the vector spaces $V_i$ and $W_i$ have been split into direct summands $V_{i,\zeta}^x$ and $W_{i,\zeta}$, the linear maps induced by $\Gamma_i$ and $\Delta_i$ go between the corresponding summands, and the linear maps induced by $B_h$ go from a summand $V_{s(h),\zeta}^x$ to a summand $V_{t(h),\zeta'}^x$ where $\zeta\overset{h}{\longrightarrow}\zeta'$. However, this new configuration of induced linear maps contains redundant information, inasmuch as specifying one map is enough to determine all the maps `in its $\langle a\rangle$-orbit', by~\eqref{eqn:fixed}. Quotienting out the $\langle a\rangle$-action, we exactly get a configuration of linear maps for the split-quotient quiver $(\tI,\tH,\tOmega)$ of $(I,H,\Omega,a,\bn)$.

\subsection{Description of the fixed-point subvariety}
\label{subsect:description}

We can now state our theorem describing the fixed-point subvariety $\fM(V,W)^\theta$ in terms of quiver varieties for the split-quotient quiver $(\tI,\tH,\tOmega)$. 

Note that on $\fM(V,W)^\theta$ we have an action of the fixed-point subgroup $G_W^\theta$ of $G_W$, consisting of those $(\alpha_i)\in G_W$ such that $\alpha_{a(i)}=\sigma_i\alpha_i\sigma_i^{-1}$ for all $i$. Since this condition implies that $\alpha_i$ commutes with $\sigma_{a^{d_i-1}(i)}\cdots\sigma_{a(i)}\sigma_i$, $G_W^\theta$ preserves each subspace $W_{i,\zeta}$. Let $\tW$ be the $\tI$-graded vector space $\bigoplus_{(i,\zeta)\in\tI}W_{i,\zeta}$, and accordingly define $G_{\tW}$ to be $\prod_{(i,\zeta)\in\tI}GL(W_{i,\zeta})$. Then we have a group isomorphism
\begin{equation} \label{eqn:gp-isom}
\rho:G_W^\theta\simto G_{\tW}:(\alpha_i)\mapsto(\alpha_i|_{W_{i,\zeta}}).
\end{equation} 

Let $\cD(\bv)$ denote the set of $\tI$-tuples $\tbv=(v_{i,\zeta})$ of nonnegative integers satisfying $\sum_{\zeta\in\mu_{e_i}}v_{i,\zeta}=v_i$ for all $i\in\hI$. For every $\tbv=(v_{i,\zeta})\in\cD(V)$, we fix, for each $i\in\hI$, a direct sum decomposition $V_i=\bigoplus_{\zeta\in\mu_{e_i}}V_{i,\zeta}^{\tbv}$ satisfying $\dim V_{i,\zeta}^{\tbv}=v_{i,\zeta}$. We then let $\tV^{\tbv}$ be the $\tI$-graded vector space $\bigoplus_{(i,\zeta)\in\tI}V_{i,\zeta}^{\tbv}$.

\begin{thm} \label{thm:fixed-points}
Let $\fM(V,W)^\theta$ be the fixed-point subvariety of the diagram automorphism $\theta$ of the quiver variety $\fM(V,W)$ for the quiver $(I,H,\Omega)$, defined as above. Then there is a $G_{\tW}$-equivariant variety isomorphism
\begin{equation}
\Psi:\coprod_{\tbv\in\cD(\bv)}\fM(\tV^{\tbv},\tW)\simto \fM(V,W)^\theta.
\end{equation} 
Here the domain is the disconnected union of the quiver varieties $\fM(\tV^{\tbv},\tW)$ for the quiver $(\tI,\tH,\tOmega)$, and $G_{\tW}$ acts on $\fM(V,W)^\theta$ via the isomorphism $\rho$ of \eqref{eqn:gp-isom}. 
\end{thm}

Note that if $\fM(V,W)^\theta$ is empty, Theorem~\ref{thm:fixed-points} implies that every $\fM(\tV^{\tbv},\tW)$ is empty also. If $\fM(V,W)^\theta$ is nonempty, it is certainly possible that some of the varieties $\fM(\tV^{\tbv},\tW)$ are empty; in view of Proposition~\ref{prop:connected}, Theorem~\ref{thm:fixed-points} implies that the nonempty varieties $\fM(\tV^{\tbv},\tW)$ are respectively isomorphic to the connected components of $\fM(V,W)^\theta$.

Without specifying the graded vector spaces, the isomorphism of Theorem~\ref{thm:fixed-points} could be written
\begin{equation} \label{eqn:fixed-points-data}
\fM(\bv,\bw)^\theta\cong\coprod_{\tbv\in\cD(\bv)}\fM(\tbv,\tbw)
\end{equation} 
where $\tbw=(w_{i,\zeta})$ denotes the graded dimension of $\tW$. It is clear that, as $(\bv,\bw)$ varies, $(\tbv,\tbw)$ runs over all possible ordered pairs of $\tI$-tuples, so we conclude:

\begin{cor} \label{cor:exhaustion}
Every Nakajima quiver variety for the quiver $(\tI,\tH,\tOmega)$ is isomorphic to a connected component of the fixed-point subvariety of a diagram automorphism of a Nakajima quiver variety for the quiver $(I,H,\Omega)$.
\end{cor}

See \S\ref{subsect:lagrangian} for the behaviour of the affine varieties $\fM_0$ and $\fM_1$ and the projective varieties $\fL$ under our isomorphism.

\begin{ex} \label{ex:silly}
Suppose $a$ is the identity. Then the automorphism $\theta$ of $\fM(V,W)$ is just the action of the element $(\sigma_i)\in G_W$, whose $\bn$th power is the identity. In view of Example~\ref{ex:redundant}, the quiver varieties $\fM(\tV^{\tbv},\tW)$ are $\bn$-fold products of quiver varieties for $(I,H,\Omega)$. This case of Theorem~\ref{thm:fixed-points} can be proved in the same way as~\cite[Lemma 3.2]{nak3}. Similar arguments reduce Theorem~\ref{thm:fixed-points} to the case where $\bn$ is the order of $a$, but there seems to be no advantage in this reduction.
\end{ex}

\begin{ex} \label{ex:typea-typed-rule}
In the context of Example~\ref{ex:typea-typed}, Theorem~\ref{thm:fixed-points} (or rather~\eqref{eqn:fixed-points-data}) becomes Theorem~\ref{thm:fixed-points-AtoD}, and Corollary~\ref{cor:exhaustion} becomes the statement that every quiver variety of type D$_{n+1}$ arises as a connected component of the fixed-point subvariety of a diagram involution of a quiver variety of type A$_{2n-1}$. To give a preview of the proof of Theorem~\ref{thm:fixed-points} in this special case, our isomorphism $\Psi$ will send the orbit of the type-$\D_{n+1}$ configuration of linear maps depicted in Example~\ref{ex:typed} to the orbit of the following type-$\A_{2n-1}$ configuration:
\begin{equation*}
\vcenter{
\xymatrix@R=15pt@C=50pt{
\cdots&W_{n-1}\ar@/_/[dd]_(.4){\Gamma_{n-1}}&W_+\oplus W_-\ar@/_/[dd]_(.3){\left(\begin{smallmatrix}2\Gamma_+&0\\0&2\Gamma_-\end{smallmatrix}\right)}&W_{n-1}\ar@/_/[dd]_(.4){\Gamma_{n-1}}&\cdots
\\
\\
\cdots\ar@/^/[r]^-{B_{n-1,n-2}}&V_{n-1}\ar@/^/[l]^-{B_{n-2,n-1}}\ar@/^/[r]^-{\left(\begin{smallmatrix}B_{+,n-1}\\B_{-,n-1}\end{smallmatrix}\right)}\ar@/_/[uu]_(.6){\Delta_{n-1}}&V_{+}\oplus V_-\ar@/^/[l]^-{(B_{n-1,+}\ B_{n-1,-})}\ar@/^/[r]^-{(B_{n-1,+}\ -B_{n-1,-})}\ar@/_/[uu]_(.7){\left(\begin{smallmatrix}\Delta_+&0\\0&\Delta_-\end{smallmatrix}\right)}&V_{n-1}\ar@/^/[l]^-{\left(\begin{smallmatrix}B_{+,n-1}\\-B_{-,n-1}\end{smallmatrix}\right)}\ar@/^/[r]^-{B_{n-2,n-1}}\ar@/_/[uu]_(.6){\Delta_{n-1}}&\cdots\ar@/^/[l]^-{B_{n-1,n-2}}
}
}
\end{equation*}
The reader may find it helpful to verify directly that the equations ~\eqref{eqn:d-quadratic} imply that this configuration satisfies~\eqref{eqn:quadratic}, and to keep this example in mind throughout the following general proof.
\end{ex}
 
\subsection{Proof of Theorem~\ref{thm:fixed-points}}
\label{subsect:proof}

First, fix some $\tbv\in\cD(\bv)$. We will give an explicit definition of a linear map $\psi_{\tbv}:\bM(\tV^{\tbv},\tW)\to\bM(V,W)$ and show that it induces a morphism $\Psi_{\tbv}:\fM(\tV^{\tbv},\tW)\to\fM(V,W)^\theta$. The arguments in \S\ref{subsect:motivation} motivate the definition of $\psi_{\tbv}$, bearing in mind that here we want to go in the reverse direction, from a configuration of maps for the quiver $(\tI,\tH,\tOmega)$ to one for the quiver $(I,H,\Omega)$.

Recall that $\hI$ denotes a fixed set of representatives for the $\langle a\rangle$-orbits in $I$, and that for each $i\in\hI$ we have fixed a direct sum decomposition $V_i=\bigoplus_{\zeta\in\mu_{e_i}}V_{i,\zeta}^{\tbv}$ such that $\dim V_{i,\zeta}^{\tbv}=v_{i,\zeta}$. Because of the assumption~\eqref{eqn:varphi}, we can uniquely extend this to a definition of a direct sum decomposition $V_i=\bigoplus_{\zeta\in\mu_{e_i}}V_{i,\zeta}^{\tbv}$ for every $i\in I$ in such a way that $\varphi_i(V_{i,\zeta}^{\tbv})=V_{a(i),\zeta}^{\tbv}$.

Now define an element $(g_i^{\tbv})\in G_V$ by the following rule: for $i\in\hI$, $g_i^{\tbv}$ acts on the subspace $V_{i,\zeta}^{\tbv}$ of $V_i$ as scalar multiplication by $\zeta$, and for $i\notin\hI$, $g_i^{\tbv}$ is the identity. For any $i\in I$, define
\begin{equation}
\tau_i^{\tbv}=g_i^{\tbv}\varphi_{a^{d_i-1}(i)}\cdots g_{a^2(i)}^{\tbv}\varphi_{a(i)}g_{a(i)}^{\tbv}\varphi_i\in GL(V_i).
\end{equation}
By construction, $V_{i,\zeta}^{\tbv}$ is the $\zeta$-eigenspace of $\tau_i^{\tbv}$. In particular, $(\tau_i^{\tbv})^{e_i}$ is the identity.

If $(h_{i,\zeta})\in G_{\tV^{\tbv}}$, then there is a unique element $(h_i)\in G_V$ such that
\begin{equation}
\begin{split}
h_{a(i)}&=g_{a(i)}^{\tbv}\varphi_i h_i(g_{a(i)}^{\tbv}\varphi_i)^{-1}\ \text{ for all $i\in I$, and}\\
h_i&=\bigoplus_{\zeta\in\mu_{e_i}}h_{i,\zeta}\ \text{ for $i\in\hI$.}
\end{split}
\end{equation} 
(This definition is consistent because the subspaces $V_{i,\zeta}^{\tbv}$ are the eigenspaces of $\tau_i^{\tbv}$.) Clearly this defines a group homomorphism
\begin{equation} \label{eqn:rhov}
\rho_{\tbv}:G_{\tV^{\tbv}}\to G_V:(h_{i,\zeta})\mapsto (h_i).
\end{equation}

Recall that $\hH$ and $\hOmega$ denote the sets of $\langle a\rangle$-orbits in $H$ and $\Omega$ respectively. For any $\hh\in\hOmega$, we have defined $s(\hh)$ and $t(\hh)$ to be the unique elements of $\hI\cap\{s(h)\,|\,h\in\hh\}$ and $\hI\cap\{t(h)\,|\,h\in\hh\}$ respectively. Choose a specific $h_1\in\hh$ such that $s(h_1)=s(\hh)$, and let $f(h_1)$ be the unique element of $\{0,1,\cdots,d_{t(\hh)}-1\}$ such that $t(h_1)=a^{f(h_1)}(t(\hh))$. 
Let $\Omega_1$ be the set of all the representatives $h_1$ as $\hh$ runs over $\hOmega$.

Now we can introduce the linear map
\begin{equation}
\psi_{\tbv}:\bM(\tV^{\tbv},\tW)\to\bM(V,W):(\tB_{\th},\tGamma_{i,\zeta},\tDelta_{i,\zeta})\mapsto(B_h,\Gamma_i,\Delta_i)
\end{equation}
where $B_h,\Gamma_i,\Delta_i$ are defined in terms of $\tB_{\th},\tGamma_{i,\zeta},\tDelta_{i,\zeta}$ as follows. 

If $i\in\hI$, we set
\begin{equation} \label{eqn:gamma-delta-special}
\Gamma_i=e_i\left(\bigoplus_{\zeta\in\mu_{e_i}}\tGamma_{i,\zeta}\right),\
\Delta_i=\left(\bigoplus_{\zeta\in\mu_{e_i}}\tDelta_{i,\zeta}\right).
\end{equation}
Here, the direct sum notation means that $e_i^{-1}\Gamma_i$ and $\Delta_i$ are block-diagonal relative to the decompositions $V_i=\bigoplus V_{i,\zeta}^{\tbv}$ and $W_i=\bigoplus W_{i,\zeta}$, and the diagonal blocks are given by $\tGamma_{i,\zeta}$ or $\tDelta_{i,\zeta}$ respectively. 

We then extend these definitions to define $\Gamma_i$ and $\Delta_i$ for general $i\in I$, in the unique way that makes the following rules hold (cf.~\eqref{eqn:fixed}):
\begin{equation} \label{eqn:gamma-delta-general}
\begin{split}
\Gamma_{a(i)}&=g_{a(i)}^{\tbv}\varphi_i\Gamma_{i}\sigma_i^{-1},\\
\Delta_{a(i)}&=\sigma_i\Delta_i(g_{a(i)}^{\tbv}\varphi_i)^{-1}.
\end{split}
\end{equation}
This definition is consistent because the following equations (cf.~\eqref{eqn:intertwine}) hold for $i\in\hI$ by the above block-diagonality:
\begin{equation} \label{eqn:gamma-delta-intertwine}
\begin{split}
\Gamma_i&=\tau_i^{\tbv}\Gamma_i(\sigma_{a^{d_i-1}(i)}\cdots\sigma_{a(i)}\sigma_i)^{-1},\\
\Delta_i&=(\sigma_{a^{d_i-1}(i)}\cdots\sigma_{a(i)}\sigma_i)\Delta_i(\tau_i^{\tbv})^{-1}.
\end{split}
\end{equation}
Then~\eqref{eqn:gamma-delta-general} implies~\eqref{eqn:gamma-delta-intertwine} for general $i\in I$.

If $h\in\Omega_1$ and $\hh=\langle a\rangle\cdot h$, we set
\begin{equation} \label{eqn:b-special}
\begin{split}
B_h&=e_h (g_{a^{f(h)}(t(\hh))}^{\tbv}\varphi_{a^{f(h)-1}(t(\hh))}\cdots g_{a(t(\hh))}^{\tbv}\varphi_{t(\hh)})\left(\bigoplus_{\substack{\zeta\in\mu_{e_{s(\hh)}}\\\zeta'\in\mu_{e_{t(\hh)}}\\\zeta\overset{\hh}{\longrightarrow}\zeta'}} \tB_{(\hh,\zeta,\zeta')}\right),\\
B_{\overline{h}}&=\left(\bigoplus_{\substack{\zeta\in\mu_{e_{s(\hh)}}\\\zeta'\in\mu_{e_{t(\hh)}}\\\zeta\overset{\hh}{\longrightarrow}\zeta'}} \tB_{\overline{(\hh,\zeta,\zeta')}}\right)(g_{a^{f(h)}(t(\hh))}^{\tbv}\varphi_{a^{f(h)-1}(t(\hh))}\cdots g_{a(t(\hh))}^{\tbv}\varphi_{t(\hh)})^{-1}.
\end{split}
\end{equation}
Here, the direct sum notation in the definition of $B_h$, say, means that we consider the map $V_{s(\hh)}\to V_{t(\hh)}$ whose blocks relative to the decompositions of these vector spaces are given by $\tB_{(\hh,\zeta,\zeta')}$ for those $(\zeta,\zeta')$ such that $\zeta\overset{\hh}{\longrightarrow}\zeta'$, and zero for other $(\zeta,\zeta')$. Note that all the $g_i^{\tbv}$ in the expression $g_{a^{f(h)}(t(\hh))}^{\tbv}\varphi_{a^{f(h)-1}(t(\hh))}\cdots g_{a(t(\hh))}^{\tbv}\varphi_{t(\hh)}$ are equal to the identity; we write it this way to make later computations more transparent.

We then extend these definitions to define $B_h$ for general $h\in H$, in the unique way that makes the following rule hold (cf.~\eqref{eqn:fixed}):
\begin{equation} \label{eqn:b-general}
B_{a(h)}=g_{a(t(h))}^{\tbv}\varphi_{t(h)}B_h(g_{a(s(h))}^{\tbv}\varphi_{s(h)})^{-1}.
\end{equation}
This definition is consistent because the following equation (cf.~\eqref{eqn:intertwine}) holds for $h\in\Omega_1\cup\overline{\Omega_1}$ because of the block form of the definition~\eqref{eqn:b-special}:
\begin{equation} \label{eqn:b-intertwine} 
B_h=(\tau_{t(h)}^{\tbv})^{e_{t(h)}/e_h}B_h(\tau_{s(h)}^{\tbv})^{-e_{s(h)}/e_h}.
\end{equation}
Then~\eqref{eqn:b-general} implies~\eqref{eqn:b-intertwine} for general $h\in H$.

\begin{lem} \label{lem:invariance}
The linear map $\psi_{\tbv}:\bM(\tV^{\tbv},\tW)\to\bM(V,W)$ is $G_{\tV^{\tbv}}\times G_{\tW}$-equivariant, where $G_{\tV^{\tbv}}$ acts on $\bM(V,W)$ via the homomorphism $\rho_{\tbv}$ of~\eqref{eqn:rhov} and $G_{\tW}$ acts on $\bM(V,W)$ via the isomorphism $\rho$ of~\eqref{eqn:gp-isom}. 
\end{lem}

\begin{proof}
This is immediate from the definitions.
\end{proof}

\begin{lem} \label{lem:adhm}
The linear map $\psi_{\tbv}:\bM(\tV^{\tbv},\tW)\to\bM(V,W)$ restricts to a morphism $\psi_{\tbv}:\Lambda(\tV^{\tbv},\tW)\to\Lambda(V,W)$.
\end{lem}

\begin{proof}
We assume that $(\tB_{\th},\tGamma_{i,\zeta},\tDelta_{i,\zeta})\in\Lambda(\tV^{\tbv},\tW)$. That is, the equation analogous to~\eqref{eqn:adhm} is satisfied at every vertex $(i,\zeta)\in\tI$. We must show that, with $(B_h,\Gamma_i,\Delta_i)$ defined as above, the equation~\eqref{eqn:adhm} itself is satisfied for all $i\in I$. 

We first observe that, for all $i\in I$ and $h\in\Omega$,
\begin{equation} \label{eqn:there-and-back}
\begin{split}
\Gamma_{a(i)}\Delta_{a(i)}&=g_{a(i)}^{\tbv}\varphi_i\Gamma_i\Delta_i(g_{a(i)}^{\tbv}\varphi_i)^{-1},\\
B_{\overline{a(h)}}B_{a(h)}&=g_{a(s(h))}^{\tbv}\varphi_{s(h)}B_{\overline{h}}B_h(g_{a(s(h))}^{\tbv}\varphi_{s(h)})^{-1},\\
B_{a(h)}B_{\overline{a(h)}}&=g_{a(t(h))}^{\tbv}\varphi_{t(h)}B_hB_{\overline{h}}(g_{a(t(h))}^{\tbv}\varphi_{t(h)})^{-1}.
\end{split}
\end{equation}
This follows from~\eqref{eqn:gamma-delta-general} and~\eqref{eqn:b-general}. Consequently, it suffices to prove~\eqref{eqn:adhm} for $i\in\hI$.

The right-hand side of~\eqref{eqn:adhm} for $i\in\hI$ is $\Gamma_i\Delta_i=e_i(\bigoplus\tGamma_{i,\zeta}\tDelta_{i,\zeta})$, by the definition~\eqref{eqn:gamma-delta-special}. Using~\eqref{eqn:there-and-back}, we see that each of the sums on the left-hand side of~\eqref{eqn:adhm} commutes with $\tau_i^{\tbv}$, and hence preserves the decomposition $V_i=\bigoplus V_{i,\zeta}^{\tbv}$. Therefore we need only consider the component of each term that maps $V_{i,\zeta}^{\tbv}$ to $V_{i,\zeta}^{\tbv}$, for some fixed $\zeta\in\mu_{e_i}$; we indicate this with the symbol $|_{\zeta}^{\zeta}$. Since we know the equation analogous to~\eqref{eqn:adhm} at the vertex $(i,\zeta)\in\tI$, it suffices to prove that
\begin{equation}\label{eqn:desired}
\begin{split}
\sum_{\substack{h\in\Omega\\s(h)=i}}(B_{\overline{h}}B_h)|_{\zeta}^{\zeta}
&=e_i\sum_{\substack{\hh\in\hOmega\\s(\hh)=i}}\sum_{\substack{\zeta'\in\mu_{e_{t(\hh)}}\\\zeta\overset{\hh}{\longrightarrow}\zeta'}}\tB_{\overline{(\hh,\zeta,\zeta')}}\tB_{(\hh,\zeta,\zeta')},\\
\sum_{\substack{h\in\Omega\\t(h)=i}}(B_hB_{\overline{h}})|_{\zeta}^{\zeta}&=e_i\sum_{\substack{\hh\in\hOmega\\t(\hh)=i}}\sum_{\substack{\zeta'\in\mu_{e_{s(\hh)}}\\\zeta'\overset{\hh}{\longrightarrow}\zeta}}\tB_{(\hh,\zeta',\zeta)}\tB_{\overline{(\hh,\zeta',\zeta)}}.
\end{split}
\end{equation}
Now an arbitrary $h\in\Omega$ such that $s(h)=i$ can be written uniquely as $a^{pd_i}(h_1)$ where $h_1\in\Omega_1$, $s(h_1)=i$, and $0\leq p<d_{h_1}/d_{i}=e_i/e_{h_1}$. Using~\eqref{eqn:there-and-back} and the definition~\eqref{eqn:b-special}, we obtain
\[
\begin{split}
\sum_{\substack{h\in\Omega\\s(h)=i}}(B_{\overline{h}}B_h)|_{\zeta}^{\zeta}
&=\sum_{\substack{h_1\in\Omega_1,0\leq p<e_i/e_{h_1}\\s(h_1)=i}}((\tau_i^{\tbv})^pB_{\overline{h_1}}B_{h_1}(\tau_i^{\tbv})^{-p})|_{\zeta}^{\zeta}\\
&=\sum_{\substack{h_1\in\Omega_1,0\leq p<e_i/e_{h_1}\\s(h_1)=i}}\zeta^{p}\zeta^{-p}\,(B_{\overline{h_1}}B_{h_1})|_{\zeta}^{\zeta}\\
&=\sum_{\substack{h_1\in\Omega_1\\s(h_1)=i}}(e_i/e_{h_1})(B_{\overline{h_1}}B_{h_1})|_{\zeta}^{\zeta}\\
&=e_i\sum_{\substack{\hh\in\hOmega\\s(\hh)=i}}\sum_{\substack{\zeta'\in\mu_{e_{t(\hh)}}\\\zeta\overset{\hh}{\longrightarrow}\zeta'}}\tB_{\overline{(\hh,\zeta,\zeta')}}\tB_{(\hh,\zeta,\zeta')}.
\end{split}
\]
The proof of the other equation in~\eqref{eqn:desired} is similar, using the fact that an arbitrary $h\in\Omega$ such that $t(h)=i$ can be written uniquely as $a^{qd_i-f(h_1)}(h_1)$ where $h_1\in\Omega_1$, $t(h_1)=a^{f(h_1)}(i)$, and $0\leq q<e_i/e_{h_1}$.
\end{proof}

\begin{lem} \label{lem:stability}
With $\psi_{\tbv}:\Lambda(\tV^{\tbv},\tW)\to\Lambda(V,W)$ defined as above, the preimage $\psi_{\tbv}^{-1}(\Lambda(V,W)^s)$ equals $\Lambda(\tV^{\tbv},\tW)^s$.
\end{lem}

\begin{proof}
Let $(\tB_{\th},\tGamma_{i,\zeta},\tDelta_{i,\zeta})\in\Lambda(\tV^{\tbv},\tW)$ and define $(B_h,\Gamma_i,\Delta_i)\in\Lambda(V,W)$ as above. We must show that $(\tB_{\th},\tGamma_{i,\zeta},\tDelta_{i,\zeta})$ is stable if and only if $(B_h,\Gamma_i,\Delta_i)$ is stable. 

First, the `only if' direction: assume that $(\tB_{\th},\tGamma_{i,\zeta},\tDelta_{i,\zeta})$ is stable. Let $V'$ denote the subspace of $V$ generated by $\bigoplus_{i\in I}\img(\Gamma_i)$ under the action of all the maps $B_h$. We need to show that every $V_i$ is contained in $V'$. Because of~\eqref{eqn:gamma-delta-general} and~\eqref{eqn:b-general}, $V'$ is stable under each $g_{a(i)}^{\tbv}\varphi_i$. So it suffices to show that every $V_i$ with $i\in\hI$ is contained in $V'$, and hence it suffices to show that every $V_{i,\zeta}^{\tbv}$ with $(i,\zeta)\in\tI$ is contained in $V'$. This follows from the stability of $(\tB_{\th},\tGamma_{i,\zeta},\tDelta_{i,\zeta})$, using~\eqref{eqn:gamma-delta-special} and~\eqref{eqn:b-special}.

For the `if' direction, we assume that  $(B_h,\Gamma_i,\Delta_i)$ is stable, which by Proposition~\ref{prop:geom-pts}(1) means that there exists some $f\in A(V,W)_n$, $n>0$, such that $f((B_h,\Gamma_i,\Delta_i))\neq 0$. Then $f'=f\circ\psi_{\tbv}$ is a regular function on $\Lambda(\tV^{\tbv},\tW)$ that does not vanish on $(\tB_{\th},\tGamma_{i,\zeta},\tDelta_{i,\zeta})$. Now $f'$ does not belong to $A(\tV^{\tbv},\tW)_n$; instead, the fact that $\psi_{\tbv}$ is $G_{\tV^{\tbv}}$-equivariant and the definition of the homomorphism $\rho_{\tbv}:G_{\tV^{\tbv}}\to G_V$ imply that $f'$ satisfies the rule
\begin{equation} \label{eqn:modified-char}
f'((h_{i,\zeta})\cdot x)=(\prod_{(i,\zeta)\in\hI}\det(h_{i,\zeta})^{d_i})^n f'(x),
\end{equation}
for all $(h_{i,\zeta})\in G_{\tV^{\tbv}}$, $x\in\Lambda(\tV^{\tbv},\tW)$.
However, the proof of Proposition~\ref{prop:geom-pts}(1) goes through unchanged if one modifies the character so that each determinant is replaced with a positive power in this way. Hence, $(\tB_{\th},\tGamma_{i,\zeta},\tDelta_{i,\zeta})$ is stable.
\end{proof}

\begin{lem} \label{lem:induces}
The morphism $\psi_{\tbv}:\Lambda(\tV^{\tbv},\tW)^s\to\Lambda(V,W)^s$ induces a morphism $\Psi_{\tbv}:\fM(\tV^{\tbv},\tW)\to\fM(V,W)^\theta$.
\end{lem}

\begin{proof}
Recall that $\fM(\tV^{\tbv},\tW)$ is a geometric quotient of $\Lambda(\tV^{\tbv},\tW)^s$ by the action of $G_{\tV^{\tbv}}$, and $\fM(V,W)$ is a geometric quotient of $\Lambda(V,W)^s$ by the action of $G_V$. Since $\psi_{\tbv}$ is $G_{\tV^{\tbv}}$-equivariant, it induces a morphism $\Psi_{\tbv}:\fM(\tV^{\tbv},\tW)\to\fM(V,W)$. We must show that the image of $\Psi_{\tbv}$ lies in the $\theta$-fixed subvariety. It suffices to check this on geometric points: in other words, we must show that for stable $(\tB_{\th},\tGamma_{i,\zeta},\tDelta_{i,\zeta})$, the $G_V$-orbit of the resulting stable $x=(B_h,\Gamma_i,\Delta_i)$ is $\theta$-stable. But the criterion for $\theta$-stability of the orbit of $x$ is~\eqref{eqn:fixed}, which holds by definition with $(g_i^x)=(g_i^{\tbv})$.
\end{proof}

\begin{lem} \label{lem:injective}
The morphism $\Psi_{\tbv}:\fM(\tV^{\tbv},\tW)\to\fM(V,W)^\theta$ is injective on geometric points.
\end{lem}

\begin{proof}
Suppose that $x=\psi_{\tbv}(y)$ and $(h_i)\cdot x=\psi_{\tbv}(y')$ where $y,y'\in\Lambda(\tV^{\tbv},\tW)^s$ and $(h_i)\in G_V$. We must show that $y'=(h_{i,\zeta})\cdot y$ for some $(h_{i,\zeta})\in G_{\tV^{\tbv}}$. As we have seen in the proof of Lemma~\ref{lem:induces}, the fact that $x\in\psi_{\tbv}(\Lambda(\tV^{\tbv},\tW)^s)$ implies that $x=(g_i^{\tbv})\cdot\theta(x)$. Likewise, we have 
\[ (h_i)\cdot x = (g_i^{\tbv})\cdot\theta((h_i)\cdot x)=(g_i^{\tbv})\theta((h_i))(g_i^{\tbv})^{-1}\cdot x, \]
which by Proposition~\ref{prop:triv-stab} implies that
\begin{equation} \label{eqn:nice}
h_i=(g_i^{\tbv}\varphi_{a^{-1}(i)})h_{a^{-1}(i)}(g_i^{\tbv}\varphi_{a^{-1}(i)})^{-1}\ \text{ for all }i\in I.
\end{equation} 
This implies that $h_i$ commutes with $\tau_i^{\tbv}$ for all $i\in\hI$, and therefore there is some $(h_{i,\zeta})\in G_{\tV^{\tbv}}$ such that $h_i=\bigoplus_{\zeta\in\mu_{e_i}}h_{i,\zeta}$ for all $i\in\hI$. From~\eqref{eqn:nice} we see that $(h_i)=\rho_{\tbv}((h_{i,\zeta}))$. Since $\psi_{\tbv}$ is $G_{\tV^{\tbv}}$-equivariant, we conclude that $\psi_{\tbv}(y')=\psi_{\tbv}((h_{i,\zeta})\cdot y)$, which suffices because the linear map $\psi_{\tbv}$ is obviously injective.
\end{proof}

\begin{lem} \label{lem:surjective}
Every geometric point of $\fM(V,W)^\theta$ belongs to $\Psi_{\tbv}(\fM(\tV^{\tbv},\tW))$ for a unique $\tbv\in\cD(\bv)$. 
\end{lem}

\begin{proof}
Let $x\in\Lambda(V,W)^s$ be such that its $G_V$-orbit is $\theta$-stable. As seen in \S\ref{subsect:motivation}, there is a unique $(g_i^x)\in G_V$ such that~\eqref{eqn:fixed} holds, and we have resulting definitions of $(\tau_i^x)$ and decompositions $V_i=\bigoplus V_{i,\zeta}^x$. Define $\tbv^x\in\cD(\bv)$ by $v_{i,\zeta}^x=\dim V_{i,\zeta}^x$. We want to show that some element of the $G_V$-orbit of $x$ belongs to $\psi_{\tbv^x}(\Lambda(\tV^{{\tbv}^x},\tW)^s)$. By Lemma~\ref{lem:stability}, it suffices to show that some element of the $G_V$-orbit of $x$ belongs to $\psi_{\tbv^x}(\Lambda(\tV^{{\tbv}^x},\tW))$.

Consider an element of the $G_V$-orbit of $x$, namely $x'=(h_i)\cdot x$ for some $(h_i)\in G_V$. Then $(g_i^{x'})\cdot\theta((h_i)\cdot x)=(h_i)\cdot x$, implying that $(g_i^{x'})=(h_i)(g_i^x)\theta((h_i))^{-1}$. That is,
\begin{equation} \label{eqn:change}
g_i^{x'}=h_i g_i^x \varphi_{a^{-1}(i)}h_{a^{-1}(i)}^{-1}\varphi_{a^{-1}(i)}^{-1}\ \text{ for all }i\in I. 
\end{equation} 
As a consequence,
\begin{equation} \label{eqn:change2}
\tau_i^{x'}=h_i\tau_i^x h_i^{-1}\text{ and }V_{i,\zeta}^{x'}=h_i V_{i,\zeta}^x\ \text{ for all }i,\zeta.
\end{equation}
Therefore we can choose $(h_i)$ in such a way that $V_{i,\zeta}^{x'}=V_{i,\zeta}^{\tbv^x}$ for all $(i,\zeta)\in\hI$, and $g_i^{x'}$ is the identity for all $i\notin\hI$. It follows by definition that $V_{i,\zeta}^{x'}=V_{i,\zeta}^{\tbv^x}$ for all $i,\zeta$, and that for $i\in\hI$, $g_i^{x'}$ acts on $V_{i,\zeta}^{\tbv^x}$ as scalar multiplication by $\zeta$.
We then want to show that this particular $x'=(B_h,\Gamma_i,\Delta_i)$ arises as $\psi_{\tbv^x}(\tB_{\th},\tGamma_{i,\zeta},\tDelta_{i,\zeta})$ for some $(\tB_{\th},\tGamma_{i,\zeta},\tDelta_{i,\zeta})\in\Lambda(\tV^{{\tbv}^x},\tW)$. 

Of course, we must define $\tB_{\th},\tGamma_{i,\zeta},\tDelta_{i,\zeta}$ in the ways that are forced by the equations~\eqref{eqn:gamma-delta-special} and~\eqref{eqn:b-special}. The fact that $\psi_{\tbv^x}(\tB_{\th},\tGamma_{i,\zeta},\tDelta_{i,\zeta})=(B_h,\Gamma_i,\Delta_i)$ then follows from~\eqref{eqn:fixed}. To show that  $(\tB_{\th},\tGamma_{i,\zeta},\tDelta_{i,\zeta})\in\Lambda(\tV^{{\tbv}^x},\tW)$, we need to show that the equation analogous to~\eqref{eqn:adhm} holds for each $(i,\zeta)\in\tI$. We know that $(B_h,\Gamma_i,\Delta_i)\in\Lambda(V,W)$, so we know that the equation~\eqref{eqn:adhm} itself holds for each $i\in I$, and in particular for each $i\in\hI$. Taking the $|_\zeta^\zeta$ component of each side of this equation (in the notation of the proof of Lemma~\ref{lem:adhm}), and using~\eqref{eqn:desired}, we obtain the desired equation.

This proves that the geometric point of $\fM(V,W)^\theta$ given by the $G_V$-orbit of $x$ belongs to $\Psi_{\tbv^x}(\fM(\tV^{\tbv^x},\tW))$. As seen in~\eqref{eqn:change2}, we have $\tbv^{x'}=\tbv^x$ for all $x'$ in the $G_V$-orbit of $x$. On the other hand, we already observed that if $x'\in\psi_{\tbv}(\Lambda(\tV^{\tbv},\tW)^s)$ for some $\tbv\in\cD(\bv)$, then $(g_i^{x'})=(g_i^{\tbv})$, so $\tbv^{x'}=\tbv$. We conclude that the $G_V$-orbit of $x$ cannot belong to any $\Psi_{\tbv}(\fM(\tV^{\tbv},\tW))$ for $\tbv\neq\tbv^x$.
\end{proof}

We can now define the desired morphism $\Psi:\coprod_{\tbv\in\cD(\bv)}\fM(\tV^{\tbv},\tW)\to\fM(V,W)^\theta$ by assembling together the various morphisms $\Psi_{\tbv}$ as $\tbv$ runs over $\cD(\bv)$. The proof of Theorem~\ref{thm:fixed-points} is as follows.

\begin{proof}
By definition of $\Psi$ and Lemma~\ref{lem:invariance}, $\Psi$ is $G_{\tW}$-equivariant. We claim that $\fM(V,W)^\theta$ is the disconnected union of the images $\Psi(\fM(\tV^{\tbv},\tW))$. We already know from Lemma~\ref{lem:surjective} that these images are disjoint and that their union is all of $\fM(V,W)^\theta$. To show that this is a disconnected union, it suffices to show that for $x\in\Lambda(V,W)^s$ such that its $G_V$-orbit is $\theta$-stable, the element $\tbv^x\in\cD(\bv)$ defined in the proof of Lemma~\ref{lem:surjective} depends continuously on $x$. This is clear: each component $v_{i,\zeta}^x$ of $\tbv^x$ depends semicontinuously on $x$ by general results, and the sum $\sum_{\zeta\in\mu_{e_i}}v_{i,\zeta}^x$ is constrained to equal the fixed $v_i$.

So all that remains is to show that $\Psi_{\tbv}:\fM(\tV^{\tbv},\tW)\to\Psi(\fM(\tV^{\tbv},\tW))$ is an isomorphism, when $\tbv\in\cD(\bv)$ is such that $\fM(\tV^{\tbv},\tW)$ is nonempty. By Lemma~\ref{lem:injective}, this is a morphism of complex varieties that is bijective on geometric points. Invoking Zariski's main theorem, we need only show that the codomain $\Psi(\fM(\tV^{\tbv},\tW))$ is normal. But $\Psi(\fM(\tV^{\tbv},\tW))$ is a connected component of $\fM(V,W)^\theta$, so by Proposition~\ref{prop:geom-pts}(3) it is nonsingular. 
\end{proof}

\subsection{Complements to Theorem~\ref{thm:fixed-points}}
\label{subsect:lagrangian}

We need to extend the statement of Theorem~\ref{thm:fixed-points} in two directions. Firstly, we can consider the varieties $\fM_0$, $\fM_1$, $\fL$.

\begin{prop} \label{prop:referee}
Continue the setting of Theorem~\ref{thm:fixed-points}. For every $\tbv\in\cD(\bv)$, there is a $G_{\tW}$-equivariant morphism $\Psi_{\tbv,0}:\fM_0(\tV^{\tbv},\tW)\to\fM_0(V,W)^\theta$ such that the following diagram commutes:
\[
\xymatrix{
\fM(\tV^{\tbv},\tW)\ar[r]^{\Psi_{\tbv}}\ar[d] & \fM(V,W)^\theta\ar[d]\\
\fM_0(\tV^{\tbv},\tW)\ar[r]^{\Psi_{\tbv,0}} & \fM_0(V,W)^\theta
}
\]
Here the vertical map is the canonical morphism $\pi$ in each case. Hence the morphism $\Psi_{\tbv,0}$ restricts to a morphism $\Psi_{\tbv,1}:\fM_1(\tV^{\tbv},\tW)\to\fM_1(V,W)^\theta$. 
\end{prop}

\begin{proof}
By Lemma \ref{lem:invariance} and the $n=0$ case of~\eqref{eqn:modified-char}, the linear map $\psi_{\tbv}$ induces a $G_{\tW}$-equivariant morphism $\Psi_{\tbv,0}:\fM_0(\tV^{\tbv},\tW)\to\fM_0(V,W)$. On geometric points, $\Psi_{\tbv,0}$ takes the closed $G_{\tV^{\tbv}}$-orbit of $(\tB_{\th},\tGamma_{i,\zeta},\tDelta_{i,\zeta})\in\Lambda(\tV^{\tbv},\tW)$ to the unique closed $G_V$-orbit in the closure of the $G_V$-orbit of $(B_h,\Gamma_i,\Delta_i)=\psi_{\tbv}(\tB_{\th},\tGamma_{i,\zeta},\tDelta_{i,\zeta})$. Since the $G_V$-orbit of $(B_h,\Gamma_i,\Delta_i)$ is $\theta$-stable (because~\eqref{eqn:fixed} holds with $(g_i^x)=(g_i^{\tbv})$), the unique closed $G_V$-orbit in its closure is also $\theta$-stable. So the image of $\Psi_{\tbv,0}$ lies in the $\theta$-fixed subvariety $\fM_0(V,W)^\theta$. The rest of the statement is easy.
\end{proof}

Assembling the various $\Psi_{\tbv,0}$, $\Psi_{\tbv,1}$ as $\tbv$ runs over $\cD(\bv)$, we obtain morphisms
\[
\begin{split}
\Psi_0&:\coprod_{\tbv\in\cD(\bv)}\fM_0(\tV^{\tbv},\tW)\to\fM_0(V,W)^\theta,\\
\Psi_1&:\coprod_{\tbv\in\cD(\bv)}\fM_1(\tV^{\tbv},\tW)\to\fM_1(V,W)^\theta.
\end{split}
\]
In contrast to the isomorphism $\Psi$, these morphisms are not injective in general (or even in the special case of Example~\ref{ex:silly}). We do not know whether $\Psi_0$ and $\Psi_1$ are surjective in general; in the special cases relevant for Theorem~\ref{thm:isomorphisms-intro}, we will deduce the surjectivity of $\Psi_1$ from an explicit description of $\fM_1(V,W)^\theta$ (see \S\ref{subsect:the-end}).

\begin{prop} \label{prop:lagrangian}
Continue the setting of Theorem~\ref{thm:fixed-points}.
\begin{enumerate}
\item We have $\Psi(\fL(\tV^{\tbv},\tW))\subseteq\fL(V,W)^\theta$ for all $\tbv\in\cD(\bv)$. 
\item Under the additional assumption that the underlying graph $(\tI,\tH)$ of the split-quotient quiver is of simply-laced finite Dynkin type, the preimage $\Psi^{-1}(\fL(V,W)^\theta)$ is exactly $\coprod_{\tbv\in\cD(\bv)}\fL(\tV^{\tbv},\tW)$. 
\end{enumerate}
\end{prop}

\begin{proof}
Part (1) follows from the commutative diagram in Proposition~\ref{prop:referee}, since $\Psi_{\tbv,0}([0])=[0]$. For part (2) we need to show that if $(\tB_{\th},\tGamma_{i,\zeta},\tDelta_{i,\zeta})\in\Lambda(\tV^{\tbv},\tW)^s$ and $\psi_{\tbv}((\tB_{\th},\tGamma_{i,\zeta},\tDelta_{i,\zeta}))=(B_h,\Gamma_i,\Delta_i)\in\Lambda(V,W)^s$ as above, and the closure of $G_V\cdot(B_h,\Gamma_i,\Delta_i)$ contains $0$, then the closure of $G_{\tV^{\tbv}}\cdot (\tB_{\th},\tGamma_{i,\zeta},\tDelta_{i,\zeta})$ contains $0$. This follows easily from the criterion of Proposition~\ref{prop:nilpotency}: the hypothesis includes the fact that $\Delta_i=0$ for all $i\in I$, from which it is obvious that $\tDelta_{i,\zeta}=0$ for all $(i,\zeta)\in\tI$, since $\tDelta_{i,\zeta}$ is a block in the block-matrix form of $\Delta_i$.
\end{proof}

Secondly, we can extend the $G_{\tW}$-equivariance part of Theorem~\ref{thm:fixed-points}. By definition, the diagonally-embedded scalar subgroup $\C^\times$ of $G_V\times G_W$ acts trivially on $\bM(V,W)$, so the scalar subgroup of $G_W$ acts trivially on $\fM(V,W)$. Hence the action of $G_W^\theta$ on $\fM(V,W)^\theta$ extends to the group
\begin{equation}
G_W^{\theta\sim}=\{\alpha\in G_W\,|\,\theta(\alpha)=\lambda\alpha\text{ for some }\lambda\in\C^\times\}.
\end{equation}
Observe that $\lambda$ here is uniquely determined by $\alpha$, and the map $\alpha\mapsto\lambda$ is a group homomorphism $G_W^{\theta\sim}\to\mu_{\bn}$ whose kernel is $G_W^\theta$. Thus, $G_W^\theta$ is the identity component of $G_W^{\theta\sim}$.

If $\alpha=(\alpha_i)$, the equation $\theta(\alpha)=\lambda\alpha$ means that $\sigma_i\alpha_i\sigma_i^{-1}=\lambda\alpha_{a(i)}$ for all $i\in I$, which implies that $\alpha_i(W_{i,\zeta})=W_{i,\zeta\lambda^{d_i}}$ for all $i\in I$ and $\zeta\in\mu_{e_i}$, and hence that
\begin{equation} \label{eqn:lambda}
w_{i,\zeta}=w_{i,\zeta\lambda^{d_i}}\text{ for all }(i,\zeta)\in\tI.
\end{equation}
In fact, it is easy to see that the image of $G_W^{\theta\sim}\to\mu_{\bn}$ consists precisely of those $\lambda\in\mu_{\bn}$ satisfying~\eqref{eqn:lambda}.

Using the isomorphism $\Psi$ of Theorem~\ref{thm:fixed-points}, we obtain an action of $G_W^{\theta\sim}$ on the disconnected union $\coprod_{\tbv\in\cD(\bv)}\fM(\tV^{\tbv},\tW)$. We already know that $G_W^\theta$ acts via~\eqref{eqn:gp-isom}.

\begin{prop} \label{prop:conn-compts}
Continue the setting of Theorem~\ref{thm:fixed-points}. Let $\alpha\in G_W^{\theta\sim}$, $\theta(\alpha)=\lambda\alpha$.
\begin{enumerate}
\item For any $\tbv=(v_{i,\zeta})\in\cD(\bv)$, $\alpha$ maps $\fM(\tV^{\tbv},\tW)$ to $\fM(\tV^{\tbv'},\tW)$ where $\tbv'\in\cD(\bv)$ is defined by $v_{i,\zeta}'=v_{i,\zeta\lambda^{-d_i}}$ for all $(i,\zeta)\in\tI$.
\item Let $\tbv=(v_{i,\zeta})\in\cD(\bv)$ be such that $v_{i,\zeta}=v_{i,\zeta\lambda^{d_i}}$ for all $(i,\zeta)\in\tI$, so that the action of $\alpha$ stabilizes $\fM(\tV^{\tbv},\tW)$. Assume furthermore that $\alpha^\bn=\mathrm{id}_W$ and that $\lambda^{f(h)}=1$ for all $h\in\Omega_1$, in the notation of \S\ref{subsect:proof}. Then the action of $\alpha$ on $\fM(\tV^{\tbv},\tW)$ is a diagram automorphism as defined in \S\ref{subsect:diag-aut}, relative to some power of the automorphism $\ta$ of $(\tI,\tH,\tOmega)$ introduced in \S\ref{subsect:split-quotient}.
\end{enumerate}
\end{prop} 

\begin{proof}
Take $\tbv=(v_{i,\zeta})\in\cD(\bv)$, and suppose $x\in\Lambda(V,W)^s$ is such that $G_V\cdot x$ belongs to $\Psi(\fM(\tV^{\tbv},\tW))$. As seen in the proof of Lemma~\ref{lem:surjective}, $v_{i,\zeta}=\dim V_{i,\zeta}^x$ for all $(i,\zeta)\in\tI$. By definition we have $g_i^{\alpha\cdot x}=\lambda g_i^x$ and $\tau_i^{\alpha\cdot x}=\lambda^{d_i}\tau_i^x$, which implies that $V_{i,\zeta}^{\alpha\cdot x}=V_{i,\zeta\lambda^{-d_i}}^x$. Part (1) follows.

Under the assumptions of part (2), we let $m$ be such that $\lambda=\eta^m$ where $\eta$ is the primitive $\bn$th root of $1$ used to define $\ta$. Then $\ta^m(i,\zeta)=(i,\zeta\lambda^{d_i})$ for all $(i,\zeta)\in\tI$, and we aim to construct a diagram automorphism of $\fM(\tV^{\tbv},\tW)$ relative to $\ta^m$ (and keeping the same $\bn$) that coincides with the action of $\alpha$. We choose our isomorphisms $\varphi_{i,\zeta}:V_{i,\zeta}\simto V_{i,\zeta\lambda^{d_i}}$ arbitrarily so that the analogue of~\eqref{eqn:varphi} holds. We define our isomorphisms $\sigma_{i,\zeta}:W_{i,\zeta}\simto W_{i,\zeta\lambda^{d_i}}$ by $\sigma_{i,\zeta}=\alpha_i|_{W_{i,\zeta}}$. Then the analogue of~\eqref{eqn:sigma} holds because of the assumption that $\alpha^\bn=\mathrm{id}_W$. Let $\widetilde{\theta}$ denote the resulting diagram automorphism of $\bM(\tV^{\tbv},\tW)$ and $\fM(\tV^{\tbv},\tW)$.

We can define $(g_i)\in G_V$ uniquely so that $g_i|_{V_{i,\zeta}}=\varphi_{i,\zeta}$ for all $(i,\zeta)\in\tI$ and
\begin{equation}
g_{a(i)}=\lambda^{-1}(g_{a(i)}^{\tbv}\varphi_i)g_i(g_{a(i)}^{\tbv}\varphi_i)^{-1}\text{ for all }i\in I.
\end{equation} 
It is straightforward to check from the definition of $\psi_{\tbv}$ given in \S\ref{subsect:proof} that
\begin{equation}
\psi_{\tbv}(\widetilde{\theta}(\tB_{\th},\tGamma_{i,\zeta},\tDelta_{i,\zeta}))=\alpha\cdot((g_i)\cdot\psi_{\tbv}(\tB_{\th},\tGamma_{i,\zeta},\tDelta_{i,\zeta}))
\end{equation}
for all $(\tB_{\th},\tGamma_{i,\zeta},\tDelta_{i,\zeta})\in\bM(\tV^{\tbv},\tW)$. (The assumption that $\lambda^{f(h)}=1$ for all $h\in\Omega_1$ is needed because of the form of the definition~\eqref{eqn:b-special}.) Hence the action of $\alpha$ on $\fM(\tV^{\tbv},\tW)$ coincides with $\widetilde{\theta}$. 
\end{proof}

\subsection{An alternative proof}
\label{subsect:nakajima}

We conclude this section by outlining an alternative proof of Theorem~\ref{thm:fixed-points}, based on a suggestion of Nakajima, which applies when the underlying graph $(I,H)$ is of simply-laced affine Dynkin type, $a$ is an admissible automorphism of order $2$ or $3$ that fixes at least one vertex, and $\bn$ is its order. These assumptions are equivalent to saying that $(I,H,a,\bn)$ arises from a pair $(\Gamma,\Gamma')$ of subgroups of $SL_2(\C)$ as in \S\ref{subsect:McKay}. 

In particular, one could take $(I,H)$ to be of type $\widetilde{\A}_{2n-1}$ (for $n\geq 2$) and $a$ to be the involution that fixes the extra vertex $0$ and is the usual diagram involution of the sub-diagram of type $\A_{2n-1}$, as in Example~\ref{ex:affinetypea-typed}. Then the split-quotient quiver is of type $\widetilde{\D}_{n+2}$. From this case of Theorem~\ref{thm:fixed-points}, one recovers Theorem~\ref{thm:fixed-points-AtoD} simply by appending zero vector spaces at the extra vertices. A similar argument applies to the other finite Dynkin cases of Theorem~\ref{thm:fixed-points}.

Let $\Gamma$ be any finite subgroup of $SL_2(\C)$. As in~\cite[Section 2]{vv}, one can define a quiver variety $\fM^\Gamma(\cV,\cW)$ for any pair of finite-dimensional representations $\cV$ and $\cW$ of $\Gamma$. This is essentially a special case (indeed, the motivating case; see~\cite{nak1}) of the definition of Nakajima recalled in \S\ref{subsect:quiver-vars}, where $(I,H,\Omega)$ is an orientation of the affine Dynkin diagram $\widetilde{\Delta}_\Gamma$, and $V_i$ and $W_i$ are the multiplicity spaces in $\cV$ and $\cW$ respectively of the irreducible representation of $\Gamma$ labelled by $i$. We say ``essentially'' because there are two caveats. One is that the $\Gamma=\{1\}$ case was excluded by our former conventions, since it corresponds to the graph with a single vertex and a loop (the so-called $\widetilde{\A}_0$ diagram). The other is that one must make some changes of sign in the coordinates, depending on the orientation $\Omega$, to match up the equation~\eqref{eqn:adhm} with the equations in~\cite{vv}; these sign changes have no effect on the quiver variety.

For any representation $\cW$ of $\Gamma$ and any $n\in\N$, one has a natural action of $\Gamma$ on $\fM^{\{1\}}(\C^n,\cW)$. The key fact is the following description of the $\Gamma$-fixed subvariety: 
\begin{equation} \label{eqn:nakajima}
\fM^{\{1\}}(\C^n,\cW)^\Gamma\cong\coprod_{\substack{\cV\\\dim\cV=n}}\fM^\Gamma(\cV,\cW),
\end{equation} 
where on the right-hand side the disconnected union is over representatives of the isomorphism classes of $n$-dimensional representations of $\Gamma$. Note the formal similarity between~\eqref{eqn:nakajima} and Theorem~\ref{thm:fixed-points}. One can give a proof of~\eqref{eqn:nakajima} along very similar lines to the above proof of Theorem~\ref{thm:fixed-points}, defining a morphism from the right-hand side to the left-hand side and checking injectivity and surjectivity; the argument is perhaps most transparent if one uses the interpretations of both sides in terms of moduli spaces of torsion-free sheaves on $\mathbb{P}^2$ as in~\cite[Theorem 1]{vv}.

Granting~\eqref{eqn:nakajima}, the proof of Theorem~\ref{thm:fixed-points} in the cases arising from a pair $(\Gamma,\Gamma')$ is as follows. First suppose that $(I,H)=\widetilde{\Delta}_\Gamma$ with the automorphism $a$ as in (1) of \S\ref{subsect:McKay}, so that we are considering an automorphism $\theta$ of a quiver variety $\fM^\Gamma(\cV,\cW)$. The existence of the isomorphisms $\varphi_i$ as in \S\ref{subsect:diag-aut} amounts to the existence of a representation $\cV'$ of $\Gamma'$ such that $\Res_\Gamma^{\Gamma'}(\cV')=\cV$, and the choice of the isomorphisms $\sigma_i$ amounts to the choice of a representation $\cW'$ of $\Gamma'$ such that $\Res_\Gamma^{\Gamma'}(\cW')=\cW$. Moreover, one can check that the resulting automorphism $\theta$ of $\fM^\Gamma(\cV,\cW)$ corresponds under~\eqref{eqn:nakajima} to the action of a generator of $\Gamma'/\Gamma$ on the appropriate connected component of $\fM^{\{1\}}(\C^{\dim\cV},\cW')^{\Gamma}$. Taking the trivial equation $\fM^{\{1\}}(\C^{\dim\cV},\cW')^{\Gamma'}=(\fM^{\{1\}}(\C^{\dim\cV},\cW')^\Gamma)^{\Gamma'/\Gamma}$ and using~\eqref{eqn:nakajima}, one obtains
\begin{equation} \label{eqn:fixed-nakajima}
\fM^{\Gamma}(\cV,\cW)^{\theta}\cong\coprod_{\substack{\cV'\\ \Res_\Gamma^{\Gamma'}(\cV')=\cV}}\fM^{\Gamma'}(\cV',\cW'),
\end{equation}
where on the right-hand side the disconnected union is over representatives of the isomorphism classes of representations $\cV'$ of $\Gamma'$ such that $\Res_\Gamma^{\Gamma'}(\cV')=\cV$. This is exactly what Theorem~\ref{thm:fixed-points} says in this case.

Now suppose that $(I,H)=\widetilde{\Delta}_{\Gamma'}$ with the automorphism $a'$ as in (2) of \S\ref{subsect:McKay}, so that we are considering an automorphism $\theta'$ of a quiver variety $\fM^{\Gamma'}(\cV',\cW')$. Let $\cV=\Res_\Gamma^{\Gamma'}(\cV')$ and $\cW=\Res_\Gamma^{\Gamma'}(\cW')$; then we have an automorphism $\theta$ of $\fM^\Gamma(\cV,\cW)$ such that $\fM^{\Gamma'}(\cV',\cW')$ is isomorphic to a connected component of $\fM^\Gamma(\cV,\cW)^\theta$, as seen in~\eqref{eqn:fixed-nakajima}. The existence of the isomorphisms $\varphi_i$ in the definition of $\theta'$ amounts to the existence of a representation $\overline{\cV}$ of $\Gamma$ such that $\Ind_\Gamma^{\Gamma'}(\overline{\cV})=\cV'$, and the choice of the isomorphisms $\sigma_i$ amounts to the choice of a representation $\overline{\cW}$ of $\Gamma$ such that $\Ind_\Gamma^{\Gamma'}(\overline{\cW})=\cW'$. The representations of $\Gamma$ on $\cV=\Res_\Gamma^{\Gamma'}(\Ind_\Gamma^{\Gamma'}(\overline{\cV}))$ and $\cW=\Res_\Gamma^{\Gamma'}(\Ind_\Gamma^{\Gamma'}(\overline{\cW}))$ extend to $\Gamma\times\widehat{\Gamma'/\Gamma}$, and hence we get an action of $\widehat{\Gamma'/\Gamma}$ on $\fM^\Gamma(\cV,\cW)$ that commutes with the automorphism $\theta$. One can check that the automorphism $\theta'$ of $\fM^{\Gamma'}(\cV',\cW')$ corresponds under~\eqref{eqn:fixed-nakajima} to the action of a generator of $\widehat{\Gamma'/\Gamma}$ on the appropriate connected component of $\fM^{\Gamma}(\cV,\cW)^{\theta}$. By the easy special case of Theorem~\ref{thm:fixed-points} mentioned in Example~\ref{ex:silly}, we have
\begin{equation}
\fM^\Gamma(\cV,\cW)^{\widehat{\Gamma'/\Gamma}}\cong\coprod_{\substack{(\overline{\cV}_x)_{x\in\Gamma'/\Gamma}\\ \bigoplus_x \overline{\cV}_x=\cV}}\,\prod_{x\in\Gamma'/\Gamma}\fM^\Gamma(\overline{\cV}_x,\overline{\cW}),
\end{equation}
and the action of $\theta$ on the left-hand side corresponds to a cyclic place permutation of the tuples $(\overline{\cV}_x)$ on the right-hand side. Hence we get
\begin{equation}
(\fM^\Gamma(\cV,\cW)^{\theta})^{\widehat{\Gamma'/\Gamma}}\cong\coprod_{\substack{\overline{\cV}_1\\ \Res_\Gamma^{\Gamma'}(\Ind_\Gamma^{\Gamma'}(\overline{\cV}_1))=\cV}}\fM^\Gamma(\overline{\cV}_1,\overline{\cW}),
\end{equation}
from which we deduce
\begin{equation}
\fM^{\Gamma'}(\cV',\cW')^{\theta'}\cong\coprod_{\substack{\overline{\cV}_1\\ \Ind_\Gamma^{\Gamma'}(\overline{\cV}_1)=\cV'}}\fM^\Gamma(\overline{\cV}_1,\overline{\cW}).
\end{equation}
This is exactly what Theorem~\ref{thm:fixed-points} says in this case.


\section{The diagram involution in Maffei's description}
\label{sect:maffei}


In this section, we assume we are in the setting of Example~\ref{ex:typea-inv}, where $(I,H,\Omega)$ is a quiver of type $\A_{2n-1}$, $V_i=V_{2n-i}$ and $W_i=W_{2n-i}$ for all $i$, $\varphi_i$ is the identity for all $i$, and $\sigma_i$ is the identity for all $i\neq n$. We have a diagram involution $\theta$ of $\fM^{\A_{2n-1}}(V,W)$, depending on the choice of an involution $\sigma_n:W_n\to W_n$.

Maffei has shown in~\cite[Theorem 8]{maffei} that $\fM^{\A_{2n-1}}(V,W)$ is isomorphic to a resolution of the intersection of the Slodowy slice to a particular nilpotent orbit of $\fsl_D$ with the closure of another nilpotent orbit of $\fsl_D$, where $D=\sum_i iw_i$. This raises an obvious question: what involution of this Slodowy variety corresponds to the diagram involution $\theta$? We will answer this question, in a very special case, in Theorem~\ref{thm:involutions} and Corollary~\ref{cor:involutions}.

\subsection{Maffei's isomorphism in the small case}
\label{subsect:setup}

For the remainder of the section, we take $\bw$ to be of a special form: we assume that there is some $k\in\{1,2,\cdots,n\}$ such that
\begin{equation} \label{eqn:small}
w_i=\begin{cases}
0,&\text{ if $i\notin\{k,2n-k\}$,}\\
1,&\text{ if $k<n$ and $i\in\{k,2n-k\}$,}\\
2,&\text{ if $k=n$ and $i=n$.}
\end{cases}
\end{equation}
This assumption implies that $D=2n$, so that the Lie algebra arising in Maffei's isomorphism has the same type as the quiver, namely $\A_{2n-1}$. 

Note that the group $G_W$ is now either $GL(W_k)\times GL(W_{n-k})\cong \C^\times\times\C^\times$ (if $k<n$) or $GL(W_n)\cong GL_2$ (if $k=n$). In the $k<n$ case, $W_n=0$ so there is no choice for $\sigma_n$; in the $k=n$ case, there are (up to $G_W$-conjugacy) three choices for $\sigma_n$, namely $\sigma_n=\mathrm{id}_{W_n}$, $\sigma_n=-\mathrm{id}_{W_n}$, and $\sigma_n$ having eigenvalues $1$ and $-1$; the respective values of $(w_+,w_-)$ are $(2,0)$, $(0,2)$ and $(1,1)$. Also note that the defining equations~\eqref{eqn:quadratic} of $\Lambda^{\A_{2n-1}}(V,W)$ now take a simpler form, since $\Gamma_i=\Delta_i=0$ for $i\notin\{k,2n-k\}$.

In order to define $\theta$, we have already assumed that
\begin{equation} \label{eqn:v-symmetry}
v_{2n-i}=v_i\ \text{ for all }i,
\end{equation}
so $v_1,\cdots,v_n$ are enough to specify $\bv$. Define integers $s_1,s_2,\cdots,s_{n}$ by
\[
s_i=
\begin{cases}
1-v_1,&\text{ if $i=1$,}\\
1-v_i+v_{i-1},&\text{ if $2\leq i\leq k$,}\\
-v_i+v_{i-1},&\text{ if $k+1\leq i\leq n$.}
\end{cases}
\]
Let $\ell=|\{i\,|\,s_i\neq 0\}|$. By~\cite[Lemma 7]{maffei}, the quiver variety $\fM^{\A_{2n-1}}(V,W)$ is nonempty if and only if
\begin{equation} \label{eqn:nonempty}
s_i\in\{-1,0,1\}\text{ for all }i,\text{ and }\ell\leq k.
\end{equation}
We assume this henceforth.

\begin{rmk}
The assumptions~\eqref{eqn:small},~\eqref{eqn:v-symmetry}, and~\eqref{eqn:nonempty} together mean that the quiver variety $\fM^{\A_{2n-1}}(V,W)$ corresponds, in the setting of Remark~\ref{rmk:weight-space}, to a non-vanishing weight space of a small self-dual representation of $\fsl_{2n}$. Namely, using the usual identification of the root lattice of $\fsl_{2n}$ with 
\[ \{(a_1,\cdots,a_{2n})\in\Z^{2n}\,|\,a_1+\cdots+a_{2n}=0\}, \] 
it corresponds to the $(s_1,\cdots,s_n,-s_n,\cdots,-s_1)$-weight space of the irreducible representation with highest weight $(1,\cdots,1,0,\cdots,0,-1,\cdots,-1)$ ($k$ ones). 
\end{rmk} 

In Maffei's construction, the $2n$-dimensional vector space on which the Lie algebra $\fsl_{2n}$ is based is
\[
\tV_0:=\begin{cases}
\underbrace{W_k\oplus\cdots\oplus W_k}_{\text{($k$ copies)}}\oplus \underbrace{W_{2n-k}\oplus\cdots\oplus W_{2n-k}}_{\text{($2n-k$ copies)}},&\text{ if $k<n$,}\\
\underbrace{W_n\oplus\cdots\oplus W_n}_{\text{($n$ copies)}},&\text{ if $k=n$.}
\end{cases}
\]
The relevant Slodowy slice is defined using an $\fsl_2$-triple $\{E_0,H_0,F_0\}\subset\fsl(\tV_0)$, specified as follows.

Given any associative $\C$-algebra $R$ with unit element $1_R$ and any nonnegative integer $m$, we define three $(m+1)\times(m+1)$ matrices with entries in $R$:
\[
\begin{split}
E^{[m+1]}&=\begin{pmatrix}0&1_R&0&\cdots&0&0\\
0&0&1_R&\cdots&0&0\\
0&0&0&\cdots&0&0\\
\vdots&\vdots&\vdots&\ddots&\vdots&\vdots\\
0&0&0&\cdots&0&1_R\\
0&0&0&\cdots&0&0
\end{pmatrix},\\
H^{[m+1]}&=
\begin{pmatrix}m\cdot 1_R&0&\cdots&0&0\\
0&(m-2)\cdot 1_R&\cdots&0&0\\
\vdots&\vdots&\ddots&\vdots&\vdots\\
0&0&\cdots&(-m+2)\cdot 1_R&0\\
0&0&\cdots&0&-m\cdot 1_R
\end{pmatrix},\\
F^{[m+1]}&=\begin{pmatrix}0&0&0&\cdots&0&0\\
m\cdot 1_R&0&0&\cdots&0&0\\
0&2(m-1)\cdot 1_R&0&\cdots&0&0\\
\vdots&\vdots&\vdots&\ddots&\vdots&\vdots\\
0&0&0&\cdots&0&0\\
0&0&0&\cdots&m\cdot 1_R&0
\end{pmatrix}.
\end{split}
\]
where the scalars on the sub-diagonal in $F^{[m+1]}$ are the numbers $i(m+1-i)$ for $1\leq i\leq m$. It is well known that these matrices satisfy the $\fsl_2$ commutation relations.

When $k=n$, we take $R=\End(W_n)$ and let $(E_0,H_0,F_0)=(E^{[n]},H^{[n]},F^{[n]})$. When $k<n$, we take $R=\End(W_k)=\End(W_{2n-k})=\C$ and let
\[
(E_0,H_0,F_0)=\left(\begin{pmatrix}E^{[k]}&0\\0&E^{[2n-k]}\end{pmatrix},\begin{pmatrix}H^{[k]}&0\\0&H^{[2n-k]}\end{pmatrix},\begin{pmatrix}F^{[k]}&0\\0&F^{[2n-k]}\end{pmatrix}\right),
\]
where the diagonal blocks correspond respectively to the $W_k$ summands and to the $W_{2n-k}$ summands of $\tV_0$.

Clearly $E_0$ has Jordan type $(2n-k,k)$ in either case. The Slodowy slice associated to the $\fsl_2$-triple $(E_0,H_0,F_0)$ in the nilpotent cone of $\fsl(\tV_0)$ is by definition
\begin{equation}
\cS^{\A_{2n-1}}_{(2n-k,k)}=\{X\in\fsl(\tV_0)\,|\,[X-E_0,F_0]=0,\,X\text{ nilpotent}\}.
\end{equation}
(We will only consider this nilpotent part, never the full Slodowy slice.) We refer to $E_0$ as the base-point of the slice.
Since $\ell\leq k$, $E_0$ belongs to the closure $\overline{\cO}^{\A_{2n-1}}_{(2n-\ell,\ell)}$ of the orbit of nilpotent elements of $\fsl(\tV_0)$ with Jordan type $(2n-\ell,\ell)$. The intersection $\cS^{\A_{2n-1}}_{(2n-k,k)}\cap\overline{\cO}^{\A_{2n-1}}_{(2n-\ell,\ell)}$ is a transverse slice to the orbit of $E_0$ in $\overline{\cO}^{\A_{2n-1}}_{(2n-\ell,\ell)}$ by~\cite[Section 7.4]{slod}.

Let $\cF$ be the variety of partial flags $0=U_0\subseteq U_1\subseteq\cdots\subseteq U_{2n}=\tV_0$ such that 
\begin{equation} \label{eqn:partial-flag}
\dim U_i-\dim U_{i-1}=
\begin{cases}
s_i+1,&\text{ if $1\leq i\leq n$,}\\
-s_{2n+1-i}+1,&\text{ if $n+1\leq i\leq 2n$.}
\end{cases} 
\end{equation}
By~\eqref{eqn:nonempty}, the right-hand side of~\eqref{eqn:partial-flag} takes the value $2$ on $\ell$ occasions, $1$ on $2n-2\ell$ occasions, and $0$ on $\ell$ occasions. We have the usual expression for the cotangent bundle of $\cF$:
\begin{equation}
T^*\cF=\{(X,(U_i))\in\fsl(\tV_0)\times\cF\,|\,X(U_i)\subseteq U_{i-1}\text{ for all }i\}.
\end{equation}
Projection onto the first factor gives a resolution of singularities $\mu:T^*\cF\to\overline{\cO}^{\A_{2n-1}}_{(2n-\ell,\ell)}$. We can then consider the Slodowy variety $\mu^{-1}(\cS^{\A_{2n-1}}_{(2n-k,k)}\cap\overline{\cO}^{\A_{2n-1}}_{(2n-\ell,\ell)})$.

Note that the common centralizer $Z_{GL(\tV_0)}(E_0,H_0,F_0)$ of $E_0,H_0,F_0$ in $GL(\tV_0)$ can be identified with $G_W$, acting diagonally on the copies of $W_k$ and $W_{2n-k}$ (in the $k<n$ case) or on the copies of $W_n$ (in the $k=n$ case). By definition, $Z_{GL(\tV_0)}(E_0,H_0,F_0)$ acts on $\cS^{\A_{2n-1}}_{(2n-k,k)}\cap\overline{\cO}^{\A_{2n-1}}_{(2n-\ell,\ell)}$ and $\mu^{-1}(\cS^{\A_{2n-1}}_{(2n-k,k)}\cap\overline{\cO}^{\A_{2n-1}}_{(2n-\ell,\ell)})$, and the map $\mu$ is $Z_{GL(\tV_0)}(E_0,H_0,F_0)$-equivariant.

\begin{prop}[A special case of~{\cite[Theorem 8]{maffei}}] \label{prop:maffei1}
There are isomorphisms $\Phi$ and $\Phi_1$ making the following diagram commute:
\[
\xymatrix{
\fM^{\A_{2n-1}}(V,W)\ar[r]^-{\Phi}_-{\sim}\ar[d]_-{\pi}&\mu^{-1}(\cS^{\A_{2n-1}}_{(2n-k,k)}\cap\overline{\cO}^{\A_{2n-1}}_{(2n-\ell,\ell)})\ar[d]^-{\mu}\\
\fM_1^{\A_{2n-1}}(V,W)\ar[r]^-{\Phi_1}_-{\sim}&\cS^{\A_{2n-1}}_{(2n-k,k)}\cap\overline{\cO}^{\A_{2n-1}}_{(2n-\ell,\ell)}
}
\]
Moreover, $\Phi_1([0])=E_0$, and the action of $G_W$ on the domains of $\Phi$ and $\Phi_1$ corresponds to the action of $Z_{GL(\tV_0)}(E_0,H_0,F_0)$ on their codomains.
\end{prop}
\noindent
The statement about group actions is not explicitly mentioned in~\cite{maffei}, but is evident from the definition of the isomorphisms $\Phi$ and $\Phi_1$ (which Maffei calls $\widetilde{\varphi}$ and $\varphi_1$ respectively). We will recall the definition of $\Phi_1$ in our case below.

\begin{rmk} \label{rmk:sign}
There is a minor caveat: in~\cite[Definition 4]{maffei}, Maffei implicitly uses the orientation for the quiver of type A$_{2n-1}$ consisting of the edges from $i$ to $i+1$ for $1\leq i\leq 2n-2$. But this makes no difference to the quiver varieties. Indeed, one can translate from his orientation to ours simply by introducing some minus signs in the coordinates of $\bM^{\A_{2n-1}}(V,W)$: specifically, we will let our $B_{n,n+1}$ be his $-B_n$, our $B_{n+1,n+2}$ be his $-B_{n+1}$, and so on.
\end{rmk}

\subsection{Statement of the result} 
\label{subsect:form}

The main result of this section asserts that under Maffei's isomorphism $\Phi_1$ the involution $\theta$ of $\fM_1^{\A_{2n-1}}(V,W)$ corresponds to an involution $\Theta$ of $\cS^{\A_{2n-1}}_{(2n-k,k)}\cap\overline{\cO}^{\A_{2n-1}}_{(2n-\ell,\ell)}$ that arises naturally from a symmetric or skew-symmetric nondegenerate bilinear form $(\cdot,\cdot)$ on $\tV_0$. We now define this form.

First suppose that $k<n$. Fix a nondegenerate symmetric bilinear form $(\cdot,\cdot)_k=(\cdot,\cdot)_{2n-k}$ on $W_k=W_{2n-k}$ (this is unique up to scalar, since $\dim W_k=1$). Define the nondegenerate bilinear form $(\cdot,\cdot)$ on $\tV_0$ by:
\begin{equation}
\begin{split}
&((w_k^{(1)},\cdots,w_k^{(k)},w_{2n-k}^{(1)},\cdots,w_{2n-k}^{(2n-k)}),(y_k^{(1)},\cdots,y_k^{(k)},y_{2n-k}^{(1)},\cdots,y_{2n-k}^{(2n-k)}))\\
&=
\sum_{i=1}^{k}(-1)^{i-1}(w_k^{(i)},y_k^{(k+1-i)})_k+\sum_{i=1}^{2n-k}(-1)^{n-k+i}(w_{2n-k}^{(i)},y_{2n-k}^{(2n-k+1-i)})_{2n-k}.
\end{split}
\end{equation}
Here $w_j^{(i)},y_j^{(i)}\in W_j$ for $j\in\{k,2n-k\}$ and $1\leq i\leq j$.
The form $(\cdot,\cdot)$ is clearly symmetric if $k$ is odd and skew-symmetric if $k$ is even. 

Now suppose that $k=n$. Fix a nondegenerate skew-symmetric bilinear form $\langle\cdot,\cdot\rangle_n$ on $W_n$ (this is unique up to scalar, since $\dim W_n=2$). Let $(\cdot,\cdot)_n$ be the nondegenerate bilinear form on $W_n$ defined by $(w,y)_n=\langle w,\sigma_n(y)\rangle_n$. Note that $(\cdot,\cdot)_n$ is symmetric if $\sigma_n$ has eigenvalues $1$ and $-1$, and is skew-symmetric if $\sigma_n=\pm\mathrm{id}_{W_n}$.
Define the nondegenerate bilinear form $(\cdot,\cdot)$ on $\tV_0$ by:
\begin{equation}
((w_n^{(1)},\cdots,w_n^{(n)}),(y_1^{(1)},\cdots,y_n^{(n)}))=
\sum_{i=1}^{n}(-1)^{i-1}(w_n^{(i)},y_n^{(n+1-i)})_n.
\end{equation}   
Here $w_n^{(i)},y_n^{(i)}\in W_n$ for $1\leq i\leq n$.
The form $(\cdot,\cdot)$ is the same type as $(\cdot,\cdot)_n$ if $n$ is odd (i.e.\ they are either both symmetric or both skew-symmetric), and the opposite type to $(\cdot,\cdot)_n$ if $n$ is even.

In either the $k<n$ or the $k=n$ case, we define a Lie algebra involution $\Theta$ of $\fsl(\tV_0)$ as the negative transpose map with respect to the form $(\cdot,\cdot)$. In other words, it is defined by the property that
\begin{equation}
(Xw,y)+(w,\Theta(X)y)=0,\text{ for all }X\in\fsl(\tV_0),\ w,y\in \tV_0.
\end{equation} 
The fixed-point subalgebra $\fsl(\tV_0)^\Theta$ is isomorphic to $\fso_{2n}$ if $(\cdot,\cdot)$ is symmetric, or to $\fsp_{2n}$ if $(\cdot,\cdot)$ is skew-symmetric. It is immediate from the definitions that $E_0,H_0,F_0\in\fsl(\tV_0)^\Theta$, so $\Theta$ preserves the Slodowy slice $\cS^{\A_{2n-1}}_{(2n-k,k)}$. It is well known that $\Theta$ preserves each nilpotent orbit of $\fsl(\tV_0)$, so we have an induced involution $\Theta$ of the variety $\cS^{\A_{2n-1}}_{(2n-k,k)}\cap\overline{\cO}^{\A_{2n-1}}_{(2n-\ell,\ell)}$.

\begin{thm} \label{thm:involutions}
Continue the above assumptions, in particular~\eqref{eqn:small},~\eqref{eqn:v-symmetry},~\eqref{eqn:nonempty}.
The involution $\Theta$ of $\cS^{\A_{2n-1}}_{(2n-k,k)}\cap\overline{\cO}^{\A_{2n-1}}_{(2n-\ell,\ell)}$ corresponds, under the isomorphism $\Phi_1$ of Proposition~\ref{prop:maffei1}, to the diagram involution $\theta$ of $\fM_1(V,W)$.
\end{thm}

We will prove the $k=n$ case of Theorem~\ref{thm:involutions} in \S\ref{subsect:proofk=n} and the $k<n$ case in \S\ref{subsect:proofk<n}. We will then deduce an analogous result for the isomorphism $\Phi$ in Corollary~\ref{cor:involutions}.

As a first indication of how vital assumption~\eqref{eqn:small} is to Theorem~\ref{thm:involutions}, we note a more obvious relationship between $\theta$ and the form $(\cdot,\cdot)$. Recall that we have identified $G_W$ with $Z_{GL(\tV_0)}(E_0,H_0,F_0)$. Also recall the subgroup $G_W^{\theta\sim}\subset G_W$ introduced in \S\ref{subsect:lagrangian} whose identity component is $G_W^\theta$, and let $G_{(\cdot,\cdot)}$ be the subgroup of $GL(\tV_0)$ consisting of elements $g$ for which there exists some $\lambda\in\C^\times$ such that
\begin{equation}
(gw,gy)=\lambda(w,y)\text{ for all }w,y\in\tV_0.
\end{equation}
Thus, $G_{(\cdot,\cdot)}$ is the group usually denoted $GO(\tV_0)$ if $(\cdot,\cdot)$ is symmetric or $GSp(\tV_0)$ if $(\cdot,\cdot)$ is skew-symmetric. 

\begin{prop} \label{prop:reductive}
Under the assumption~\eqref{eqn:small}, $G_W^{\theta\sim}=Z_{G_{(\cdot,\cdot)}}(E_0,H_0,F_0)$. If $k<n$, this group has two connected components and its identity component is isomorphic to $\C^\times$. If $k=n$ and $\sigma_n$ has eigenvalues $1$ and $-1$, this group has two connected components and its identity component is isomorphic to $\C^\times\times\C^\times$. If $k=n$ and $\sigma_n=\pm\mathrm{id}_{W_n}$, this group is isomorphic to $GL_2$. 
\end{prop}

\begin{proof}
If $k<n$, $G_W\cong\C^\times\times\C^\times$ with $\theta$ interchanging the two factors. For $(\alpha,\beta)\in G_W$, we have $(\alpha,\beta)\in G_W^{\theta\sim}$ if and only if $\alpha=\pm\beta$, which is the same as the condition for $(\alpha,\beta)$, thought of as an element of $GL(\tV_0)$, to belong to $G_{(\cdot,\cdot)}$.

If $k=n$, $G_W=GL(W_n)\cong GL_2$ with $\theta$ being conjugation by $\sigma_n$. The condition for $\alpha\in G_W$ to belong to $G_{(\cdot,\cdot)}$ is the same as the condition for it to belong to the analogous subgroup $G_{(\cdot,\cdot)_n}$ of $GL(W_n)$. If $\sigma_n$ has eigenvalues $1$ and $-1$, then both $G_W^{\theta\sim}$ and $G_{(\cdot,\cdot)_n}$ equal the subgroup of $GL(W_n)$ stabilizing the decomposition of $W_n$ into eigenspaces of $\sigma_n$. If $\sigma_n=\pm\mathrm{id}_{W_n}$, then both $G_W^{\theta\sim}$ and $G_{(\cdot,\cdot)_n}=G_{\langle\cdot,\cdot\rangle_n}$ equal $GL(W_n)$.
\end{proof}

\subsection{Maffei's isomorphism $\Phi_1$ in the $k=n$ case}
\label{subsect:recursion}

For this subsection and the next, we take $k=n$. We need to recall, and slightly rephrase, Maffei's definition of the isomorphism $\Phi_1$ in our special case. 

It is helpful to state the recursive part of the definition in an abstract setting. Let $R$ be any associative $\C$-algebra with unit element $1_R$. Given a sequence $r_1,r_2,r_3,\cdots$ of elements of $R$, we can define two sequences of matrices with entries in $R$, denoted $M_j$ and $N_j$ for $j\in\N$, where $M_j$ and $N_j$ depend only on $r_1,\cdots,r_j$. These matrices $M_j$ and $N_j$ are defined recursively by the following three conditions:
\begin{enumerate}
\item[(A)] $M_j$ is a $j\times(j+1)$ matrix and $N_j$ is a $(j+1)\times j$ matrix. Moreover, some of their entries are fixed to be $0$ or $1_R$ as follows:
\[
\begin{split}
M_j&=\begin{pmatrix}
\alpha_{j;1,1}&1_R&0&\cdots&0&0\\
\alpha_{j;2,1}&\alpha_{j;2,2}&1_R&\cdots&0&0\\
\vdots&\vdots&\vdots&\ddots&\vdots&\vdots\\
\alpha_{j;j-1,1}&\alpha_{j;j-1,2}&\alpha_{j;j-1,3}&\cdots&1_R&0\\
\alpha_{j;j,1}&\alpha_{j;j,2}&\alpha_{j;j,3}&\cdots&\alpha_{j;j,j}&1_R
\end{pmatrix},
\\
N_j&=\begin{pmatrix}
1_R&0&\cdots&0&0\\
\beta_{j;1,1}&1_R&\cdots&0&0\\
\vdots&\vdots&\ddots&\vdots&\vdots\\
\beta_{j;j-2,1}&\beta_{j;j-2,2}&\cdots&1_R&0\\
\beta_{j;j-1,1}&\beta_{j;j-1,2}&\cdots&\beta_{j,j-1,j-1}&1_R\\
\beta_{j;j,1}&\beta_{j;j,2}&\cdots&\beta_{j;j,j-1}&\beta_{j;j,j}
\end{pmatrix}.
\end{split}
\]
Here $\alpha_{j;a,b},\beta_{j;a,b}$ for $j\geq a\geq b\geq 1$ are elements of $R$, to be specified by the other conditions (observe that the subscript $a$ in $\beta_{j;a,b}$ is one less than the row number). Note that $M_0$ and $N_0$ must be the unique `empty' $0\times 1$ and $1\times 0$ matrices respectively, which gives us a base case for the recursion.
\item[(B)] For all $j\geq 1$, we have 
\[ M_jN_j=N_{j-1}M_{j-1}+X_j, \]
where $X_j$ is the $j\times j$ matrix all of whose entries are zero except for the lower-left entry, which is $r_j$. Thus, the entries of $M_{j-1}$ and $N_{j-1}$, together with $r_j$, determine each entry of $M_jN_j$. Note that the $(a,b)$ entry of $M_jN_j$, for $j\geq a\geq b\geq 1$, has the form $\alpha_{j;a,b}+\beta_{j;a,b}+\text{other terms}$, where the `other terms' involve only those $\alpha_{j;a',b'},\beta_{j;a',b'}$ where $a'-b'<a-b$.  
\item[(C)] For all $j\geq 1$, the matrix $N_jM_j$ satisfies the `Slodowy slice' condition that
\[ N_jM_j-E^{[j+1]}\text{ commutes with }F^{[j+1]}. \]
The matrix $N_jM_j-E^{[j+1]}$ is automatically lower-triangular by (A), so the condition that it commutes with $F^{[j+1]}$ amounts to a collection of linear equations on the diagonal and below-diagonal entries of $N_j M_j$. Note that the $(a,b)$ entry of $N_jM_j$, for $j+1\geq a\geq b\geq 1$, has the form $\alpha_{j;a,b}+\beta_{j;a-1,b-1}+\text{other terms}$, where the `other terms' involve only those $\alpha_{j;a',b'},\beta_{j;a',b'}$ where $a'-b'<a-b$, and $\alpha_{j;j+1,b}$ and $\beta_{j;a-1,0}$ are to be interpreted as zero. One can easily deduce from this that the linear equations in these entries, together with the equations given by (B), determine all $\alpha_{j;a,b},\beta_{j;a,b}$ uniquely.   
\end{enumerate}

\begin{ex}
The matrices $M_1,N_1,M_2,N_2$ are determined as follows. The $j=1$ case of (A) says that $M_1=(\alpha_{1;1,1}\ 1_R)$ and $N_1=(\begin{smallmatrix}1_R\\ \beta_{1;1,1}\end{smallmatrix})$ for some $\alpha_{1;1,1},\beta_{1;1,1}\in R$. The $j=1$ case of (B) says that $M_1 N_1=(r_1)$, i.e.\ $\alpha_{1;1,1}+\beta_{1;1,1}=r_1$. The $j=1$ case of (C) says that $N_1M_1-(\begin{smallmatrix}0&1_R\\0&0\end{smallmatrix})$ commutes with $(\begin{smallmatrix}0&0\\1_R&0\end{smallmatrix})$, i.e.\  $\alpha_{1;1,1}=\beta_{1;1,1}$. We conclude that 
\begin{equation}
M_1=\begin{pmatrix}\frac{r_1}{2}& 1_R\end{pmatrix},
\quad
N_1=\begin{pmatrix}1_R\\ \frac{r_1}{2}\end{pmatrix}.
\end{equation} 
The $j=2$ case of (A) says that $M_2$ and $N_2$ have the form:
\[
M_2=\begin{pmatrix}
\alpha_{2;1,1}&1_R&0\\
\alpha_{2;2,1}&\alpha_{2;2,2}&1_R
\end{pmatrix},
\quad
N_2=\begin{pmatrix}
1_R&0\\
\beta_{2;1,1}&1_R\\
\beta_{2;2,1}&\beta_{2;2,2}
\end{pmatrix}.
\]
The $j=2$ case of (B) is equivalent to the equations:
\[
\begin{split}
\alpha_{2;1,1}+\beta_{2;1,1}=\alpha_{2;2,2}+\beta_{2;2,2}&=\frac{r_1}{2},\\
\alpha_{2;2,1}+\beta_{2;2,1}+\alpha_{2;2,2}\beta_{2;1,1}&=\frac{r_1^2}{4}+r_2.
\end{split}
\]
The $j=2$ case of (C) is equivalent to the equations:
\[
\begin{aligned}
\alpha_{2;1,1}=\beta_{2;1,1}+\alpha_{2;2,2}=\beta_{2;2,2},\\
\alpha_{2;2,1}+\beta_{2;1,1}\alpha_{2;1,1}=\beta_{2;2,1}+\beta_{2;2,2}\alpha_{2;2,2}.
\end{aligned}
\]
We conclude that
\begin{equation}
M_2=\begin{pmatrix}
\frac{r_1}{3}&1_R&0\\
\frac{r_1^2}{9}+\frac{r_2}{2}&\frac{r_1}{6}&1_R
\end{pmatrix},
\quad
N_2=\begin{pmatrix}
1_R&0\\
\frac{r_1}{6}&1_R\\
\frac{r_1^2}{9}+\frac{r_2}{2}&\frac{r_1}{3}
\end{pmatrix}.
\end{equation}
\end{ex}

We do not need a general solution of the recursion, merely the following information about the dependence on the parameters.

\begin{lem} \label{lem:parameters}
Define the matrices $M_j,N_j$ with entries $\alpha_{j;a,b},\beta_{j;a,b}$ as above using the parameters $r_j$, and similarly define matrices $M_j',N_j'$ with entries $\alpha_{j;a,b}',\beta_{j;a,b}'$ using the parameters $r_j'$.
\begin{enumerate}
\item Let $\lambda\in\C$ and suppose $r_j'=\lambda^j r_j$ for all $j$. Then for all $j\geq a\geq b\geq 1$,
\[
\alpha_{j;a,b}'=\lambda^{a-b+1}\alpha_{j;a,b},\ 
\beta_{j;a,b}'=\lambda^{a-b+1}\beta_{j;a,b}.
\]
\item Let $r\mapsto r^*$ be an anti-automorphism of the $\C$-algebra $R$, and suppose that $r_j'=r_j^*$ for all $j$. Then for all $j\geq a\geq b\geq 1$,
\[
\alpha_{j;a,b}'=(\beta_{j;j-b+1,j-a+1})^*,\ 
\beta_{j;a,b}'=(\alpha_{j;j-b+1,j-a+1})^*.
\]
\end{enumerate}
\end{lem}

\begin{proof}
In each case it is straightforward to check that the matrices $M_j',N_j'$ defined by plugging the stated formulas into the expressions of condition (A) satisfy conditions (B) and (C) with the parameters $r_j'$.
\end{proof}

To express our special case of Maffei's isomorphism $\Phi_1$, we want to consider this recursion for $R=\End(W_n)$, up to the stage where it defines $M_n$ and $N_n$ (so we only need the parameters $r_1,\cdots,r_n$). Then $M_nN_n$, an $n\times n$ matrix with entries in $\End(W_n)$, can be regarded as an element of $\End(\tV_0)$. Recall the notation for compositions of the maps $B_h$ introduced in Example~\ref{ex:typea}.

\begin{lem} \label{lem:maffei-rephrased}
Continue the above assumptions, in particular~\eqref{eqn:small} with $k=n$. Let $(B_h,\Gamma_i,\Delta_i)\in\Lambda^{\A_{2n-1}}(V,W)^s$. Then $\Phi_1(\pi(G_V\cdot(B_h,\Gamma_i,\Delta_i)))$ is the matrix $M_nN_n$ obtained from the above recursion with $R=\End(W_n)$ and
\[
r_j=\Delta_n B_{n,n-j+1,n}\Gamma_n\ \text{ for $1\leq j\leq n$.}
\]
\end{lem}

\begin{proof}
We need to explain how to pass from Maffei's general definition of $\Phi_1$ to the above statement in our special case. For convenience, set $V_0=0$ and let $B_{0,1}:V_1\to V_0$ and $B_{1,0}:V_0\to V_1$ be zero maps. Along with $\tV_0=V_0\oplus W_n^{\oplus n}$, define the vector spaces
\[
\begin{split}
\tV_i&:=\begin{cases}
V_i\oplus W_n^{\oplus n-i},&\text{ for $1\leq i\leq n-1$,}\\
V_i,&\text{ for $n\leq i\leq 2n-1$.}
\end{cases}
\end{split}
\]
In~\cite[Lemma 18]{maffei}, Maffei associates to $(B_h,\Gamma_i,\Delta_i)\in\Lambda^{\A_{2n-1}}(V,W)$ a collection of linear maps
\begin{equation} \label{eqn:config}
\xymatrix{
\tV_0\ar@/^/[r]^{\tA_0}&\tV_1\ar@/^/[l]^{\tB_0}\ar@/^/[r]^{\tA_1}&\tV_2\ar@/^/[l]^{\tB_1}\ar@/^/[r]^{\tA_2}&\cdots\ar@/^/[l]^{\tB_2}\ar@/^/[r]^{\tA_{2n-2}}&\tV_{2n-1}\ar@/^/[l]^{\tB_{2n-2}}
}
\end{equation}
which by~\cite[Definition 16, Lemma 17]{maffei} is the unique such collection satisfying the following three conditions. (Recall from Remark~\ref{rmk:sign} that there are some signs involved in changing from Maffei's orientation to ours.)
\begin{enumerate}
\item[(A')] For $n\leq i\leq 2n-2$, $\tA_i=B_{i+1,i}$ and $\tB_i=-B_{i,i+1}$. For $0\leq i\leq n-1$, $\tA_i$ and $\tB_i$ have the block-matrix form:
\[
\begin{split}
\tA_i&=\begin{pmatrix}
B_{i+1,i}&B_{i+1,n}\Gamma_n&0&0&\cdots&0\\
0&\talpha_{i,n,1}^{n,1}&\mathrm{id}_{W_n}&0&\cdots&0\\
0&\talpha_{i,n,2}^{n,1}&\talpha_{i,n,2}^{n,2}&\mathrm{id}_{W_n}&\cdots&0\\
\vdots&\vdots&\vdots&\vdots&\ddots&\vdots\\
0&\talpha_{i,n,n-i-1}^{n,1}&\talpha_{i,n,n-i-1}^{n,2}&\talpha_{i,n,n-i-1}^{n,3}&\cdots&\mathrm{id}_{W_n}
\end{pmatrix},
\\
\tB_i&=\begin{pmatrix}
B_{i,i+1}&0&0&\cdots&0\\
0&\mathrm{id}_{W_n}&0&\cdots&0\\
0&\sbeta_{i,n,2}^{n,1}&\mathrm{id}_{W_n}&\cdots&0\\
\vdots&\vdots&\vdots&\ddots&\vdots\\
0&\sbeta_{i,n,n-i-1}^{n,1}&\sbeta_{i,n,n-i-1}^{n,2}&\cdots&\mathrm{id}_{W_n}\\
\Delta_n B_{n,i+1}&\sbeta_{i,n,n-i}^{n,1}&\sbeta_{i,n,n-i}^{n,2}&\cdots&\sbeta_{i,n,n-i}^{n,n-i-1}
\end{pmatrix},
\end{split}
\]
for some $\talpha_{i,n,h}^{n,h'},\sbeta_{i,n,h}^{n,h'}\in\End(W_n)$ (to use Maffei's notation).
\item[(B')] $\tA_i\tB_i=\tB_{i+1}\tA_{i+1}$ for $0\leq i\leq 2n-3$. 
\item[(C')] $[\tB_i\tA_i-E_i,F_i]=0$ for $0\leq i\leq n-1$, where $E_i,F_i\in\fsl(\tV_i)$ are defined as block matrices which have the form $E^{[n-i]},F^{[n-i]}$
as regards the $W_n$ summands and have zero blocks in the row and column corresponding to the $V_i$ summand.
\end{enumerate}
By~\cite[Definition 20]{maffei}, if $(B_h,\Gamma_i,\Delta_i)\in\Lambda^{\A_{2n-1}}(V,W)^s$, then 
\begin{equation}
\Phi_1(\pi(G_V\cdot(B_h,\Gamma_i,\Delta_i)))=\tB_0\tA_0.
\end{equation}
Now for $0\leq j\leq n-1$ let $M_j$ and $N_j$ be obtained from $\tA_{n-j-1}$ and $\tB_{n-j-1}$ respectively by deleting the first row and column in the above block-matrix expressions. Thus $M_j$ is a $j\times (j+1)$ matrix with entries in $\End(W_n)$ and $N_j$ is a $(j+1)\times j$ matrix with entries in $\End(W_n)$. As may be easily checked, conditions (A')--(C') in the definition of $\tA_i,\tB_i$ imply respectively that $M_j$ and $N_j$ satisfy conditions (A)--(C) in the above recursion, with the parameters $r_j$ for $1\leq j\leq n-1$ defined as in the statement.
Using condition (B) in the recursion once more, one sees that $\tB_0\tA_0=M_nN_n$ where the parameter $r_n$ is defined by the same formula.
\end{proof}

\begin{rmk}
As the proof shows, Lemma~\ref{lem:maffei-rephrased} does not in fact require $\dim W_n=2$.
\end{rmk}

\subsection{Proof of Theorem~\ref{thm:involutions} in the $k=n$ case}
\label{subsect:proofk=n}

The proof requires some preliminary calculations. First we have a general result about configurations of linear maps satisfying~\eqref{eqn:quadratic}. 

\begin{lem} \label{lem:inverses}
Let $(B_h,\Gamma_i,\Delta_i)\in\Lambda^{\A_{2n-1}}(V,W)$ for arbitrary graded vector spaces $V$ and $W$. Define $A,B\in\End(V)$, $\Gamma\in\Hom(W,V)$ and $\Delta\in\Hom(V,W)$ by the following block-matrix formulas:
\[
\begin{array}{ll}
A=\displaystyle\bigoplus_{i=1}^{2n-2}B_{i+1,i},& B=\displaystyle\bigoplus_{i=1}^{2n-2}\pm B_{i,i+1},\\
\\
\Gamma=\displaystyle\bigoplus_{i=1}^{2n-1}\Gamma_i,& \Delta=\displaystyle\bigoplus_{i=1}^{2n-1} \Delta_i,
\end{array}
\]
where the sign in front of $B_{i,i+1}$ is $+$ if $i<n$ and $-$ if $i\geq n$.
Then the following elements of $\End(W)[z]$ are inverse to each other:
\[
\begin{split}
X(z)&=\mathrm{id}_W - \sum_{j,k\geq 0} \Delta A^j B^k \Gamma\,z^{j+k+2},\\
Y(z)&=\mathrm{id}_W + \sum_{j,k\geq 0} \Delta B^k A^j \Gamma\,z^{j+k+2}.
\end{split}
\]
\end{lem}

\begin{proof}
First note that $X(z)$ and $Y(z)$ do indeed belong to $\End(W)[z]$ and not $\End(W)[\![z]\!]$, since $A$ and $B$ are nilpotent.
The equations~\eqref{eqn:quadratic} can be combined into a single equation in $\End(V)$:
\begin{equation} \label{eqn:adhm-again}
BA-AB=\Gamma\Delta.
\end{equation}
Hence we have
\begin{equation*}
\begin{split}
X(z)Y(z)&=\mathrm{id}_W+ \sum_{j,k\geq 0} \Delta B^k A^j \Gamma\,z^{j+k+2}-\sum_{j,k\geq 0} \Delta A^j B^k \Gamma\,z^{j+k+2}\\
&\quad -\sum_{j,k,j',k'\geq 0} \Delta A^j B^k(BA-AB)B^{k'}\!A^{j'}\Gamma\,z^{j+k+j'+k'+4},
\end{split}
\end{equation*}
and it is easy to see that all terms except $\mathrm{id}_W$ cancel out.
\end{proof}

\begin{lem} \label{lem:sums}
Revert to the standing assumptions, in particular~\eqref{eqn:small} with $k=n$. Let $(B_h,\Gamma_i,\Delta_i)\in\Lambda^{\A_{2n-1}}(V,W)$. Then for $1\leq j\leq n$,
\[
(\Delta_nB_{n,n-j+1,n}\Gamma_n)^\bt=(-1)^{j}\sigma_n\Delta_nB_{n,n+j-1,n}\Gamma_n\sigma_n,
\]
where $x\mapsto x^\bt$ denote the transpose anti-automorphism of $\End(W_n)$ with respect to the form $(\cdot,\cdot)_n$.
\end{lem}

\begin{proof}
Under our assumptions, $W=W_n$ and the elements of $\End(W_n)[z]$ defined in Lemma~\ref{lem:inverses} become
\[
\begin{split}
X(z)&=\mathrm{id}_{W_n}-\sum_{j=1}^n\Delta_nB_{n,n-j+1,n}\Gamma_n\, z^{2j},\\
Y(z)&=\mathrm{id}_{W_n}+\sum_{j=1}^n(-1)^{j-1}\Delta_nB_{n,n+j-1,n}\Gamma_n\, z^{2j}.
\end{split}
\]
Thinking of $X(z)$ and $Y(z)$ as elements of $\End(W_n\otimes_\C \C[z])$, their determinants are polynomials in $z$ with constant term $1$, and by Lemma~\ref{lem:inverses} must be identically $1$. Since $\dim W_n=2$, this implies that $X(z)$ stabilizes the $\C[z]$-valued skew-symmetric form $\langle\cdot,\cdot\rangle$ on $W_n\otimes_\C \C[z]$ obtained by extension of scalars from the form $\langle\cdot,\cdot\rangle_n$. Hence we have, for all $w,y\in W_n$,
\begin{equation}
\langle X(z) w,y\rangle=\langle w,Y(z)y\rangle,
\end{equation}
or in other words
\begin{equation}
\langle \Delta_nB_{n,n-j+1,n}\Gamma_n w,y\rangle_n=(-1)^{j}\langle w, \Delta_nB_{n,n+j-1,n}\Gamma_n y\rangle_n,
\end{equation}
for $1\leq j\leq n$. By definition of $(\cdot,\cdot)_n$, this is equivalent to the desired statement.
\end{proof}

The proof of Theorem~\ref{thm:involutions} in the $k=n$ case is then as follows.

\begin{proof}
Let $(B_h,\Gamma_i,\Delta_i)\in\Lambda^{\A_{2n-1}}(V,W)^s$. We must prove that 
\begin{equation} \label{eqn:involutions}
\Theta(\Phi_1(\pi(G_V\cdot(B_h,\Gamma_i,\Delta_i))))=\Phi_1(\pi(G_V\cdot \theta(B_h,\Gamma_i,\Delta_i))).
\end{equation}
Recall from Lemma~\ref{lem:maffei-rephrased} that $\Phi_1(\pi(G_V\cdot(B_h,\Gamma_i,\Delta_i)))$ equals $M_nN_n$ where $M_j,N_j$ are obtained from the recursion with the parameters $r_j=\Delta_nB_{n,n-j+1,n}\Gamma_n$. By the definition of $\theta$, the right-hand side of~\eqref{eqn:involutions} equals $M_n'N_n'$ where $M_j',N_j'$ are obtained from the recursion with the parameters $r_j'=\sigma_n\Delta_nB_{n,n+j-1,n}\Gamma_n\sigma_n$. By Lemma~\ref{lem:sums} we have $r_j'=(-1)^jr_j^\bt$. Applying Lemma~\ref{lem:parameters}(1) with $\lambda=-1$ and Lemma~\ref{lem:parameters}(2) with ${}^*={}^{\bt}$, we find that the entries $\alpha_{j;a,b}',\beta_{j;a,b}'$ of $M_j',N_j'$ are related to the entries $\alpha_{j;a,b},\beta_{j;a,b}$ of $M_j,N_j$ by the rule
\begin{equation}
\begin{split}
\alpha_{j;a,b}'&=(-1)^{a-b+1}(\beta_{j;j-b+1,j-a+1})^\bt,\\ 
\beta_{j;a,b}'&=(-1)^{a-b+1}(\alpha_{j;j-b+1,j-a+1})^\bt.
\end{split}
\end{equation}
It follows that if $x_{c,d}\in\End(W_n)$ denotes the $(c,d)$ entry of $M_nN_n$, then the $(c,d)$-entry of $M_n'N_n'$ is $(-1)^{c-d+1}(x_{n-d+1,n-c+1})^\bt$. This means exactly that $\Theta(M_nN_n)=M_n'N_n'$, so~\eqref{eqn:involutions} is proved.
\end{proof}

\subsection{Maffei's isomorphism $\Phi_1$ in the $k<n$ case}

Now we turn to the $k<n$ case. This is fairly similar to the $k=n$ case, but notationally more complex because we have the two vector spaces $W_k$ and $W_{2n-k}$ to keep track of.

We need to modify the recursion defined in \S\ref{subsect:recursion} as follows. Let $m$ be an even positive integer. In addition to the parameters $r_1,r_2,\cdots\in R$, we take three further sequences of parameters $r_j^{0,0},r_j^{0,1},r_j^{1,0}\in R$ for $j\geq 1$. For uniformity of notation, we let $r_j^{1,1}=r_j$. We use these parameters to define $M_j$ and $N_j$ for $j\in\N$, matrices with entries in $R$. For $0\leq j\leq m-1$, the definition of $M_j$ and $N_j$ is exactly as before (so these depend only on $r_1^{1,1},\cdots,r_j^{1,1}$). The entries formerly called $\alpha_{j;a,b},\beta_{j;a,b}$ we rename $\alpha_{j;a,b}^{1,1},\beta_{j;a,b}^{1,1}$.

For $j\geq m$, $M_j$ and $N_j$ are defined recursively by the following three conditions:
\begin{enumerate}
\item[(A)] $M_j$ and $N_j$ have the following block-matrix shape:
\[
M_j=\begin{pmatrix}
M_j^{0,0}&M_j^{0,1}\\M_j^{1,0}&M_j^{1,1}
\end{pmatrix},\
N_j=\begin{pmatrix}
N_j^{0,0}&N_j^{0,1}\\N_j^{1,0}&N_j^{1,1}
\end{pmatrix},
\]
where $M_j^{0,0}$ is a $(j-m)\times(j-m+1)$ matrix, $M_j^{1,1}$ is a $j\times (j+1)$ matrix, $N_j^{0,0}$ is a $(j-m+1)\times(j-m)$ matrix, $N_j^{1,1}$ is a $(j+1)\times j$ matrix, and the sizes of the other blocks are then determined. Moreover, the blocks of $M_j$ have the form:
\[
\begin{split}
M_j^{0,0}&=\begin{pmatrix}
\alpha_{j;1,1}^{0,0}&1_R&0&\cdots&0&0\\
\alpha_{j;2,1}^{0,0}&\alpha_{j;2,2}^{0,0}&1_R&\cdots&0&0\\
\vdots&\vdots&\vdots&\ddots&\vdots&\vdots\\
\alpha_{j;j-m-1,1}^{0,0}&\alpha_{j;j-m-1,2}^{0,0}&\alpha_{j;j-m-1,3}^{0,0}&\cdots&1_R&0\\
\alpha_{j;j-m,1}^{0,0}&\alpha_{j;j-m,2}^{0,0}&\alpha_{j;j-m,3}^{0,0}&\cdots&\alpha_{j;j-m,j-m}^{0,0}&1_R
\end{pmatrix},
\\
M_j^{0,1}&=\begin{pmatrix}
\alpha_{j;1,1}^{0,1}&0&\cdots&0&0&\cdots&0\\
\vdots&\vdots&\ddots&\vdots&\vdots&&\vdots\\
\alpha_{j;j-m-1,1}^{0,1}&\alpha_{j;j-m-1,2}^{0,1}&\cdots&0&0&\cdots&0\\
\alpha_{j;j-m,1}^{0,1}&\alpha_{j;j-m,2}^{0,1}&\cdots&\alpha_{j;j-m,j-m}^{0,1}&0&\cdots&0
\end{pmatrix},
\\
M_j^{1,0}&=\begin{pmatrix}
0&0&\cdots&0&0\\
\vdots&\vdots&&\vdots&\vdots\\
0&0&\cdots&0&0\\
\alpha_{j;m+1,1}^{1,0}&0&\cdots&0&0\\
\vdots&\vdots&\ddots&\vdots&\vdots\\
\alpha_{j;j-1,1}^{1,0}&\alpha_{j;j-1,2}^{1,0}&\cdots&0&0\\
\alpha_{j;j,1}^{1,0}&\alpha_{j;j,2}^{1,0}&\cdots&\alpha_{j;j,j-m}^{1,0}&0
\end{pmatrix},
\\
M_j^{1,1}&=\begin{pmatrix}
\alpha_{j;1,1}^{1,1}&1_R&0&\cdots&0&0\\
\alpha_{j;2,1}^{1,1}&\alpha_{j;2,2}^{1,1}&1_R&\cdots&0&0\\
\vdots&\vdots&\vdots&\ddots&\vdots&\vdots\\
\alpha_{j;j-1,1}^{1,1}&\alpha_{j;j-1,2}^{1,1}&\alpha_{j;j-1,3}^{1,1}&\cdots&1_R&0\\
\alpha_{j;j,1}^{1,1}&\alpha_{j;j,2}^{1,1}&\alpha_{j;j,3}^{1,1}&\cdots&\alpha_{j;j,j}^{1,1}&1_R
\end{pmatrix},
\end{split}
\]
and the blocks of $N_j$ have the form:
\[
\begin{split}
N_j^{0,0}&=\begin{pmatrix}
1_R&0&\cdots&0&0\\
\beta_{j;1,1}^{0,0}&1_R&\cdots&0&0\\
\vdots&\vdots&\ddots&\vdots&\vdots\\
\beta_{j;j-m-2,1}^{0,0}&\beta_{j;j-m-2,2}^{0,0}&\cdots&1_R&0\\
\beta_{j;j-m-1,1}^{0,0}&\beta_{j;j-m-1,2}^{0,0}&\cdots&\beta_{j;j-m-1,j-m-1}^{0,0}&1_R\\
\beta_{j;j-m,1}^{0,0}&\beta_{j;j-m,2}^{0,0}&\cdots&\beta_{j;j-m,j-m-1}^{0,0}&\beta_{j;j-m,j-m}^{0,0}
\end{pmatrix},
\\
N_j^{0,1}&=\begin{pmatrix}
0&0&\cdots&0&0&\cdots&0\\
\beta_{j;1,1}^{0,1}&0&\cdots&0&0&\cdots&0\\
\vdots&\vdots&\ddots&\vdots&\vdots&&\vdots\\
\beta_{j;j-m-1,1}^{0,1}&\beta_{j;j-m-1,2}^{0,1}&\cdots&0&0&\cdots&0\\
\beta_{j;j-m,1}^{0,1}&\beta_{j;j-m,2}^{0,1}&\cdots&\beta_{j;j-m,j-m}^{0,1}&0&\cdots&0
\end{pmatrix},
\\
N_j^{1,0}&=\begin{pmatrix}
0&0&\cdots&0\\
\vdots&\vdots&&\vdots\\
0&0&\cdots&0\\
\beta_{j;m+1,1}^{1,0}&0&\cdots&0\\
\vdots&\vdots&\ddots&\vdots\\
\beta_{j;j-1,1}^{1,0}&\beta_{j;j-1,2}^{1,0}&\cdots&0\\
\beta_{j;j,1}^{1,0}&\beta_{j;j,2}^{1,0}&\cdots&\beta_{j;j,j-m}^{1,0}
\end{pmatrix},
\\
N_j^{1,1}&=\begin{pmatrix}
1_R&0&\cdots&0&0\\
\beta_{j;1,1}^{1,1}&1_R&\cdots&0&0\\
\vdots&\vdots&\ddots&\vdots&\vdots\\
\beta_{j;j-2,1}^{1,1}&\beta_{j;j-2,2}^{1,1}&\cdots&1_R&0\\
\beta_{j;j-1,1}^{1,1}&\beta_{j;j-1,2}^{1,1}&\cdots&\beta_{j;j-1,j-1}^{1,1}&1_R\\
\beta_{j;j,1}^{1,1}&\beta_{j;j,2}^{1,1}&\cdots&\beta_{j;j,j-1}^{1,1}&\beta_{j;j,j}^{1,1}
\end{pmatrix}.
\end{split}
\]
\item[(B)] For all $j\geq m$, we have an equation of $(2j-m)\times (2j-m)$ matrices:
\[
M_jN_j=N_{j-1}M_{j-1}+\begin{pmatrix}
X_j^{0,0}&X_j^{0,1}\\
X_j^{1,0}&X_j^{1,1}
\end{pmatrix},
\]
where $X_j^{0,0}$ is a $(j-m)\times (j-m)$ matrix, $X_j^{1,1}$ is a $j\times j$ matrix, the sizes of $X_j^{0,1}$ and $X_j^{1,0}$ are then determined, and all entries of these matrices are zero except that the lower-left entry of $X_j^{e,f}$ is $r_{j-m}^{e,f}$ for $(e,f)\neq(1,1)$, and the lower-left entry of $X_j^{1,1}$ is $r_j^{1,1}$. (When $j=m$, $X_m^{0,0}$, $X_m^{0,1}$ and $X_m^{1,0}$ are empty matrices.)
\item[(C)] For all $j\geq m$,
\[ N_jM_j-\begin{pmatrix}E^{[j-m+1]}&0\\0&E^{[j+1]}\end{pmatrix}\text{
commutes with }\begin{pmatrix}F^{[j-m+1]}&0\\0&F^{[j+1]}\end{pmatrix}. \]
\end{enumerate}

The analogue of Lemma~\ref{lem:parameters}, proved in the same way, is:

\begin{lem} \label{lem:parametersk<n}
Define the matrices $M_j,N_j$ with entries $\alpha_{j;a,b}^{e,f},\beta_{j;a,b}^{e,f}$ as above using the parameters $r_j^{e,f}$, and similarly define matrices $M_j',N_j'$ with entries $\alpha_{j;a,b}^{\prime,e,f},\beta_{j;a,b}^{\prime,e,f}$ using the parameters $r_j^{\prime,e,f}$.
\begin{enumerate}
\item Let $\lambda\in\C$, $\varepsilon\in\{\pm 1\}$ and suppose 
\[ 
r_j^{\prime,e,f}=\begin{cases}
\lambda^{j} r_j^{e,f}&\text{ if $e=f$,}\\
\varepsilon\lambda^{j+\frac{m}{2}} r_j^{e,f}&\text{ if $e\neq f$.}
\end{cases}
\]
Then
\[
\alpha_{j;a,b}^{\prime,e,f}=\varepsilon^{f-e}\lambda^{a-b+(f-e)\frac{m}{2}+1}\alpha_{j;a,b}^{e,f},\ 
\beta_{j;a,b}^{\prime,e,f}=\varepsilon^{f-e}\lambda^{a-b+(f-e)\frac{m}{2}+1}\beta_{j;a,b}^{e,f}.
\]
\item Let $r\mapsto r^*$ be an anti-automorphism of the $\C$-algebra $R$, and suppose that $r_j^{\prime,e,f}=(r_j^{f,e})^*$ for all $j,e,f$. Then
\[
\begin{split}
\alpha_{j;a,b}^{\prime,e,f}=(\beta_{j;j+(f-1)m-b+1,j+(e-1)m-a+1}^{f,e})^*,\\
\beta_{j;a,b}^{\prime,e,f}=(\alpha_{j;j+(f-1)m-b+1,j+(e-1)m-a+1}^{f,e})^*.
\end{split}
\]
\end{enumerate}
\end{lem}

Recall that $W_k=W_{2n-k}$ is one-dimensional, so we can identify each of $\End(W_k)$, $\End(W_{2n-k})$, $\Hom(W_k,W_{2n-k})$ and $\Hom(W_{2n-k},W_k)$ with $\C$. To express the relevant special case of Maffei's isomorphism $\Phi_1$, we want to consider the above recursion for $R=\C$ and $m=2n-2k$, up to the stage where it defines $M_{2n-k}$ and $N_{2n-k}$ (so we only need the parameters $r_j^{e,f}$ for $1\leq j\leq k$, $(e,f)\neq (1,1)$, and $r_j^{1,1}$ for $1\leq j\leq 2n-k$). Then $M_{2n-k}N_{2n-k}$, a $2n\times 2n$ complex matrix, can be regarded as an element of $\End(\tV_0)$, with the first $k$ rows and columns corresponding to the $W_k$ summands and the remaining rows and columns corresponding to the $W_{2n-k}$ summands.

\begin{lem} \label{lem:maffei-rephrasedk<n}
Continue the above assumptions, in particular~\eqref{eqn:small} with $k<n$. Let $(B_h,\Gamma_i,\Delta_i)\in\Lambda^{\A_{2n-1}}(V,W)^s$. Then $\Phi_1(\pi(G_V\cdot(B_h,\Gamma_i,\Delta_i)))$ is the matrix $M_{2n-k}N_{2n-k}$ obtained from the above recursion with $R=\C$, $m=2n-2k$ and
\[
\begin{split}
r_j^{0,0}&=\Delta_k B_{k,k-j+1,k}\Gamma_k,\\
r_j^{0,1}&=(-1)^{n-k}\Delta_k B_{k,k-j+1,2n-k}\Gamma_{2n-k},\\
r_j^{1,0}&=\Delta_{2n-k} B_{2n-k,k-j+1,k}\Gamma_k,\\
r_j^{1,1}&=(-1)^{\min\{j-1,n-k\}}\Delta_{2n-k} B_{2n-k,2n-k-j+1,2n-k}\Gamma_{2n-k}.
\end{split}
\]
\end{lem}

\begin{proof}
This is a straightforward rephrasing of Maffei's definition, along the lines of Lemma~\ref{lem:maffei-rephrased}. The signs are explained by Remark~\ref{rmk:sign}.
\end{proof}

\subsection{Proof of Theorem~\ref{thm:involutions} in the $k<n$ case}
\label{subsect:proofk<n}

Following the pattern of \S\ref{subsect:proofk=n}, we first determine what Lemma~\ref{lem:inverses} implies in this case.

\begin{lem} \label{lem:victory}
Continue the standing assumptions, in particular~\eqref{eqn:small} with $k<n$. Let $(B_h,\Gamma_i,\Delta_i)\in\Lambda^{\A_{2n-1}}(V,W)$. Then for all $j$,
\[
\begin{split}
\Delta_k B_{k,k-j+1,k}\Gamma_k
&=(-1)^j\Delta_{2n-k} B_{2n-k,2n-k+j-1,2n-k}\Gamma_{2n-k},\\
\Delta_k B_{k,k-j+1,2n-k}\Gamma_{2n-k}
&=(-1)^{j-1}\Delta_k B_{k,2n-k+j-1,2n-k}\Gamma_{2n-k},\\
\Delta_{2n-k} B_{2n-k,k-j+1,k}\Gamma_k
&=(-1)^{j-1}\Delta_{2n-k} B_{2n-k,2n-k+j-1,k}\Gamma_k,\\
\Delta_{2n-k} B_{2n-k,2n-k-j+1,2n-k}\Gamma_{2n-k}
&=(-1)^j\Delta_k B_{k,k+j-1,k}\Gamma_k.
\end{split}
\]
\end{lem}

\begin{proof}
We can identify $\End(W_k\oplus W_{2n-k})$ with the ring of $2\times 2$ complex matrices, and hence identify $\End(W_k\oplus W_{2n-k})[z]$ with the ring of $2\times 2$ matrices over $\C[z]$. The elements of the latter ring defined in Lemma~\ref{lem:inverses} have the following entries:
\[
\begin{split}
X(z)_{11}&=1-\displaystyle\sum_{j=1}^k\Delta_k B_{k,k-j+1,k}\Gamma_k\, z^{2j},\\
X(z)_{12}&=\displaystyle\sum_{j=1}^k(-1)^{n-k+1}\Delta_k B_{k,k-j+1,2n-k}\Gamma_{2n-k}\, z^{2j+2n-2k},\\
X(z)_{21}&=-\displaystyle\sum_{j=1}^k\Delta_{2n-k} B_{2n-k,k-j+1,k}\Gamma_k\, z^{2j+2n-2k},\\
X(z)_{22}&=1+\displaystyle\sum_{j=1}^{2n-k}(-1)^{\min\{j,n-k+1\}}\Delta_{2n-k} B_{2n-k,2n-k-j+1,2n-k}\Gamma_{2n-k}\, z^{2j},\\
Y(z)_{11}&=1+\displaystyle\sum_{j=1}^{2n-k}(-1)^{j-\min\{j,n-k+1\}}\Delta_k B_{k,k+j-1,k}\Gamma_k\, z^{2j},\\
Y(z)_{12}&=\displaystyle\sum_{j=1}^k(-1)^{n-k+j-1}\Delta_k B_{k,2n-k+j-1,2n-k}\Gamma_{2n-k}\, z^{2j+2n-2k},\\
Y(z)_{21}&=\displaystyle\sum_{j=1}^k(-1)^{j-1}\Delta_{2n-k} B_{2n-k,2n-k+j-1,k}\Gamma_k\, z^{2j+2n-2k},\\
Y(z)_{22}&=1+\displaystyle\sum_{j=1}^k(-1)^{j-1}\Delta_{2n-k} B_{2n-k,2n-k+j-1,2n-k}\Gamma_{2n-k}\, z^{2j}.
\end{split}
\]
Lemma~\ref{lem:inverses} says that $X(z)$ and $Y(z)$ are inverse to each other. So, as in the proof of Lemma~\ref{lem:sums}, their determinants are identically $1$. Hence
\[
\begin{split}
X(z)_{11}&=Y(z)_{22},\\
X(z)_{12}&=-Y(z)_{12},\\
X(z)_{21}&=-Y(z)_{21},\\
X(z)_{22}&=Y(z)_{11},
\end{split}
\]
which gives the result.
\end{proof}

The proof of Theorem~\ref{thm:involutions} in the $k<n$ case is then as follows.

\begin{proof}
Let $(B_h,\Gamma_i,\Delta_i)\in\Lambda^{\A_{2n-1}}(V,W)^s$. We must prove~\eqref{eqn:involutions}.
Recall the description of $\Phi_1(\pi(G_V\cdot(B_h,\Gamma_i,\Delta_i)))$ given in Lemma~\ref{lem:maffei-rephrasedk<n}. By the definition of $\theta$, $\Phi_1(\pi(G_V\cdot\theta(B_h,\Gamma_i,\Delta_i)))$ equals $M_{2n-k}'N_{2n-k}'$ where $M_j',N_j'$ are obtained from the same recursion but with the parameters 
\[
\begin{split}
r_j^{\prime,0,0}&=\Delta_{2n-k} B_{2n-k,2n-k+j-1,2n-k}\Gamma_{2n-k},\\
r_j^{\prime,0,1}&=(-1)^{n-k}\Delta_{2n-k} B_{2n-k,2n-k+j-1,k}\Gamma_{k},\\
r_j^{\prime,1,0}&=\Delta_{k} B_{k,2n-k+j-1,2n-k}\Gamma_{2n-k},\\
r_j^{\prime,1,1}&=(-1)^{\min\{j-1,n-k\}}\Delta_{k} B_{k,k+j-1,k}\Gamma_{k}.
\end{split}
\]
By Lemma~\ref{lem:victory} we have $r_j^{\prime,e,f}=(-1)^jr_j^{f,e}$ when $e=f$ and $r_j^{\prime,e,f}=(-1)^{j+n-k-1}r_j^{f,e}$ when $e\neq f$. Applying Lemma~\ref{lem:parametersk<n}(1) with $\lambda=\varepsilon=-1$ and Lemma~\ref{lem:parametersk<n}(2) with ${}^*$ being the identity, we find that the entries $\alpha_{j;a,b}^{\prime,e,f},\beta_{j;a,b}^{\prime,e,f}$ of $M_j',N_j'$ are related to the entries $\alpha_{j;a,b}^{e,f},\beta_{j;a,b}^{e,f}$ of $M_j,N_j$ by the rule
\begin{equation*}
\begin{split}
\alpha_{j;a,b}^{\prime,0,0}&=(-1)^{a-b+1}\beta_{j;j-2n+2k-b+1,j-2n+2k-a+1}^{0,0},\\
\alpha_{j;a,b}^{\prime,0,1}&=(-1)^{a-b+n-k}\beta_{j;j-b+1,j-2n+2k-a+1}^{1,0},\\
\alpha_{j;a,b}^{\prime,1,0}&=(-1)^{a-b+n-k}\beta_{j;j-2n+2k-b+1,j-a+1}^{0,1},\\
\alpha_{j;a,b}^{\prime,1,1}&=(-1)^{a-b+1}\beta_{j;j-b+1,j-a+1}^{1,1},\\
\beta_{j;a,b}^{\prime,0,0}&=(-1)^{a-b+1}\alpha_{j;j-2n+2k-b+1,j-2n+2k-a+1}^{0,0},\\
\beta_{j;a,b}^{\prime,0,1}&=(-1)^{a-b+n-k}\alpha_{j;j-b+1,j-2n+2k-a+1}^{1,0},\\
\beta_{j;a,b}^{\prime,1,0}&=(-1)^{a-b+n-k}\alpha_{j;j-2n+2k-b+1,j-a+1}^{0,1},\\
\beta_{j;a,b}^{\prime,1,1}&=(-1)^{a-b+1}\alpha_{j;j-b+1,j-a+1}^{1,1}.
\end{split}
\end{equation*}
It follows easily that $\Theta(M_{2n-k}N_{2n-k})=M_{2n-k}'N_{2n-k}'$, so~\eqref{eqn:involutions} is proved.
\end{proof}

\subsection{Lifting Theorem~\ref{thm:involutions} to $\fM^{\A_{2n-1}}(V,W)$}

To deduce from Theorem~\ref{thm:involutions} a similar result for the variety $\fM^{\A_{2n-1}}(V,W)$, we need to make the following further assumption on $\bv$:
\begin{equation} \label{eqn:symmetry}
\begin{split}
&\text{the set }\{s_1+s_2+\cdots+s_i+i,s_1+s_2+\cdots+s_i+(2n-i)\,|\,1\leq i\leq n\}\\
&\text{is stable under the involution }a\mapsto 2n-a.
\end{split}
\end{equation}
Note that the elements of the set in~\eqref{eqn:symmetry} are exactly the dimensions $\dim U_i$, $1\leq i\leq 2n-1$, for partial flags $(U_i)$ in the variety $\cF$ introduced in \S\ref{subsect:setup}. So this assumption ensures that we can lift the involution $\Theta$ of $\fsl(\tV_0)$ to an involution of $T^*\cF$, also denoted $\Theta$, defined by 
\begin{equation}
\Theta(X,(U_i))=(\Theta(X),(U_i)^\perp),
\end{equation} 
where $(U_i)^\perp$ is the flag obtained from $U_i$ by taking perpendicular subspaces relative to the form $(\cdot,\cdot)$. (We have not assumed that $\dim U_{2n-i}=2n-\dim U_i$ or that the subspaces $U_i$ are distinct, so the numbering of the subspaces of $(U_i)^\perp$ may not be given by any obvious rule.) Clearly $\Theta\circ\mu=\mu\circ\Theta$, so we have an induced involution $\Theta$ of $\mu^{-1}(\cS^{\A_{2n-1}}_{(2n-k,k)}\cap\overline{\cO}^{\A_{2n-1}}_{(2n-\ell,\ell)})$.

\begin{cor} \label{cor:involutions}
Continue the above assumptions, in particular~\eqref{eqn:small},~\eqref{eqn:v-symmetry},~\eqref{eqn:nonempty} and~\eqref{eqn:symmetry}.
The involution $\Theta$ of $\mu^{-1}(\cS^{\A_{2n-1}}_{(2n-k,k)}\cap\overline{\cO}^{\A_{2n-1}}_{(2n-\ell,\ell)})$ corresponds, under the isomorphism $\Phi$ of Proposition~\ref{prop:maffei1}, to the diagram involution $\theta$ of $\fM(V,W)$.
\end{cor}

\begin{proof}
Let $X$ denote the irreducible variety $\mu^{-1}(\cS^{\A_{2n-1}}_{(2n-k,k)}\cap\overline{\cO}^{\A_{2n-1}}_{(2n-\ell,\ell)})$ and write $U$ for its open subvariety $\mu^{-1}(\cS^{\A_{2n-1}}_{(2n-k,k)}\cap\cO^{\A_{2n-1}}_{(2n-\ell,\ell)})$. Temporarily let $\Theta'$ denote the involution $\Phi\theta\Phi^{-1}$ of $X$. We want to prove that $\Theta'(x)=\Theta(x)$ for all $x\in X$, and it suffices to prove this for $x$ in the dense subset $U\cap\Theta'(U)$. But for $x\in U\cap\Theta'(U)$ we have $\Theta'(x),\Theta(x)\in U\subset \mu^{-1}(\cO^{\A_{2n-1}}_{(2n-\ell,\ell)})$, and $\mu$ is injective on $\mu^{-1}(\cO^{\A_{2n-1}}_{(2n-\ell,\ell)})$, so it suffices to prove that $\mu(\Theta'(x))=\mu(\Theta(x))$ for all $x\in U\cap\Theta'(U)$. Using Proposition~\ref{prop:maffei1} and Theorem~\ref{thm:involutions}, we find that
$\mu(\Theta'(x))=\Phi_1(\theta(\pi(\Phi^{-1}(x))))
=\Theta(\Phi_1(\pi(\Phi^{-1}(x))))
=\Theta(\mu(x))
=\mu(\Theta(x))$,
as required.
\end{proof}


\section{Consequences}
\label{sect:conseq}


Continue the notation of the previous section. We now derive consequences from Theorem~\ref{thm:involutions} and Corollary~\ref{cor:involutions} by considering the fixed-point subvarieties of our involutions. Combining this information with Theorem~\ref{thm:fixed-points-AtoD}, we obtain the desired proof of Theorem~\ref{thm:isomorphisms-intro}.

\subsection{Two-row Slodowy slices in types C and D}

The fixed-point subvariety $(\cS^{\A_{2n-1}}_{(2n-k,k)})^\Theta$ is by definition a Slodowy slice in the nilpotent cone of either $\fsp_{2n}$ or $\fso_{2n}$, according to whether the form $(\cdot,\cdot)$ on $\tV_0$ is skew-symmetric or symmetric. Recall from \S\ref{subsect:form} that the cases in which $(\cdot,\cdot)$ is skew-symmetric are the following: 
\begin{equation} \label{eqn:skew-symm}
\begin{cases}
\text{$k<n$ and $k$ is even, }\textbf{or}\\
\text{$k=n$ is even and $(w_+,w_-)=(1,1)$, }\textbf{or}\\
\text{$k=n$ is odd and $(w_+,w_-)=(2,0)$ or $(0,2)$.}
\end{cases}
\end{equation} 
Therefore the cases in which $(\cdot,\cdot)$ is symmetric are the following:
\begin{equation} \label{eqn:symm}
\begin{cases}
\text{$k<n$ and $k$ is odd, }\textbf{or}\\
\text{$k=n$ is odd and $(w_+,w_-)=(1,1)$, }\textbf{or}\\
\text{$k=n$ is even and $(w_+,w_-)=(2,0)$ or $(0,2)$.}
\end{cases}
\end{equation} 
We write $(\cS^{\A_{2n-1}}_{(2n-k,k)})^\Theta$ as $\cS^{\mathrm{C}_n}_{(2n-k,k)}$ if~\eqref{eqn:skew-symm} holds and as $\cS^{\D_n}_{(2n-k,k)}$ if~\eqref{eqn:symm} holds. The two-row partition $(2n-k,k)$ necessarily labels a nilpotent orbit of $\fsp_{2n}$ in the first case or $\fso_{2n}$ in the second case: namely, the orbit of $E_0$. Recall that there is no nilpotent orbit of $\fsp_{2n}$ labelled by $(2n-k,k)$ if $k$ is odd and $k<n$, and no nilpotent orbit of $\fso_{2n}$ labelled by $(2n-k,k)$ if $k$ is even and $k<n$. (In other words, in these cases there are no nilpotent elements of these Lie algebras of Jordan type $(2n-k,k)$.) Recall also that if $n$ is even the partition $(n,n)$ labels two $SO_{2n}$-orbits in the nilpotent cone of $\fso_{2n}$, which form a single $O_{2n}$-orbit; of course, the Slodowy slices associated to two such orbits are isomorphic. 

The partition $(2n-\ell,\ell)$ need not label a nilpotent orbit of the appropriate type; however, to avoid introducing even more cases, we abuse notation slightly and write $(\overline{\cO}^{\A_{2n-1}}_{(2n-\ell,\ell)})^\Theta$ as $\overline{\cO}^{\mathrm{C}_n}_{(2n-\ell,\ell)}$ whenever~\eqref{eqn:skew-symm} holds and as $\overline{\cO}^{\D_n}_{(2n-\ell,\ell)}$ whenever~\eqref{eqn:symm} holds. Thus, for instance, if~\eqref{eqn:skew-symm} holds and $\ell<n$ is odd, the notation $\overline{\cO}^{\mathrm{C}_n}_{(2n-\ell,\ell)}$ means the same as $\overline{\cO}^{\mathrm{C}_n}_{(2n-\ell-1,\ell+1)}$, namely the closure of the nilpotent orbit in $\fsp_{2n}$ labelled by $(2n-\ell-1,\ell+1)$.

As an immediate corollary of Theorem~\ref{thm:involutions} and Proposition~\ref{prop:reductive} we have:

\begin{thm} \label{thm:downstairs}
Continue the assumptions~\eqref{eqn:small},~\eqref{eqn:v-symmetry},~\eqref{eqn:nonempty} on $\bv$ and $\bw$. The isomorphism $\Phi_1$ of Proposition~\ref{prop:maffei1} restricts to an isomorphism
\[
\fM_1^{\A_{2n-1}}(V,W)^\theta\cong
\begin{cases} 
\cS^{\mathrm{C}_n}_{(2n-k,k)}\cap\overline{\cO}^{\mathrm{C}_n}_{(2n-\ell,\ell)}&\text{ if~\eqref{eqn:skew-symm} holds,}\\
\cS^{\D_n}_{(2n-k,k)}\cap\overline{\cO}^{\D_n}_{(2n-\ell,\ell)}&\text{ if~\eqref{eqn:symm} holds,}
\end{cases}
\]
under which $[0]$ corresponds to the base-point $E_0$. The action of $G_W^{\theta\sim}$ on the left-hand side corresponds to the action of $Z_{G_{(\cdot,\cdot)}}(E_0,H_0,F_0)$ on the right-hand side.
\end{thm}

One interesting consequence of this result is:

\begin{cor} \label{cor:column}
Suppose that either $k$ is even or $k=n$. Then there is a variety isomorphism $\cS^{\mathrm{C}_n}_{(2n-k,k)}\cong\cS^{\D_{n+1}}_{(2n-k+1,k+1)}$ under which the base-points correspond.
\end{cor} 

\begin{proof}
Note that $\cS^{\mathrm{C}_n}_{(2n-k,k)}=\cS^{\mathrm{C}_n}_{(2n-k,k)}\cap\overline{\cO}^{\mathrm{C}_n}_{(2n)}$ and $\cS^{\D_{n+1}}_{(2n-k+1,k+1)}=\cS^{\D_{n+1}}_{(2n-k+1,k+1)}\cap\overline{\cO}^{\D_{n+1}}_{(2n+1,1)}$.
Hence as cases of Theorem~\ref{thm:downstairs} we have $\cS^{\mathrm{C}_n}_{(2n-k,k)}\cong\fM_1^{\A_{2n-1}}(\bv,\bw)^\theta$ and $\cS^{\D_{n+1}}_{(2n-k+1,k+1)}\cong\fM_1^{\A_{2n+1}}(\bv',\bw')^\theta$, where the $(2n-1)$-tuples $\bv,\bw$ and the $(2n+1)$-tuples $\bv',\bw'$ are defined by
\[
\begin{split}
\bv&=(1,2,\cdots,k-1,k,k,\cdots,k,k-1,\cdots,2,1),\\
\bw&=\begin{cases}
&(0,\cdots,0,1,0,\cdots,0,1,0,\cdots,0)\\
&\quad\text{ if $k<n$, where the $1$s are in positions $k$ and $2n-k$,}\\
&(0,\cdots,0,2,0,\cdots,0)\\
&\quad\text{ if $k=n$, where the $2$ is in position $n$,}
\end{cases}\\
\bv'&=(0,1,2,\cdots,k-1,k,k,\cdots,k,k-1,\cdots,2,1,0),\\
\bw'&=\begin{cases}
&(0,\cdots,0,1,0,\cdots,0,1,0,\cdots,0)\\
&\quad\text{ if $k<n$, where the $1$s are in positions $k+1$ and $2n-k+1$,}\\
&(0,\cdots,0,2,0,\cdots,0)\\
&\quad\text{ if $k=n$, where the $2$ is in position $n+1$,}
\end{cases}\\
\end{split}
\]
and, if $k=n$, the signature $(w_+,w_-)$ of the involution used to define $\theta$ is $(1,1)$ if $n$ is even and $(2,0)$ if $n$ is odd. But for these dimension vectors we have an obvious isomorphism $\fM_1^{\A_{2n-1}}(\bv,\bw)\cong\fM_1^{\A_{2n+1}}(\bv',\bw')$ respecting the diagram involutions, given by appending zero vector spaces at both ends of the type $\A_{2n-1}$ quiver.
\end{proof}

As mentioned in the introduction, it was previously known only that the singularities of $\cS^{\mathrm{C}_n}_{(2n-k,k)}$ and $\cS^{\D_{n+1}}_{(2n-k+1,k+1)}$ at the base-point were smoothly equivalent (a special case of the Kraft--Procesi column removal rule~\cite[Proposition 13.5]{kp}) and isomorphic as complex analytic germs (part of~\cite[Proposition 12.1]{lns}).

\subsection{Resolutions of two-row Slodowy slices in types C and D}
\label{subsect:the-end}

In deriving consequences of Corollary~\ref{cor:involutions}, we make the simplifying assumption that $s_i=0$ for all $i$, or in other words
\begin{equation}\label{eqn:v-special}
\bv=(1,2,\cdots,k-1,k,k,\cdots,k,k-1,\cdots,2,1),
\end{equation}
which certainly entails~\eqref{eqn:v-symmetry},~\eqref{eqn:nonempty} and~\eqref{eqn:symmetry}.
Then $\cF$ is the variety of complete flags, and $\mu:T^*\cF\to\overline{\cO}^{\A_{2n-1}}_{(2n)}$ is the Springer resolution of the nilpotent cone of $\fsl_{2n}$. If~\eqref{eqn:skew-symm} holds, then $(T^*\cF)^\Theta$ is by definition the Springer resolution of the nilpotent cone of $\fsp_{2n}$, and $\mu^{-1}(\cS^{\A_{2n-1}}_{(2n-k,k)})^\Theta$ equals $\tcS^{\mathrm{C}_n}_{(2n-k,k)}$, the Springer resolution of the Slodowy slice $\cS^{\mathrm{C}_n}_{(2n-k,k)}$. By contrast, if~\eqref{eqn:symm} holds, then $(T^*\cF)^\Theta$ is the disconnected union of two copies of the Springer resolution of the nilpotent cone of $\fso_{2n}$, because the variety of complete isotropic flags is the flag variety of $O_{2n}$ rather than that of $SO_{2n}$ (which consists of isotropic flags that are complete except for lacking a middle-dimensional subspace.) Hence $\mu^{-1}(\cS^{\A_{2n-1}}_{(2n-k,k)})^\Theta$ is the disconnected union of two copies of $\tcS^{\D_n}_{(2n-k,k)}$, the Springer resolution of the Slodowy slice $\cS^{\D_n}_{(2n-k,k)}$. Recall that we write $\cB^{\mathrm{C}_n}_{(2n-k,k)}$ and $\cB^{\D_n}_{(2n-k,k)}$ for the Springer fibres at the base-point of these slices.

From Propositions~\ref{prop:maffei1} and~\ref{prop:reductive} and Corollary~\ref{cor:involutions} we immediately obtain:

\begin{thm} \label{thm:upstairs}
Continue the assumptions~\eqref{eqn:small} and~\eqref{eqn:v-special} on $\bv$ and $\bw$. The isomorphism $\Phi$ of Proposition~\ref{prop:maffei1} restricts to an isomorphism
\[
\fM^{\A_{2n-1}}(V,W)^\theta\cong
\begin{cases} 
\tcS^{\mathrm{C}_n}_{(2n-k,k)}&\text{ if~\eqref{eqn:skew-symm} holds,}\\
\tcS^{\D_n}_{(2n-k,k)}\coprod \tcS^{\D_n}_{(2n-k,k)}&\text{ if~\eqref{eqn:symm} holds.}
\end{cases}
\]
Under this isomorphism, the subvariety $\fL^{\A_{2n-1}}(V,W)^\theta$ corresponds to $\cB^{\mathrm{C}_n}_{(2n-k,k)}$ or $\cB^{\D_n}_{(2n-k,k)}\coprod\cB^{\D_n}_{(2n-k,k)}$ as appropriate, and the action of $G_W^{\theta\sim}$ on the left-hand side corresponds to the action of $Z_{G_{(\cdot,\cdot)}}(E_0,H_0,F_0)$ on the right-hand side. 
\end{thm}

We can then deduce Theorem~\ref{thm:isomorphisms-intro} as follows.

\begin{proof}
Replace $\fM^{\A_{2n-1}}(V,W)^\theta$ in Theorem~\ref{thm:upstairs} with the expression for it as a disconnected union of varieties $\fM^{\D_{n+1}}((v_1,\cdots,v_{n-1},v_+,v_-),(w_1,\cdots,w_{n-1},w_+,w_-))$ given by Theorem~\ref{thm:fixed-points-AtoD}. If $k=n$,~\eqref{eqn:skew-symm} and~\eqref{eqn:symm} dictate whether to take $(w_+,w_-)=(1,1)$ or $(w_+,w_-)\in\{(2,0),(0,2)\}$; in the latter case, for definiteness we set $(w_+,w_-)=(0,2)$. Exchanging the $+$ and $-$ nodes of the Dynkin diagram obviously gives an isomorphism of quiver varieties.

One has to determine which choices of $(v_+,v_-)$ with $v_++v_-=k$ give nonempty varieties $\fM^{\D_{n+1}}((v_1,\cdots,v_{n-1},v_+,v_-),(w_1,\cdots,w_{n-1},w_+,w_-))$, which is a simple calculation using the type D analogue of~\cite[Lemma 7]{maffei}. There is no need to consider all $k+1$ possibilities since we already know that, if~\eqref{eqn:skew-symm} holds, there is a unique nonempty variety $\fM^{\D_{n+1}}(\cdot,\cdot)$ (which is then isomorphic to $\tcS^{\mathrm{C}_n}_{(2n-k,k)}$); and, if~\eqref{eqn:symm} holds, there are two nonempty varieties $\fM^{\D_{n+1}}(\cdot,\cdot)$ (each of which is isomorphic to $\tcS^{\D_n}_{(2n-k,k)}$). In the latter case, we chose one of the two possible quiver varieties arbitrarily for the isomorphism statement of Theorem~\ref{thm:isomorphisms-intro}. The alternative choices are given by exchanging $v_+$ and $v_-$ if $k$ is odd, and as follows if $k$ is even:
\begin{equation} \label{eqn:extra}
\tcS^{\D_n}_{(n,n)}\cong \fM^{\D_{n+1}}((1,2,\cdots,n-1,{\textstyle\frac{n-2}{2}},{\textstyle\frac{n+2}{2}}),(0,\cdots,0,2)).
\end{equation}
Each isomorphism $\tcS^{\mathrm{C}_n/\D_n}_{(2n-k,k)}\cong \fM^{\D_{n+1}}(\cdot,\cdot)$ listed in Theorem~\ref{thm:isomorphisms-intro} (or in~\eqref{eqn:extra}) induces the analogous isomorphism $\cB^{\mathrm{C}_n/\D_n}_{(2n-k,k)}\cong\fL^{\D_{n+1}}(\cdot,\cdot)$, as one sees by combining the corresponding statement of Theorem~\ref{thm:upstairs} with Proposition~\ref{prop:lagrangian}(2).

To obtain the analogous isomorphisms of affine varieties, we can argue as follows. The projective morphism $\pi:\fM^{\D_{n+1}}(\cdot,\cdot)\to\fM_1^{\D_{n+1}}(\cdot,\cdot)$ is, by definition, the ``affinization'' map of $\fM^{\D_{n+1}}(\cdot,\cdot)$, that is, the unique map from $\fM^{\D_{n+1}}(\cdot,\cdot)$ to an affine variety that induces an isomorphism of rings of regular functions. Since each Slodowy slice $\cS^{\mathrm{C}_n/\D_n}_{(2n-k,k)}$ is a normal variety (this follows from the well-known normality of the nilpotent cone), the projective resolution map $\mu:\tcS^{\mathrm{C}_n/\D_n}_{(2n-k,k)}\to\cS^{\mathrm{C}_n/\D_n}_{(2n-k,k)}$ is the affinization map of $\tcS^{\mathrm{C}_n/\D_n}_{(2n-k,k)}$. So each isomorphism $\fM^{\D_{n+1}}(\cdot,\cdot)\simto\tcS^{\mathrm{C}_n/\D_n}_{(2n-k,k)}$ induces a unique isomorphism $\fM_1^{\D_{n+1}}(\cdot,\cdot)\simto\cS^{\mathrm{C}_n/\D_n}_{(2n-k,k)}$ making the following diagram commute:
\begin{equation} \label{eqn:final-comm-square}
\vcenter{\xymatrix{
\fM^{\D_{n+1}}(\cdot,\cdot)\ar[r]^\sim\ar[d]_\pi & \tcS^{\mathrm{C}_n/\D_n}_{(2n-k,k)}\ar[d]^\mu\\
\fM_1^{\D_{n+1}}(\cdot,\cdot)\ar[r]^\sim & \cS^{\mathrm{C}_n/\D_n}_{(2n-k,k)}
}}
\end{equation}
In fact, the bottom isomorphism of~\eqref{eqn:final-comm-square} must equal the composition of the appropriate map $\Psi_{\tbv,1}$ (as defined in Proposition~\ref{prop:referee}) with the isomorphism of Theorem~\ref{thm:downstairs}, since that composition would also make the diagram commute (as one sees by composing the commutative diagrams of Proposition~\ref{prop:referee} and Proposition~\ref{prop:maffei1}). So here we have a case where the morphism $\Psi_1$ defined after Proposition~\ref{prop:referee} is definitely surjective.

All that remains is to describe what group action on the various quiver varieties $\fM^{\D_{n+1}}(\cdot,\cdot)$ corresponds to the action of the stabilizer on $\tcS^{\mathrm{C}_n}_{(2n-k,k)}$ or $\tcS^{\D_n}_{(2n-k,k)}$.

If~\eqref{eqn:skew-symm} holds, the stabilizer acting faithfully on $\tcS^{\mathrm{C}_n}_{(2n-k,k)}$ is $Z_{PSp_{2n}}(E_0,H_0,F_0)$, which is $Z_{G_{(\cdot,\cdot)}}(E_0,H_0,F_0)/\C^\times$. By Theorem~\ref{thm:upstairs}, the corresponding action on $\fM^{\D_{n+1}}(\cdot,\cdot)$ is that of $G_W^{\theta\sim}/\C^\times$. Thus, the identity component of the stabilizer acts as $G_W^\theta/\C^\times\cong(GL_{w_1}\times\cdots\times GL_{w_{n-1}}\times GL_{w_+}\times GL_{w_-})/\C^\times$ (which is either trivial, isomorphic to $\C^\times$, or $PGL_2$ in the three cases). When $k$ is even there is a non-identity component whose elements have order $2$; by Proposition~\ref{prop:conn-compts}(2) these elements act as diagram automorphisms for the exchange of the $+$ and $-$ nodes.

If~\eqref{eqn:symm} holds, the relevant stabilizer is $Z_{PSO_{2n}}(E_0,H_0,F_0)$, which in each case is the identity component of $Z_{G_{(\cdot,\cdot)}}(E_0,H_0,F_0)/\C^\times$, and the corresponding action on $\fM^{\D_{n+1}}(\cdot,\cdot)$ is that of $(GL_{w_1}\times\cdots\times GL_{w_{n-1}}\times GL_{w_+}\times GL_{w_-})/\C^\times$. When $k$ is odd there is a non-identity component of $G_W^{\theta\sim}/\C^\times$ corresponding to the non-identity component of $Z_{PO_{2n}}(E_0,H_0,F_0)$; this acts on $\fM^{\A_{2n-1}}(V,W)^\theta$ by interchanging the two connected components (i.e.\ the two copies of $\tcS^{\D_n}_{(2n-k,k)}$), in accordance with Proposition~\ref{prop:conn-compts}(1).
\end{proof}


\end{document}